\documentclass{amsbook}

\usepackage{amsmath,amssymb,amsthm}
\usepackage{amsmidx}
\usepackage{graphics}
\usepackage{enumerate}
\usepackage{mathtools}

\input xy
\xyoption{all}
\usepackage{notation}
\CompileMatrices

\newcommand{\idx}[1]{\index{idx}{#1}}
\newcommand{\nidx}[1]{\index{nidx}{#1}}
\makeindex{idx}
\makeindex{nidx}

\makeatletter
\ifx\SK@label\undefined\let\SK@label\label\fi
 \let\your@thm\@thm
 \def\@thm#1#2#3{\gdef\currthmtype{#3}\your@thm{#1}{#2}{#3}}
 \def\mylabel#1{{\let\your@currentlabel\@currentlabel\def\@currentlabel
  {\currthmtype~\your@currentlabel}
 \SK@label{#1@}}\label{#1}}
 \def\myref#1{\ref{#1@}}
\makeatother

\newtheorem{theorem}{Theorem}[section]
\newtheorem{lemma}[theorem]{Lemma}
\newtheorem{prop}[theorem]{Proposition}
\newtheorem{corollary}[theorem]{Corollary}

\theoremstyle{definition}
\newtheorem{definition}[theorem]{Definition}
\newtheorem{ex}[theorem]{Example}

\theoremstyle{remark}
\newtheorem{rmk}[theorem]{Remark}
\newtheorem*{ack}{Acknowledgments}

\newcommand{\ev}{\mathrm{ev}}
\newcommand{\sou}{s}
\newcommand{\tar}{t}
\newcommand{\sect}{\sigma}
\newcommand{\pro}{p}
\newcommand{\calP}{\mathcal{P}}
\newcommand{\calC}{\mathcal{C}}
\newcommand{\secttwo}{\varsigma}
\newcommand{\symf}{\phi}
\newcommand{\basepath}{\zeta}
\newcommand{\nmalbdl}{\tau}
\newcommand{\fibm}{r}
\newcommand{\fibs}{N}
\newcommand{\tr}{\mathrm{tr}}
\newcommand{\notcalP}{P}

\newcommand{\sNV}{\scr{N_V}}
\newcommand{\sNU}{\scr{N_U}}
\newcommand{\sEV}{\scr{E_V}}
\newcommand{\sEU}{\scr{E_U}}
\newcommand{\sIV}{\scr{I_V}}

\newcommand{\rc}{\triangleright}
\newcommand{\lc}{\triangleleft}
\newcommand{\ob}{\mathrm{ob}}

\newcommand{\Mod}{{\sf Mod}}
\newcommand{\Ex}{{\sf Ex}}
\newcommand{\Top}{{\sf Top}}
\newcommand{\Ch}{{\sf Ch}}
\newcommand{\Ab}{{\sf Ab}}

\newcommand{\sh}[1]{{\ensuremath{\hspace{1mm}\makebox[-1mm]{$\langle$}\makebox[0mm]{$\langle$}\hspace{1mm}{#1}\makebox[1mm]{$\rangle$}\makebox[0mm]{$\rangle$}}}}

\begin{document}

\frontmatter

\title{Fixed Point Theory and Trace for Bicategories}
\author{Kate Ponto\\ \today}

\maketitle

\chapter*{Abstract}
The Lefschetz fixed point theorem follows easily from the identification
of the Lefschetz number with the fixed point index.  This identification
is a consequence of the functoriality of the trace in symmetric monoidal
categories.  

There are refinements of 
the Lefschetz number and the fixed point index that 
give a converse to the Lefschetz fixed point theorem.
An important part of this theorem is the identification of these 
different invariants.  

We define a generalization of the trace in symmetric monoidal
categories to a trace in bicategories with shadows.  We show the 
invariants used in the converse of the Lefschetz fixed point theorem 
are examples of this trace and that the functoriality of the trace provides
some of the necessary identifications.
The methods used here do not use simplicial techniques and so generalize
readily to other contexts.

\tableofcontents

\renewcommand{\thetheorem}{\Alph{theorem}}
\chapter*{Introduction}

There are many approaches to determining when a continuous endomorphism
 of a topological space has a fixed point.  One of the simplest
is given by the Lefschetz fixed point theorem.

\begin{theorem}[Lefschetz fixed point theorem]\mylabel{lefschetz1}
Let $M$ be a compact ENR and $f\colon M\rightarrow M$ a continuous map.
If $f$ has no fixed points then the Lefschetz number of $f$ is zero.
\end{theorem}

The Lefschetz number of a map is defined using rational homology
and so is relatively easy to compute.  Further, if $M$ is a simply
connected closed smooth manifold of dimension at least three then
a converse to the Lefschetz fixed point theorem also holds.

\begin{theorem}\mylabel{lefschetz2} Let $f\colon M\rightarrow M$ be a
continuous map of a simply connected closed smooth manifold of
dimension at least three. Then the Lefschetz number of $f$ is zero
if and only if $f$ is homotopic to a map with no fixed points.
\end{theorem}

Note that we have replaced `the map $f$ has no fixed points' with
`the map $f$ is homotopic to a map with no fixed points'.  This
change only reflects the fact that the Lefschetz number is defined
using homology and so cannot distinguish between homotopic maps.
In particular, the Lefschetz number cannot determine if a map has
no fixed points, it can only determine if it is homotopic to a map
with no fixed points.

Unfortunately, \myref{lefschetz2} does not hold if we remove
the hypothesis that the space is simply connected.  However, by
sacrificing some of the computability we can refine the Lefschetz
number to an invariant, called the Nielsen number, that detects if 
the map has fixed points.

\begin{theorem}\mylabel{nielsen1} Let $f\colon M\rightarrow M$ be a
continuous map of a closed smooth manifold of dimension at least
three. The Nielsen number of $f$, $N(f)$, is the minimum number of
fixed points of any map homotopic to $f$.  In particular, $N(f)$
is zero if and only if $f$ is homotopic to a map with no fixed
points.
\end{theorem}

The idea behind the Nielsen number is to incorporate information
about the fundamental group into the invariant itself.  This
additional information corresponds to recording which fixed points
can be eliminated by a homotopy of the original map.

The Nielsen number is not the most convenient description of this
information for defining generalizations of this invariant to
other categories and for proving results about relationships
between the Nielsen number and basic topological constructions
such as cofiber sequences or products.  The invariant that retains
the necessary information is called the Reidemeister trace. This
invariant was defined by Wecken and Reidemeister in \cite{Wecken2,
Reidemeister}.   It can be used to prove a theorem similar to
\myref{nielsen1}.

\begin{theorem}\mylabel{reidemeister1} Let $f\colon M\rightarrow M$ be a
continuous map of a closed smooth manifold of dimension at least
three. The Reidemeister trace of $f$ is zero if and only
if $f$ is homotopic to a map with no fixed points.
\end{theorem}

Classically, all four of these results were proved using
simplicial techniques.  In \cite{Dold}, Dold and Puppe proposed
an alternative approach.   Their idea was to focus on the
identification of the Lefschetz number, which is a global
invariant, with a local invariant, the fixed point index. It is 
immediate from the definition  that the fixed point index is zero for
a map that has no fixed points or is homotopic to a map with no
fixed points. Using this observation, the Lefschetz fixed point
theorem is a consequence of the identification of the Lefschetz
number with the index.

Dold and Puppe approached this identification by defining a more
general construction that includes both of these invariants as
special cases. Their construction is a `trace' in any symmetric
monoidal category. In some cases the trace is functorial.  Dold
and Puppe showed that the identification of the Lefschetz number
with the index is an example of this functoriality.

In addition to giving an alternate proof of the Lefschetz fixed
point theorem, Dold and Puppe's definition of trace can be used to
describe generalizations of the fixed point index to other
categories.  If $f\colon X \rightarrow X$ and $\pro \colon X\rightarrow B$ are
continuous maps such that $\pro \circ f=\pro $ we say that $f$ is a
fiberwise map.  In \cite{Dold2}, Dold defined an index for
fiberwise maps and showed that the index is zero for a map that is 
fiberwise homotopic to a map with no fixed points.  The fiberwise
index is an example of the trace in symmetric monoidal categories.

It is possible to prove results for the trace in symmetric monoidal
categories that can be applied to the special cases of the
Lefschetz number and the index. For example, the Lefschetz number
and the index are both additive on cofiber sequences.  This
follows from the additivity of the trace in (some) symmetric monoidal
categories, see \cite{maytraces}.

Unfortunately, the trace in symmetric monoidal categories cannot
be used to describe the invariants of Theorems \ref{nielsen1} and
\ref{reidemeister1}. Invariants that include information about 
the fundamental group do not 
fit into a symmetric monoidal category. However, by replacing
symmetric monoidal categories by an appropriate bicategory and
similarly modifying the definition of the trace we can accommodate
these invariants.

Here we implement this philosophy.  First we show that the
Reidemeister trace is an example of a more general trace.  This
trace is defined here and is a trace in bicategories with some
additional structure; these bicategories are called symmetric
bicategories with shadows. Just as the Lefschetz number can be
identified with the fixed point index, there is more than one
description of the Reidemeister trace. There are generalizations
of the fixed point index, defined by Reidemeister and Wecken, and
of the Lefschetz number, defined by Husseini in \cite{Husseini}.
Both of these invariants are examples of the trace in symmetric
bicategories with shadows, and the functoriality of the trace can
be used to identify them. There is also an invariant defined by
Klein and Williams in \cite{KW} that can be identified with another 
example of the trace in a symmetric
bicategory with shadows.

Next we show that this change in perspective gives definitions and
proofs that generalize more easily than the classical approaches.
One element of the classical invariants that causes problems for
equivariant and fiberwise generalizations is the role played by a
base point. Both classical definitions of the Reidemeister trace
require that a base point be chosen, but a different choice of the
base point does not change the invariant.  Modified forms of the
Reidemeister trace can be defined without a base point.  We show
that these invariants are also examples of trace in bicategories,
and we use the formal structure of the trace to show that these
unbased invariants  can be identified with the classical
invariants.

The second source of problems for generalizations is only obvious
when trying to prove a converse to the Lefschetz fixed point
theorem like \myref{reidemeister1}. In \cite{scofield},
Scofield defined a generalization of the Nielsen number to
fiberwise maps and gave an example that showed this invariant does not give
a converse to the fiberwise Lefschetz fixed point theorem. More
recently, Klein and Williams have defined a fiberwise invariant
that does give a converse to the fiberwise Lefschetz fixed point
theorem.

\begin{theorem}\mylabel{reidemeister2}
Let $M\rightarrow B$ be a fiber bundle with compact manifold fibers
$F$ such that $\dim (F)-3\geq \dim (B)$.  Then a fiberwise map
$f\colon M\rightarrow M$ is fiberwise homotopic to a map with no fixed
points if and only if the fiberwise Reidemeister trace of $f$ is zero.
\end{theorem}

There is another invariant, defined by Crabb and James in
\cite{CrabbJames}, that can help to explain the discrepancy
between Scofield's invariant and Klein and Williams' invariant.
The invariant defined by Crabb and James is a derived form of the
Reidemeister trace and so in the transition from a classical
invariant to a fiberwise invariant it is sensitive to information
that the other forms of the Reidemeister trace, like Scofield's
invariant, miss.  Crabb and James' invariant can be identified
with the invariant defined by Klein and Williams.  Crabb and James'
invariant, in both its classical and fiberwise forms, is an example
of the trace in bicategories with shadows.

More concretely, our goal is to convert Dold and Puppe's outline
for proving \myref{lefschetz1} into an approach for proving
Theorems \ref{reidemeister1} and \ref{reidemeister2}. Dold and
Puppe's proof identified the Lefschetz number and the fixed point
index and then used the observation that the index is zero for
maps with no fixed points.  Our first step is the same. We start
by identifying the form of the Reidemeister trace defined by
Husseini with Reidemeister and Wecken's form of the Reidemeister trace.
Unfortunately, it is not obvious that Reidemeister and Wecken's form of the
Reidemeister trace is zero only when the map is homotopic to a map
with no fixed points. The next step in our proof is to identify
Reidemeister and Wecken's form of the Reidemeister trace with
Crabb and James' version.  This invariant can
then be identified with the invariant defined by Klein and Williams.
Klein and Williams' proof in \cite{KW} then completes the proof of
\myref{reidemeister1}. 

To implement this plan we need to make connections between four different
invariants.  The first three invariants are examples of the 
trace in bicategories 
with shadows.  Functoriality gives an identification of the 
Reidemeister trace defined by Husseini with the Reidemeister trace 
defined by Reidemeister and Wecken.  Functoriality also shows that the 
Reidemeister trace defined  by Reidemeister and Wecken is zero when 
the Reidemeister trace defined by Crabb and James is zero.  The
converse of this fact is not formal.  

In the fiberwise setting of \myref{reidemeister2} not all of
the steps in our proof of \myref{reidemeister1} make sense. 
Here we only have two invariants, the invariant defined
by Klein and Williams and the fiberwise version of the invariant
defined by Crabb and James.  Klein and Williams' proof has an
immediate fiberwise generalization and their fiberwise invariant
can be identified with the fiberwise version of Crabb and James'
invariant in complete analogy with the classical case.

We could interpret these proofs either as category theory with topological 
applications or as topological proofs that have a formal part.  Here
we will try to aim for the middle since, in reality, the category 
theory motivates the topology and the topology motivates the category 
theory.  This balance is reflected in the structure of this paper.  
We start with some motivation from fixed point theory and category theory.
Then we give reinterpretations of the fixed point theory that 
further suggests our definition of trace in a bicategory and the results
that are entirely category theory.  Motivated by these descriptions
we define shadows and traces in bicategories.  Using  these formal
results, we then give new proofs of some classical and fiberwise
fixed point theory
results.  The last chapter returns to category theory and consists
of further examples that are closely related to the topological 
examples given earlier.

In Chapters \ref{reviewfp} and \ref{reviewfp2} 
we recall the elements of topological
fixed point theory that will be the motivation for much of the
later chapters.  In Chapter \ref{reviewfp} we define the Lefschetz number
and the fixed point index.  We also summarize 
Dold and Puppe's results on duality and trace in symmetric
monoidal categories and its applications to fixed point theory.
In Chapter \ref{reviewfp2} we focus on the converse to the Lefschetz
fixed point theorem.  We define the Nielsen number and the two versions of the
Reidemeister trace defined by Husseini and Reidemeister and Wecken.  
We also describe Klein and Williams' proof of the converse to 
the Lefschetz fixed point theorem.

Chapter \ref{summary} serves as a transition between the classical
fixed point theory of Chapter \ref{reviewfp2} and the definition of
trace in a symmetric bicategory with shadows in Chapter
\ref{whybicat}.  Here we give alternate descriptions of the
versions of the Reidemeister trace defined by Wecken and
Reidemeister and Crabb and James that suggest the definitions of Chapter
\ref{whybicat}.  This chapter does not contain rigorous proofs,
which are delayed to Chapters \ref{classfpsec}, \ref{classfpsec2},
\ref{fibfpsec}, and \ref{fibfpsec2},
but instead makes it clear that the Reidemeister trace has many
features in common with trace in symmetric monoidal categories.

In Chapter \ref{whybicat} we define shadows in a bicategory and
trace in a bicategory with shadows.  We also prove some basic
results about the trace and give some algebraic examples.  In
Chapters \ref{classfpsec} and \ref{classfpsec2} 
we describe topological examples of
duality and trace and show that  the Reidemeister trace as defined by 
Reidemeister and Wecken and Crabb and James can be
described using the trace in a bicategory with shadows.  We also show
that functoriality gives identifications of some of the 
forms of the Reidemeister trace.  In Chapters
\ref{fibfpsec} and \ref{fibfpsec2}
we show that many of the results of Chapters \ref{classfpsec} and 
\ref{classfpsec2}
carry over to the category of fiberwise spaces.  

Chapter \ref{reviewbicat} consists of examples of bicategories with shadows 
that either motivate or are motivated by the topological examples 
in Chapters \ref{classfpsec}, \ref{classfpsec2}, \ref{fibfpsec}, and 
\ref{fibfpsec2}.   While the earlier chapters can be read without these
examples, some of the results
and constructions we use in Chapters \ref{classfpsec}, \ref{classfpsec2},
\ref{fibfpsec}, and
\ref{fibfpsec2} have more straightforward analogues in Chapter
\ref{reviewbicat}.  In the earlier chapters we will indicate when
there is a relevant section in Chapter \ref{reviewbicat}.  Chapter
\ref{reviewbicat} can be read after Chapter \ref{whybicat}.

\begin{ack}
I would like to thank my adviser Peter May for all of his help,
interest, and encouragement.  I also thank Mohammed Abouzaid and 
Mike Shulman for many helpful conversations.

I thank Vesta Coufal, Bj{\o}rn Jahren, John Klein, Andrew Nicas, and
Bruce Williams for sharing their work with me; Johann Leida and
Julia Weber for answering questions; Julie Bergner, Tom Fiore,
and Johann Sigurdsson for their comments on previous drafts of
this thesis; and Niles Johnson
for naming the shadows.

I am very grateful to the many people who have listened to me,
encouraged me, and shared their knowledge with me.  

This research was partially supported by Lucent Technologies
through
the Graduate Research Program for Women and Minorities.
\end{ack}

\mainmatter

\renewcommand{\thetheorem}{\thesection.\arabic{theorem}}
\renewcommand{\thesection}{\thechapter.\arabic{section}}

\chapter{A review of fixed point theory}\label{reviewfp}
This chapter and the next are primarily a review of the definitions and results
from classical fixed point theory that motivate the remaining
chapters. This chapter also contains an introduction to Dold and Puppe's
definitions of duality and trace in symmetric monoidal categories.

The two invariants described in this chapter, the Lefschetz
number and the fixed point index, are examples of trace in symmetric
monoidal categories.  Since the fixed point index is zero for maps that 
have no fixed points, the Lefschetz fixed point theorem follows from the 
identification of the Lefschetz number with the index.  This identification 
is a consequence of the functoriality of the trace in symmetric monoidal
categories.

\section{Classical fixed point theory}

The Lefschetz fixed point theorem is a familiar result that relates
a local, geometric invariant to a global, algebraic invariant.  The algebraic
invariant is the Lefschetz number.

\begin{definition}Let $K$ be a field and $C_*$ a finitely generated
chain complex of vector spaces over $K$.  If $f\colon C_*\rightarrow
C_*$ is a map of chain complexes, the \emph{Lefschetz number}
\idx{Lefschetz number} of
$f$, $L(f)$\nidx{Lf@$L(f)$}, is the alternating sum of the levelwise traces.
\end{definition}

\begin{theorem}[Lefschetz Fixed Point Theorem] Let $M$ be a closed 
smooth manifold
and $f\colon M\rightarrow M$ a continuous map.  If the Lefschetz number of
\[f_*\colon H_*(M;\mathbb{Q}) \rightarrow H_*(M;\mathbb{Q})\] is nonzero
then $f$ has a fixed point.\nidx{f@$f_*$}
\end{theorem}

Note that the Lefschetz number of the identity map is the Euler
characteristic.

Since homology is a homotopy invariant, we could replace the
conclusion of this theorem with ``then all maps homotopic to $f$
have a fixed point.''  Additionally, we can use the integers
rather than the rational numbers as our coefficients.  The
Lefschetz number of $H_*(f;\mathbb{Z})$ is defined and is equal to
the Lefschetz number of $H_*(f;\mathbb{Q})$.

To refine the Lefschetz fixed point theorem as
described in the introduction we need another invariant, the fixed
point index. There are many ways to define the index. We will use
Dold's definition using homology and fundamental classes from
\cite{doldbook, Dold1}.

For a generator $[S^n]$\nidx{Sn@$[S^n]$} of $H_n(S^n;\mathbb{Z})
\cong\mathbb{Z}$
and any pair $K\subset V\subset \mathbb{R}^n$, $V$ open and $K$
compact, there is a \emph{fundamental class}\idx{fundamental class} 
\[[S^n]_K\in H_n(V,V-K; \mathbb{Z})\] around $K$.  This class $[S^n]_K$\nidx{Snk@$[S^n]_K$} 
is the
image of $[S^n]$ under the map \[H_n(S^n;\mathbb{Z}) \rightarrow
H_n(S^n,S^n-K;\mathbb{Z})\cong H_n(V,V-K; \mathbb{Z}).\]

\begin{definition}\cite[VII.5]{doldbook} Let $V\subset \mathbb{R}^n$ be open
and $f\colon V\rightarrow \mathbb{R}^n$ be continuous.  
Assume \[F=\{x\in V|f(x)=x\}\] is
compact and let $[S^n]_F$ be the fundamental class of $F$. Then
$I_f\in\mathbb{Z}$\nidx{If@$I_f$}, 
the \emph{fixed point index}\idx{fixed point index}\idx{index!fixed point}
of $f$, is
defined by $I_f[S^n]=(\id-f)_*[S^n]_F$ where
\[(\id-f)\colon (V,V-F)\rightarrow (\mathbb{R}^n,\mathbb{R}^n-0)\] is
defined by $(\id-f)(x)=x-f(x)$.
\end{definition}

The index is additive.  If there are open sets $V_i$ such that
$\bigcup V_i=V$ and $(F\cap V_i)\cap(F\cap V_j)=\varnothing$ for
$i\neq j$, then $\sum I_{f|V_i}=I_f$.  The index is local. If
$F\subseteq W\subseteq V$ for some open set $W$, then
$I_{f|W}=I_f$.  The index is commutative.  If $V\subset
\mathbb{R}^n$, $V'\subset \mathbb{R}^m$ are open sets and
$f\colon V\rightarrow \mathbb{R}^m$, $g\colon V'\rightarrow \mathbb{R}^n$ are
continuous maps then \[\xymatrix{U=f^{-1}(V')\ar[r]^-{gf}&
{\mathbb{R}^n}&{\mathrm{and}}&U'=g^{-1}(V)\ar[r]^-{fg}&{\mathbb{R}^m}
}\] have homeomorphic fixed point sets.  If these sets are compact
$I_{fg}=I_{gf}$.

If $Y$ is any topological space and $U\subset Y$ is an open set which
is also an ENR, then every map $f\colon U\rightarrow Y$ admits a factorization
$f=\beta\alpha$ where \[\xymatrix{U\ar[r]^{\alpha}&V\ar[r]^{\beta}&Y}\]
and $V$ is open in some $\mathbb{R}^n$.  If $F_f=\{y\in U|f(y)=y\}$ is compact
then the fixed point index $I_{\alpha\beta}$ of $\alpha\beta\colon \beta^{-1}U
\rightarrow V\subset \mathbb{R}^n$ is defined and is independent of the
factorization $f=\beta\alpha$.  This number is defined to be the index
of $f$.  In particular, the index of an endomorphism of an ENR is well defined.

\begin{rmk} The index is an invariant of homotopy classes of maps, so
homotopic maps have the same index.  Additivity, localization,
commutativity, homotopy invariance along with an additional axiom,
normalization, characterize the index, see \cite[IV]{brown}.  The
normalization axiom ensures that the index agrees with the
Lefschetz number. Alternatively, in \cite[5.1]{Dold2}, a
characterization of the index is given using a variation of the
homotopy invariance axiom and a normalization axiom.
\end{rmk}

\begin{theorem}[Lefschetz-Hopf]\mylabel{lhopf}If $M$ is a 
closed smooth manifold
and $f\colon M\rightarrow M$ is a continuous map
then the index of $f$, $I_f$, equals the Lefschetz number of
\[f_*\colon H_*(M;\mathbb{Z})\rightarrow H_*(M;\mathbb{Z}).\]\end{theorem}

The Lefschetz Fixed Point Theorem is a consequence of this theorem
since if $f$ has no fixed points the index is zero.  There is a
familiar proof of this theorem that uses simplicial homology, see
\cite[9.6]{Armstrong} or \cite[2.C]{Hatcher}. We will describe an alternative,
conceptual proof using duality in symmetric monoidal categories in
the next two sections.

\section{Duality and trace in symmetric monoidal categories}
\mylabel{doldsec}
This section is a summary of the results of \cite{Dold} that we
will generalize.  Other references for this section include
\cite[III.1]{LMS} and \cite{may}. We will define trace and duality for any
symmetric monoidal category, but our focus will be on examples in
the category of modules over a commutative ring $R$ and the stable
homotopy category.

Let $\sC$\nidx{C@$\protect\sC$} be a symmetric monoidal category with product
$\otimes$,\nidx{$\otimes$} unit object $I$,\nidx{i@$I$} 
and symmetry isomorphism $\gamma$.\nidx{gamma@$\gamma$}\idx{symmetric monoidal
category}

\begin{definition}We say that $A\in\ob \sC$ is \emph{dualizable}\idx{dualizable} 
if there is a $B\in\ob\sC$ and morphisms  $\eta\colon I\rightarrow A\otimes B$, called
\emph{coevaluation}\idx{coevaluation},\nidx{eta@$\eta$} 
and $\epsilon\colon B\otimes A \rightarrow I$,\nidx{epsilon@$\epsilon$}
called \emph{evaluation}\idx{evaluation}, in $\sC$  such that the following
composites are the identity maps
\[\xymatrix@C=40pt{A\cong I\otimes A\ar[r]^-{\eta\otimes \id_A}&A\otimes
B\otimes A\ar[r]^-{\id_A\otimes \epsilon} &A\otimes
I\cong A}\]
\[\xymatrix@C=40pt{B\cong B\otimes I\ar[r]^-{\id_B\otimes \eta}&B\otimes
A\otimes B\ar[r]^-{\epsilon \otimes \id_B}&I\otimes
B\cong B.}\] We call $B$ the \emph{dual}\idx{dual} of $A$ and we say $(A,B)$ is
a \emph{dual pair}\idx{dual pair}.\end{definition}
Note that any two duals of a dualizable object are isomorphic.

Let $R$ be a commutative ring and $\Mod_R$ be the category
of $R$-modules.  Then $\Mod_R$\nidx{Mr@$\protect\Mod_R$} is a symmetric monoidal
category using the usual tensor product over $R$.
The ring $R$ thought of as a module over itself is the unit.
The dual of a finitely generated free $R$-module $M$ is
$\Hom_R(M,R)$.  This is also a finitely generated free
$R$-module.  If $M$ has basis
$\{m_1,m_2,\ldots, m_n\}$ and dual basis $\{m'_1,m'_2,\ldots,
m'_n\}$  the coevaluation and evaluation for the dual pair,
\[\xymatrix@C=18pt{\eta\colon R \ar[r]& M\otimes_R\Hom_R(M,R)&{\mathrm{ and}}
&\epsilon\colon \Hom_R(M,R)\otimes_RM \ar[r]& R,}\]  are $R$-module
homomorphisms given by
$\epsilon(\phi,m)=\phi(m)$ and by extending the map
$\eta(1)=\sum_im_i\otimes m'_i$. If $M$ is a finitely generated
projective module it is also dualizable with dual $\Hom_R(M,R)$.
The evaluation map is $\epsilon(\phi,m)=\phi(m)$. The dual basis
theorem implies that there is a `basis' $\{m_1,m_2,\ldots, m_n\}$
of $M$ and dual `basis' $\{m_1',m_2',\ldots, m_n'\}$ of
$\Hom_R(M,R)$.  The coevaluation map is given by extending
$\eta(1)=\sum_im_i\otimes m'_i$.

Let $\Ch_R$\nidx{chr@$\Ch_R$} be the symmetric monoidal category of chain complexes
of modules over a commutative ring $R$ and chain maps.  The dualizable chain
complexes are
the chain complexes that are projective in each
degree and finitely generated. The dual of a finitely generated
projective chain complex
$M$ is $\Hom_R(M,R)$. 

\begin{theorem}\mylabel{symotherchar} Let $A$ and $B$ be objects in $\sC$ and
$\epsilon\colon B\otimes A\rightarrow I$ be a morphism in $\sC$.
Then the following are equivalent.
\begin{enumerate}[(i)]
\item $B$ is the dual of $A$ with evaluation $\epsilon$.
\item The map
$\epsilon/(-)\colon \sC(C,D\otimes B)\rightarrow
\sC (C\otimes A, D)$ which sends $f\colon C\rightarrow D\otimes B$ to
\[\xymatrix{C\otimes A \ar[r]^-{f\otimes \id}&D\otimes B\otimes A
\ar[r]^-{\id\otimes\epsilon}&D\otimes I\cong D}\] is a bijection for all 
objects
$C, D\in \sC$.\nidx{epsilon@$\epsilon/(-)$}
\end{enumerate}\end{theorem}

There is a similar characterization for a map $\eta\colon I\rightarrow A\otimes
B$.

We say that a symmetric monoidal category $\sC$ is closed\idx{symmetric 
monoidal category!closed}\idx{closed!symmetric 
monoidal category} if there
is a functor \[\Hom\colon \sC^{op}\times \sC\rightarrow \sC\] and natural
isomorphisms \[\sC(A, \Hom(B,C))\cong \sC(A\otimes B,C)\cong
\sC(B,\Hom(A,C)).\]
We have displayed two isomorphisms rather than
just one for later comparison with bicategories.  From these
adjunctions, for all objects $A$ and $B$ in $\sC$ there are
coevaluation and evaluation maps
\[\xymatrix{\eta\colon A\ar[r]&\Hom(B, A\otimes
B)&{\mathrm{ and}} &  \epsilon\colon \Hom(A,B)\otimes A\ar[r]&B.}\] There
is also a natural map \[\nu \colon C\otimes \Hom(A,B)
\rightarrow \Hom(A,C\otimes B)\]\nidx{nu@$\nu$} defined as the 
adjoint of \[\xymatrix{C\otimes \Hom(A,B)\otimes A
\ar[r]^-{\id \otimes \epsilon}
&C\otimes B.}\]

\begin{theorem}\mylabel{monoidclosed}
The following are equivalent for an object $A$ of $\sC$.
\begin{enumerate}[(i)]
\item $A$ is dualizable.
\item The map $\nu \colon A\otimes \Hom(A,I) \rightarrow \Hom(A,A)$ is an isomorphism.
\end{enumerate}\end{theorem}

We will call $\Hom(A,I)$ the \emph{canonical dual}\idx{canonical dual} 
of $A$.  When we used the
dual basis theorem to describe the duals in $\Mod_R$ and $\Ch_R$
we used canonical duals.

\begin{definition}For a dualizable object $A$, the \emph{trace}\idx{trace} of
$f\colon A\rightarrow A$ is the composite
\[\xymatrix{I\ar[r]^-{\eta}&A\otimes B\ar[r]^-{f\otimes \id}&
A\otimes B \ar[r]^-{ \gamma}&B\otimes A\ar[r]^-\epsilon& I.}\]\end{definition}
The trace is independent of the choice of dual for $A$.

If $\sC$ is closed the trace of an endomorphism $f$ is any of the 
three endomorphisms of $I$ in the following commutative diagram.
\[\xymatrix{I\ar[r]& \Hom(A,A)\ar[d]^{f_*}\ar[r]^{\nu^{-1}}&A\otimes \Hom(A,I)
\ar[d]^{f\otimes 1}\ar[r]^\gamma&\Hom(A,I)\otimes A\ar[d]^{1\otimes f}\\
&\Hom(A,A)\ar[r]^{\nu^{-1}}& A\otimes \Hom(A,I)\ar[r]^\gamma
&\Hom(A,I)\otimes A\ar[r]^-{\ev}&I.}\]

Let $R$ be a commutative ring and $M$ a finitely generated
projective $R$-module with `basis' $\{m_1,\ldots, m_n\}$.
The trace of a map of $R$-modules
$f\colon M\rightarrow M$ is a map $R\rightarrow R$ and the image of 1 is
\[\sum_i m_i'f(m_i).\] If $R$ is a field the image of 1 is the usual
trace of $f$ regarded as a matrix.  For a map $f\colon M\rightarrow M$ of
chain complexes, the trace
is also a map $R\rightarrow R$ and the image of 1 is \[\sum_i
(-1)^{\mathrm{deg}(m_i)}m_i'(f(m_i)),\] the Lefschetz number of $f$.
The sign comes from the sign in the
symmetry isomorphism.

Let $\sC$ and $\sC'$ be symmetric monoidal categories.  A \emph{lax monoidal
functor}\idx{lax monoidal functor}
consists of a functor \[F\colon \sC\rightarrow \sC'\] and natural transformations
\[\xymatrix{{\symf}\colon FA\otimes
FB\ar[r]& F(A\otimes B)&{\mathrm{and}}& I'\ar[r]&FI}\] subject
to the standard coherence conditions.

A lax monoidal functor is \emph{symmetric}\idx{lax symmetric monoidal
functor} if the following diagram commutes.
\[\xymatrix{F(A)\otimes F(B)\ar[r]^-{\symf}\ar[d]_{\gamma'}&F(A\otimes B)
\ar[d]^{F(\gamma)}\\
F(B)\otimes F(A)\ar[r]^-{\symf}&F(B\otimes A)}\]

\begin{lemma}\mylabel{DPmonoidal} If $A$ is a dualizable object with dual
$DA$, $F\colon \sC\rightarrow \sC'$ is a lax symmetric monoidal functor, and
${\symf}\colon FA\otimes FDA\rightarrow F(A\otimes DA)$ and $I'
\rightarrow FI$ are isomorphisms, then $FA$ and $FDA$ are a dual
pair with evaluation
\[\xymatrix{FDA\otimes FA\ar[r]^-{{\symf}}&F(DA\otimes A)\ar[r]^-{F(\epsilon)}&
FI\cong I'}\] and coevaluation\[\xymatrix{I'\cong FI\ar[r]^-{F(\eta)}&
F(A\otimes DA)
\ar[r]^{{\symf}^{-1}}&FA\otimes FDA.}\]\end{lemma}

This lemma follows from the definition of a dual pair. As
an immediate consequence of this lemma we have the following
corollary.

\begin{corollary}\mylabel{Ftracesym}Let $F\colon \sC\rightarrow
\sC'$ be a lax symmetric monoidal functor and $A$ be a dualizable object of $\sC$
with dual $DA$ such that ${\symf}\colon FA\otimes FDA\rightarrow F(A\otimes
DA)$ and $I'\rightarrow FI$ are isomorphisms.  Then
\[\mathrm{trace}(Ff)=F(\mathrm{trace}(f))\] for any endomorphism
$f$ of $A$.\end{corollary}

The K\"unneth theorem implies that the homology
functor satisfies the conditions of \myref{Ftracesym}
if $C_*$ is a finitely generated  chain
complex of projective modules, the images of the boundary maps are
projective,  and each  $H_i(C_*)$ is projective.  In particular, the
rational singular or cellular chains on a compact manifold satisfy
these conditions.

\section{Duality and trace for topological spaces}
Duality for topological spaces is an example of duality in
symmetric monoidal categories using the stable homotopy category.  We
will describe an equivalent definition that is more intuitive.  For
the perspective using the stable homotopy category
see \cite{Dold}, \cite[III]{LMS}, or \cite[15]{MS}.

\begin{definition}\cite[III.3.5]{LMS}\mylabel{topdualclassic}
Let $X$ be a compact, based topological space.
A based space $Y$ is \emph{n-dual}\idx{n-dual} 
to $X$ if there are maps
$\eta\colon S^n\rightarrow X\wedge Y$, called \emph{coevaluation},\idx{coevaluation} 
 and
$\epsilon\colon Y\wedge X\rightarrow S^n$, called \emph{evaluation},\idx{evaluation}
such that the following diagrams commute stably up to
homotopy.
\[\xymatrix{S^n \wedge X \ar[d]_{\gamma}\ar[r]^-{\eta\wedge \id}
& (X\wedge Y)\wedge X \ar[d]^\cong\\
 X\wedge S^n & X\wedge (Y\wedge X)\ar[l]^-{\id\wedge \epsilon} }
\qquad
\xymatrix{ Y \wedge S^n \ar[d]_{(\sigma\wedge \id)\gamma}\ar[r]^-{\id\wedge \eta} 
& Y \wedge (X\wedge Y) \ar[d]^\cong \\
 S^n\wedge Y & (Y\wedge X)\wedge Y\ar[l]^-{\epsilon\wedge \id}}\]
Here $\sigma\colon S^n\rightarrow S^n$ is a map of degree $(-1)^n$.\end{definition}

For compact manifolds and compact ENR's we can explicitly describe
dual pairs.

\begin{theorem}\cite[3.1]{Dold}\cite[III.5.1]{LMS}\mylabel{topexdual}
\begin{enumerate}[(i)]
\item Let $K\subset \mathbb{R}^n$
be a compact ENR.  Then $K_+$ and the cone on the inclusion of
$\mathbb{R}^n\setminus K$ into $\mathbb{R}^n$ are n-dual.
\item Let $M$ be a closed smooth manifold embedded in $\mathbb{R}^n$.  Then
$M_+$ and the Thom space of the normal bundle $\nu$\nidx{nu@$\nu$} 
of $M\rightarrow \mathbb{R}^n$ are $n$-dual.
\end{enumerate}
\end{theorem}

As usual, $M_+$\nidx{M@$M_+$} is $M$ with a disjoint base point added.
We will denote the Thom space of the embedding of $M$ in $\mathbb{R}^n$
by $T\nu$\nidx{Tnu@$T\nu$}.

Many of the explicit characterizations of dual pairs for
topological spaces can be stated in this form.  Compact ENR's are
the natural generality, but compact manifolds are the practical
generality. Since the duality maps for manifolds are easier to 
describe, we will state most results in that form but there are
also generalizations to compact ENR's.

The coevaluation map for the dual pair $(M_+,T\nu )$
\[S^n\rightarrow T\nu \rightarrow M_+\wedge T\nu ,\]  is the composite
of the Pontryagin-Thom map for the normal bundle of the embedding
$M\rightarrow
S^n$ and the Thom diagonal. If $\sect\colon M\rightarrow \nu $ is the zero section,
the evaluation map \[T\nu \wedge M_+
\rightarrow M_+\wedge S^n\rightarrow S^n\] is the composite of the
Pontryagin-Thom map associated to a tubular neighborhood of
\[M\stackrel{\triangle}{\rightarrow}M\times M \stackrel{
\sect\times \id}{\rightarrow}\nu \times M\] and the projection $M_+\wedge
S^n\rightarrow S^n$.

Let $\{-,-\}$\nidx{$\{-,-\}$} denote the stable maps of based spaces.
From the characterizations of dual pairs in \myref{symotherchar}
and \myref{topexdual} we have the following corollary.
\begin{corollary}If $M$ is a closed smooth manifold embedded in $\mathbb{R}^n$
and $Z$ and $W$ are based spaces then 
\[\{Z\wedge M_+,W\}\cong \{S^n\wedge Z,W\wedge T\nu\}.\]
\end{corollary}

\begin{definition} Let $X$ be a space with $n$-dual $Y$ and coevaluation
and evaluation maps
$\eta\colon S^n\rightarrow X\wedge Y$ and $\epsilon\colon  Y\wedge X\rightarrow S^n$.  If
$f\colon X\rightarrow X$ is a continuous map, then the \emph{trace}
\idx{trace} of $f$ is
the stable homotopy class of the map
\[\xymatrix{S^n\ar[r]^-\eta&X\wedge Y\ar[r]^-{f\wedge\id}&X\wedge Y
\ar[r]^-{\gamma}&Y\wedge X\ar[r]^-{\epsilon}&S^n.}\]\end{definition}

By examining the explicit duality maps for a compact
manifold $M$ (or a compact ENR) we can see that for $f\colon M\rightarrow M$,
\[\mathrm{index}(f)=\tilde{H}_n(\mathrm{trace}(f_+);\mathbb{Q}).\]
This is described in detail in \cite{Dold, Dold1, Dold2}.

Applying rational homology to the coevaluation and evaluation maps of
the dual pair $(M_+,T\nu )$ we get a pair of maps
\[\eta\colon \mathbb{Q}\rightarrow\sum_{i}
\tilde{H}_i(M_+;\mathbb{Q})\otimes \tilde{H}_{n-i}(T\nu
;\mathbb{Q})\] and  \[\epsilon\colon \sum_{i} \tilde{H}_{n-i}(T\nu
;\mathbb{Q})\otimes \tilde{H}_i(M_+;\mathbb{Q}) \rightarrow
\mathbb{Q}.\]  Since these maps come from the dual pair $(M_+,T\nu
)$ they are coevaluation and evaluation maps that make
$(\tilde{H}_i(M_+;\mathbb{Q}), \tilde{H}_{n-i}(T\nu;\mathbb{Q}
) )$ a dual pair. The trace
 of a map does not depend on the choice of dual pair, so the
trace of $\tilde{H}_*(f_+;\mathbb{Q})$ with respect to this dual pair is
the Lefschetz number.  \myref{lhopf} follows from this observation:
\[\mathrm{index}(f)=\mathrm{Lefschetz\,\, number\,\, of \,}f.\]

\section{Duality and trace for fiberwise topological spaces}

We can also define trace and duality in the symmetric monoidal
category of ex-spaces \idx{ex-space}
over a fixed space $B$.  The objects in this
category are spaces $E$ with maps
$B\stackrel{\sect }{\rightarrow}E\stackrel{\pro }{\rightarrow} B$ such that
$\pro \circ \sect $\nidx{s@$\sect$}\nidx{p@$\pro$} 
is the identity map of $B$.  The map $\sect $ is called the
\emph{section}\idx{section} and $\pro $ is called the 
\emph{projection}\idx{projection}.  The morphisms are maps
$E\rightarrow E'$ that commute with the section and projection.
The product is the internal smash product, $\wedge_B$\nidx{$\bar{\wedge}$}.  
Given two
ex-spaces $X$ and $Y$ over $B$, we can form the pullback along the
maps to $B$, $X\times_BY$\nidx{$\times_B$}, and the pushout along the maps from
$B$, $X\vee_BY$\nidx{$\vee_B$}.  The \emph{internal smash product}\idx{internal
smash product} is the pushout
\[\xymatrix{X\vee_BY\ar[r]\ar[d]&X\times_BY\ar[d]\\
B\ar[r]&X\wedge_BY.}\]

\begin{rmk}\mylabel{exspaceconditions}
For all ex-spaces we require that the  
base space and total space are of the homotopy type of CW complexes,
the projection map is a Hurewicz fibration and the section is a
fiberwise cofibration.  The category of these spaces and the fiberwise homotopy 
classes of maps is equivalent to the model theoretic homotopy category
of ex-spaces defined in \cite{MS}.  See \cite[9.1.2]{MS}.  

If we need to consider an ex-space that does not satisfy these conditions
we will replace it with an equivalent space that does satisfy our requirements.
We will not indicate the replacement in the notation.
\end{rmk}

\begin{definition} Let $X$ be an ex-space over $B$.  An
ex-space $Y$ is \emph{$n$-dual}\idx{fiberwise n-dual}\idx{n-dual!fiberwise} 
to $X$ if there are fiberwise
stable maps
\[\eta\colon S^n\times B\rightarrow X\wedge_BY\hspace{1cm}
\mathrm{and}\hspace{1cm}
\epsilon\colon Y\wedge_BX \rightarrow S^n\times B\] 
such that 
\[\xymatrix{S^n \wedge X \ar[d]_{\gamma}\ar[r]^-{\eta\wedge \id}
& (X\wedge_B Y)\wedge_B X \ar[d]^\cong\\
 X\wedge S^n & X\wedge_B (Y\wedge_B X)\ar[l]^-{\id\wedge \epsilon} }
\qquad
\xymatrix{ Y \wedge S^n \ar[d]_{(\sigma\wedge \id)\gamma}\ar[r]^-{\id\wedge \eta} & Y \wedge_B (X\wedge_B Y) \ar[d]^\cong \\
 S^n\wedge_B Y & (Y\wedge_B X)\wedge_B Y\ar[l]^-{\epsilon\wedge \id}}\]
 commute up to stable fiberwise homotopy.
\end{definition}

This definition is very similar to the definition of $n$-duality in the
category of based topological spaces.  Using
parametrized spectra, this definition can be expressed as duality
in a symmetric monoidal category, see \cite{Dold, MS}.
We also have explicit characterizations of dual pairs similar to those
in \myref{topexdual}.

\begin{theorem}\cite[II.12.18]{CrabbJames}\cite[6.1]{Dold}\cite[15.1.1]{MS}\mylabel{fibdualclas}
\begin{enumerate}
\item Let $L$ be an $ENR_B$ over a paracompact space $B$
such that $\pro \colon L\rightarrow B$ is proper.
Then $L_+=L\amalg B$\nidx{L@$L_+$} 
and the cone of $(B\times \mathbb{R}^n)\setminus L
\rightarrow B\times \mathbb{R}^n$ are an $n$-dual pair.
\item  Let $N$ be an ex-space over $B$ that satisfies 
the conditions of \myref{exspaceconditions}.  Then $N$ is dualizable as an 
ex-space if and only if
$\pro ^{-1}(b)$ for $b\in B$ is dualizable in the sense of 
\myref{topdualclassic}.   
\end{enumerate}
\end{theorem}
In particular, if $E\rightarrow B$ is a fiber bundle with compact 
smooth manifold fibers the ex-space $E_+$ is dualizable.

The definition of trace for a fiberwise map is identical to the
definition of trace for a map of based spaces.

\begin{definition} Let $X$ be an $n$-dualizable ex-space over $B$ with dual $Y$ and
$f\colon X\rightarrow X$ be a fiberwise map over $B$.  The
\emph{trace}\idx{trace} 
of $f$ is the fiberwise stable homotopy class of the composite
\[\xymatrix{B\times S^n\ar[r]^\eta&X\wedge_BY\ar[r]^{f\wedge
\id}&X\wedge_BY \ar[r]^\gamma&Y\wedge_BX\ar[r]^{\epsilon}&B\times
S^n.}\]\end{definition}

When $X$ is a compact fiberwise ENR over a compactly generated
paracompact base space, this is the fiberwise Dold index of $f$ as
defined in \cite{Dold2}.

\begin{rmk} There is a category $\sF$\nidx{F@$\protect\sF$} 
with objects ex-spaces over $B$ as
before but whose morphisms are pairs of maps $f\colon E\rightarrow E$ and $
\bar{f}\colon B\rightarrow B$ such that
\[\xymatrix{B\ar[r]^{\bar{f}}\ar[d]^\sect &B\ar[d]^\sect \\
E\ar[r]^f\ar[d]^\pro &E\ar[d]^\pro \\
B\ar[r]_{\bar{f}}&B}\]
commutes.  The category of ex-spaces is the subcategory of $\sF$ with
the same objects whose morphisms are the pairs
$(f,\bar{f})$ where $\bar{f}$ is the identity.

Fixed point theory in $\sF$ and in the category of ex-spaces are
very different.  We will not discuss fixed point theory in  $\sF$;
some references for it include \cite{HKW, Lee}.
\end{rmk}

\chapter{The converse to the Lefschetz fixed point theorem}\label{reviewfp2}
The invariants described in the previous chapter only give a converse to the 
Lefschetz fixed point theorem under additional hypotheses.  The crucial
assumption is that the spaces are simply connected.  We can relax this 
assumption by 
changing the invariant.  The Nielsen number and various forms of the 
Reidemeister trace are choices for this refined invariant.

The Reidemeister trace, in any of its forms,  
is not an example of trace in symmetric
monoidal categories.  The reasons for this incompatibility will
become clear as we define these invariants and compare different
features of these invariants in this chapter and in the following
chapters.

Despite the many differences between the Reidemeister trace and
trace in symmetric monoidal categories it will also become clear
that the Reidemeister trace can be described using a trace very
much like the trace in a symmetric monoidal category. The structure
suggested in this chapter and the next chapter is explicitly
described in Chapters \ref{whybicat}, \ref{classfpsec}, 
and \ref{classfpsec2}.

\section{The Nielsen number}

To define the index we needed to assume that the fixed point set
\[F= \{x\in M|f(x)=x\}\] is compact; now we will also assume that it
is discrete. All invariants discussed here are invariants of
homotopy classes of maps so choosing a homotopic representative
doesn't change the invariant. If $M$ is a manifold, transversality
implies that any endomorphism of $M$ is homotopic to a map with a
discrete fixed point set.

\begin{definition} Two fixed points of $f\colon M\rightarrow M$, $x$ and $y$,
are in the same \emph{fixed point class}\idx{fixed point class} 
if there is a lift of $f$
to \[\tilde{f}\colon  \tilde{M}\rightarrow \tilde{M}\] on the universal
cover of $M$ and lifts of $x$ and $y$ to $\tilde{x}$ and
$\tilde{y}$ in $\tilde{M}$ such that $\tilde{f}(\tilde{x})=
\tilde{x}$ and $\tilde{f}(\tilde{y})=\tilde{y}$.\end{definition}

There is an equivalent definition of fixed point classes using
paths in $M$ rather than the universal cover.

\begin{lemma}Two fixed points $x$ and $y$  of $f\colon M\rightarrow M$ are in
the same fixed point class if and only if there is a path $\alpha$
from $x$ to $y$ such that $\alpha$ is homotopic to $f(\alpha)$
with endpoints fixed. \end{lemma}

We can give another interpretation of this lemma in terms of the
fundamental groupoid, $\Pi M$\nidx{pim@$\Pi M$}, of $M$.  Let $\mathrm{Fix}$
\nidx{Fix@$\protect\mathrm{Fix}$} be the
groupoid defined by the equalizer \[\xymatrix{{\mathrm{Fix}}
\ar[r]&\Pi M\ar@<.5ex>[r]^-{\id}\ar@<-.5ex>[r]_-f &\Pi M.}\] The
objects of $\mathrm{Fix}$ are the fixed points of $f$ and the
morphisms are the homotopy classes of paths with $[\alpha]=[f\circ
\alpha]$. If we think of the fixed points as a discrete category,
this category includes into the category $\mathrm{Fix}$. The lemma
shows that two fixed points are in the same fixed point class if
and only if their images are in the same connected component of
$\mathrm{Fix}$.

Since the fixed point set is assumed to be compact and discrete,
it must be finite.  Let $F_1,F_2,\ldots, F_k$\nidx{Fi@$F_i$} be the fixed point
classes of $f$.  This is also a finite set, $\cup F_i=F$, and
$F_i\cap F_j=\emptyset$ if $i\neq j$.

Let $f\colon M\rightarrow M$ be a continuous map with a compact and
discrete, and hence finite, 
fixed point set.  For a fixed point class $F_i$, let
$V_i$ be an open set in $M$ such that $F_i\subset V_i$ and
$V_i\cap F_j$ is empty if $i\neq j$.  Define the index of $F_i$\idx{index!of
a fixed point class},
$i(F_i)$\nidx{ifi@$i(F_i)$}, to be the index of $f|V_i$.  Since the index is local
this is well defined.

\begin{definition} The \emph{Nielsen number}\idx{Nielsen number} 
of $f$, $N(f)$\nidx{N@$N$},
is the number of fixed point classes with nonzero index.
\end{definition}

\begin{theorem}[\myref{nielsen1}]\mylabel{stdconverse} The Nielsen number
of a continuous endomorphism of a closed smooth manifold of dimension at
least three is zero if and only if the map is homotopic to a map
with no fixed points.\end{theorem}

The standard proof of this result uses simplicial techniques, see
\cite[VIII]{brown}. That proof shows that the theorem holds for
simplicial complexes where the star of each vertex is connected.
We will give a more conceptual proof in Chapter \ref{classfpsec2}.

The dimension hypothesis is necessary. 
In \cite{jiang2}, Jiang constructed
an endomorphism for any two dimensional connected manifold with negative
Euler characteristic that is not homotopic to a fixed point free
map but whose Nielsen number is zero.

\section{The geometric Reidemeister trace}\label{geotracesec}

Using the fixed point classes and the index of a continuous map
$f\colon M\rightarrow M$ we can also define the geometric Reidemeister
trace.  This invariant contains all of the information in the
Nielsen number so it can also be used to prove a converse to the
Lefschetz Fixed Point Theorem.

Choose a base point $\ast$ in $M$, a path $\basepath$\nidx{zeta@$\protect\basepath$} 
from $\ast$ to
$f(\ast)$, and a path $\gamma_x$\nidx{gammax@$\gamma_x$} 
from $\ast $ to $x$ for each fixed
point $x$ of $f$.  The map that takes a  fixed point $x$ to the
homotopy class of $ \gamma_x^{-1}f(\gamma_x)\basepath$ defines a
function from the fixed points of $f$ to the fundamental group of
$M$.

\begin{definition}\mylabel{semiconj} Let $\pi$ be a group and  $\psi\colon \pi\rightarrow \pi$
a homomorphism.  The set  $\sh{\pi^\psi}$ \nidx{shad@$\protect\sh{-}$} 
of \emph{semiconjugacy
classes}\idx{semiconjugacy class} 
of $\pi$ is the set $\pi$ modulo the relation $\alpha\sim
\beta\alpha \psi(\beta^{-1})$ for $\alpha, \beta\in \pi$.
\end{definition}

Using the path $\basepath$ we can define a homomorphism $\phi\colon \pi_1M
\rightarrow \pi_1M$ by $\phi(\alpha)=\basepath^{-1} f(\alpha)\basepath$. The
map from the fixed points of $f$ to $\pi_1M$ descends to a
well-defined injection from the fixed point classes of $f$ to
$\sh{\pi_1M^\phi}$ that is independent of all choices of paths.

The paths $\basepath$ and $\gamma_x$ also define a map
\[\mathrm{Fix}\rightarrow \sh{\pi_1M^\phi}\] by $x\mapsto
\gamma_x^{-1} f(\gamma_x)\basepath$, and the diagram
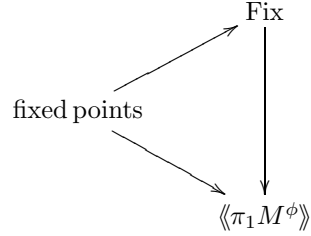
\[\xymatrix{&{\mathrm{Fix}} \ar[dd]\\
{\mathrm{fixed\, points}}\ar[ur]\ar[dr]\\
&\sh{\pi_1M^{\phi}}}\] commutes.

We let $\mathbb{Z}\sh{\pi_1M^\phi}$\nidx{Zshadpim@$\protect\bZ\protect\sh{\pi_1M^\phi}$} 
denote the free
abelian group on the set $\sh{\pi_1M^\phi}$.
The injection above gives an identification of a
fixed point class $F_k$ with its image in
$\sh{\pi_1M^\phi}$.

\begin{definition}\mylabel{geortracedef}
The \emph{geometric Reidemeister trace}\idx{geometric Reidemeister trace}
\idx{Reidemeister trace!geometric} of $f$,
$R^{geo}(f)$,\nidx{Rgeo@$R^{geo}$} is \[\sum_{\mathrm{Fixed\, Point\, Classes}\,
F_k}i(F_k)\cdot F_k\in\mathbb{Z}\sh{\pi_1(M)^\phi}.\] 
\end{definition}

The index of $f$ is the sum of the coefficients in the
Reidemeister trace.  The Nielsen number is the number of elements
in $\sh{\pi_1(M)^\phi}$ with nonzero coefficients in the
Reidemeister trace.

Note that the geometric Reidemeister trace is zero if and only if the
Nielsen number is zero.

\begin{theorem}[\myref{reidemeister1}]\mylabel{stdconverse2}
If $M$ is a closed smooth manifold of dimension at least three, then the
geometric Reidemeister trace of an endomorphism $f$ of $M$
 is zero if and only if $f$ is homotopic to a map with no fixed points.
\end{theorem}

\section{The algebraic Reidemeister trace}

We generalized the index to the geometric Reidemeister trace using
the fixed point classes of an endomorphism.  We can also use fixed
point classes to generalize the Lefschetz number to the algebraic
Reidemeister trace.

The algebraic Reidemeister trace is based on a generalization of the
trace for linear transformations to a trace for homomorphisms 
of finitely generated
projective modules.  We will define this generalized trace, called the
Hattori-Stallings trace, first and then use it to define the algebraic
Reidemeister trace.

Let $R$ be a ring. A \emph{trace function}\idx{trace function} 
$T$ is a function from square
matrices over $R$ to an abelian group such that
\renewcommand{\labelenumi}{$(\roman{enumi})$}
\begin{enumerate}
\item If $A, B\in \sM_{p\times p}(R)$ then $T(A+B)=T(A)+T(B)$.
\item If $A \in \sM_{p\times q}(R)$ and
$B\in \sM_{q\times p}(R)$ then $T(AB)=T(BA)$.
\end{enumerate}
Then $T(A)=\sum_iT(a_{ii})$ for $A=(a_{ij})$.

From a ring $R$ we define an abelian group $\sh{R}$ \nidx{shad@$\protect\sh{-}$}
as the
quotient of $R$ by the subgroup generated by elements of the form
\[r_1r_2-r_2r_1\] for $r_1,r_2\in R$.

\begin{prop}\cite{Stallings}\mylabel{staltrace1}
The universal trace function\idx{trace function!universal}\idx{universal
trace function}, a trace function through
which every trace function can be factored, is given on $1\times 1$-matrices
by the quotient map\nidx{T@${\protect\sT}$}
\[\sT\colon R\rightarrow \sh{R}.\]
This extends to $n\times n$ matrices by  $\sT(A)=\sum_i\sT(a_{ii})$ for
$A=(a_{ij})$.
\end{prop}

Let $M$ be a finitely generated free right $R$-module. Given any
trace function $T$ we can define a map $T \colon  \mathrm{End} (M)
\rightarrow  \sh{R}$. For each endomorphism $\phi$ choose a matrix
$A$ representing $\phi$ and define \[T(\phi) \coloneqq T(A).\]  By the
second property of a trace function this is well defined.

We can also use a trace function $T$ to define a map
\[T\colon \mathrm{End}(M)\rightarrow \sh{R}\] for a finitely generated
projective right $R$-module $M$.  Let $N$ be a module such that
$M\oplus N$ is a finitely generated free $R$-module. 
If $\phi$ is an endomorphism of $M$ define an
endomorphism of the free module $M\oplus N$ by $\phi\oplus 0$.
Then $T(\phi)$ is defined to be $T(\phi\oplus 0)$. This is
independent of all choices.

This description is given in terms of a `basis' since it is
defined using matrices.  There is an equivalent definition of the
universal trace function that does not require explicit use of the
basis.  For any right $R$-modules $P$ and $M$ there is a map
\[\xymatrix{\nu \colon P\otimes_R \Hom_R(M,R)
\ar[r]& \Hom_R(M,P)}\] 
defined by  $\nu(p\otimes \phi)(m)=p\phi(m)$.  If $M$ is a finitely
generated projective right $R$-module this map is an isomorphism.

\begin{prop}\mylabel{staltrace} If $M$ is a finitely generated projective
right $R$-module, the universal trace function is the composite
map \[\xymatrix@R=1pt{\Hom_R(M,M)\ar[r]^-{\nu
^{-1}}&M\otimes_R\Hom_R(M,R) \ar[r]^-\delta&{\sh{R}}
}\]
where $\delta(m\otimes \phi)=\sT(\phi(m))$.\end{prop}

The algebraic Reidemeister trace requires a generalization of the
universal trace function. 
Before we can generalize the trace we need a more general
target for a trace function.

\begin{definition}\mylabel{ringshadow}
Let $P$ be an $R$-$R$-bimodule.  Then $\sh{P}$\nidx{shad@$\protect\sh{-}$} 
is the
quotient of $P$, as an abelian group, by the subgroup generated by
elements of the form $pr-rp$ for $r\in R$ and $p\in P$. 
\end{definition}
Let $\sT\colon P \rightarrow \sh{P}$ be the quotient map.

The generalization of the universal trace function is easier to
describe when not explicitly using a basis, so we will generalize
the description given in \myref{staltrace}.
Let $R$ be a ring, $P$ be an $R$-$R$-bimodule, and $M$ be a finitely
generated projective right $R$-module.   

\begin{definition}
The \emph{Hattori-Stallings trace}\idx{Hattori-Stallings trace}, 
$\tr$, of a map $f\colon M\rightarrow
M\otimes_RP$ is the image of $f$ in under the composite
\[\xymatrix{\Hom_R(M,M\otimes_R P)\ar[r]^-{\nu ^{-1}}&
(M\otimes_RP)\otimes_R\Hom_R(M,R)
\ar[r]^-\delta
&\sh{P}.
}\] where $\delta(m\otimes p\otimes \phi)=\sT(p\phi(m))$.\end{definition}

If $C_*$ is a finitely generated chain complex of projective right
$R$-modules, $P$ is an $R$-$R$-bimodule, and $f\colon C_*\rightarrow
C_*\otimes_RP$ is a map of chain complexes, the
\emph{Hattori-Stallings trace} of $f$ is
\[\sum (-1)^i \tr (f_i),\]
the alternating sum of the levelwise traces.  The sign enters
since the evaluation map requires a transposition.

The example of the Hattori-Stallings trace we are most interested
in is the algebraic Reidemeister trace.  To define this invariant
we must first fix some conventions.  Composition of paths in a
space $X$ is given by $(\beta, \alpha)\mapsto \beta\alpha$ where
$\alpha$ is a path from $a$ to $b$, $\beta$ is a path from $b$ to
$c$ and $\beta \alpha$ is a path from $a$ to $c$. This induces the
group multiplication in $\pi_1X$. If we think of $\tilde{X}$ as
homotopy classes of paths in $X$ that start at the base point
$\ast$  there is  an action of $\pi_1X$ on $\tilde{X}$ from the
right by $\tilde{X}\times\pi_1X\rightarrow \tilde{X}$,
$(\gamma,\alpha) \mapsto \gamma \alpha$.

Let $X$ be a finite connected CW complex. Pick a base point $\ast$
in $X$. Then the cellular chain complex of $\tilde{X}$ 
is a finitely generated free right
$\mathbb{Z}\pi_1(X,\ast)$-module.  A continuous map
$f\colon X\rightarrow X$ is not required to preserve a base
point, and so we define an induced map $\tilde{f}$ on the
universal cover by $\tilde{f}(\alpha)= f(\alpha)\basepath$ for some
choice of path $\basepath$ from $\ast$ to $f(\ast)$. Define a group
homomorphism \[\phi\colon \pi_1(X,\ast)\rightarrow\pi_1(X,\ast)\]  by
$\phi(\alpha)=\basepath^{-1}f(\alpha)\basepath$.  The map $\tilde{f}$ is
$\phi$-equivariant in the sense that
\[\tilde{f}(\gamma\alpha)=\tilde{f}(\gamma)\phi(\alpha)\] for $\alpha
\in\pi_1(X,\ast)$ and $\gamma \in\tilde{X}$.

Let $\mathbb{Z}\pi_1(X,\ast)^\phi$ be the
$\mathbb{Z}\pi_1(X,\ast)-\mathbb{Z} \pi_1(X,\ast)$-bimodule that
is $\mathbb{Z}\pi_1(X,\ast)$ as an abelian group with the usual
left action of $\pi_1(X,\ast)$ and the right action given by first
applying $\phi$ and then using the group multiplication. Then
$\tilde{f}$ defines a map \[\tilde{f}_*\colon C_*\tilde{X}\rightarrow
C_*\tilde{X}\otimes_{\mathbb{Z}\pi_1(X,\ast)}\mathbb{Z}\pi_1(X,\ast)^\phi\]
and this is a map of right $\mathbb{Z}\pi_1(X,\ast)$-modules.

\begin{definition} The \emph{algebraic Reidemeister trace}\idx{algebraic 
Reidemeister trace}\idx{Reidemeister trace!algebraic} of $f$,
$R^{alg}(f)$,\nidx{Ralg@$R^{alg}$} is the Hattori-Stallings trace of
$\tilde{f}_*$.\nidx{f@$\tilde{f}_*$}\end{definition}

\begin{theorem}\mylabel{coufalLH2}There is an isomorphism of abelian 
groups \[\mathbb{Z}\sh{
\pi_1(X,\ast)^\phi}\rightarrow \sh{\mathbb{Z}\pi_1(X,\ast)^\phi}\]
and under this isomorphism
\[R^{geo}(f)=R^{alg}(f).\]\end{theorem}

A proof of this theorem can be found in \cite[3.4]{geosurv} or 
\cite[1.13]{Husseini}.
In Chapter \ref{classfpsec2} we will give a more conceptual
proof of this result.

\section{A proof of the converse to the 
Lefschetz fixed point theorem}\label{usuconverse}

In \cite{wecken3}, Wecken showed that for some finite polyhedra the Nielsen
number of an endomorphism is zero if and only if the map is homotopic to a map
with no fixed points.  Shi \cite{shi} later proved a
refinement of Wecken's result.  These proofs used simplicial
techniques. Similar techniques can be used to prove an equivariant analogue
of this result,
but they are not as useful when trying to prove fiberwise results.

Here we will present the main ideas of an alternative proof due to
Klein and Williams from  \cite{KW}.  This proof gives the converse
to the Lefschetz fixed point theorem for manifolds of dimension at
least three and has fiberwise and equivariant generalizations.  The details of
the fiberwise version are in Section \ref{fibconvers}. The missing
details in this section can be recovered from that proof.

For their proof Klein and Williams translate fiberwise homotopy theory into
equivariant homotopy theory using a loop group construction. They
observe that their proof works equally well without
this transformation and that it would be necessary to eliminate
this transition to prove a converse to the fiberwise Lefschetz
fixed point theorem.  Here we present the main ideas of  Klein and  Williams'
proof of the converse to the  Lefschetz fixed point theorem using fiberwise
homotopy theory.

\begin{prop}\cite{FadellHPara}\cite{KW}\mylabel{fptosec}
\begin{enumerate}
\item
Let $X$ be a topological space and $f\colon X\rightarrow X$ a continuous map.
Homotopies
of $f$ to a fixed point free map correspond to
liftings which make the following diagram commute up to homotopy.  
\[\xymatrix{&X\times
X-\triangle\ar[d]\\X\ar@{.>}[ur]\ar[r]_{\Gamma_f}& X\times X.}\]
Here $\Gamma_f$\nidx{gammaf@$\Gamma_f$} is the graph of $f$.

\item For a continuous map $h\colon X\rightarrow Z$ let $\fibm(h)\colon \fibs(h)
\rightarrow Z$\nidx{r@$\fibm$}\nidx{N@$\fibs$}
be a Hurewicz fibration such that \[\xymatrix{X\ar[rr]\ar[dr]_h&&{\fibs(h)}
\ar[dl]^{\fibm(h)}\\
&Z}\] commutes and the map $X\rightarrow \fibs(h)$ is a homotopy
equivalence.   There is a bijective correspondence between
liftings up to homotopy in
the diagram  \[\xymatrix{&X\ar[d]^h\\
Y\ar[r]_g\ar@{.>}[ur]&Z}\] and  sections of the fibration
$g^*\fibs(h)\rightarrow Y$.
\end{enumerate}
\end{prop}

This proposition converts a fixed point question into a question about
sections of fibrations.  We will define a fixed point theory invariant
by defining an invariant that detects
sections of Hurewicz fibrations.

Let $\pro\colon E\rightarrow B$ be a Hurewicz fibration over a connected
space $B$.  The \emph{unreduced fiberwise suspension}\idx{unreduced
fiberwise suspension} of $E$ over
$B$ is the double mapping cylinder\nidx{Sb@$S_B$}
\[S_BE\coloneqq B\times \{0\}\cup_{\pro} E\times [0,1]\cup_{\pro} B\times \{1\}.\]
The map $\pro\colon E\rightarrow B$ induces a map $q\colon S_BE\rightarrow B$
which is also a fibration.\footnote{If $\lambda\colon \fibs\pro\rightarrow E^I$
is a lifting function for $\pro$ with adjoint $\bar{\lambda}$ define
\[\bar{\chi}\colon  \fibs q\times I\rightarrow S_BE\] by
$\bar{\chi}((e,t),\beta,s)=(\bar{\lambda}( e,\beta,s),t)$ for
$(e,t)\in E\times (0,1)$ and $\bar{\chi}(b,\beta,s) =b\in B\times
\{0\}$ for $b\in B\times \{0\}$ and similarly for $b\in B\times \{1\}$. Then
the adjoint of 
$\bar{\chi}$,\[\chi\colon \fibs q\rightarrow (S_BE)^I,\] is a lifting function for $q$.
See \cite{hall, strom} for similar results.} 
Let \[\sect_-,\sect_+\colon B\rightarrow S_BE\] be the
sections of $S_BE\rightarrow B$ given by the inclusions of
$B\times \{0\}$ and $B\times \{1\}$ into $S_BE$.

\begin{prop}\cite[3.1]{KW}\mylabel{firstform}
If $\pro\colon  E\rightarrow B$ admits a section then $\sect_-$
and $\sect_+$ are homotopic over $B$.

Conversely, assume $\pro\colon E\rightarrow B$ is $(r+1)$-connected and $B$
is homotopically a retract of a cell complex with cells in
dimensions $\leq 2r+1$.  If $\sect_-$ and $\sect_+$ are homotopic over
$B$, then $\pro$ has a section.
\end{prop}

From this point we will work in the category of ex-spaces, rather than spaces
over $B$.  This means that all spaces over $B$ have a section and all
maps respect the section. In particular, $S_BE$ is an
ex-space with section $\sect_-$.

Let $S^0_B$\nidx{SB@$S^0_B$} be the ex-space over $B$ with total space two 
disjoint copies of 
$B$.  The inclusion of $B$ into one of the copies of $B$ is the section.  
The projection map is the identity on each component.
Under the assumptions in \myref{firstform},
a fiberwise version of the Freudenthal suspension theorem
\cite[3.19]{CrabbJames}
gives the following isomorphism
\[[S_B^0,S_BE]_B\cong \{S_B^0,S_BE\}_B.\]
The $\{-,-\}_B$\nidx{$\{-,-\}_B$} notation indicates fiberwise (sectioned)
stable homotopy classes
of fiberwise maps.

\begin{definition}\cite[3.4]{KW}
The \emph{stable cohomotopy Euler class}\idx{stable cohomotopy 
Euler class} of $\pro$ is the element 
of \[\{S_B^0,S_BE\}_B\] that corresponds to the map $\sect_-\amalg \sect_+$. 
\end{definition}

For a continuous map $f\colon M\rightarrow M$ there is  an associated fibration 
\[\Gamma_f^*(\fibm(i))\colon \Gamma_f^*(\fibs(i))\rightarrow M\] given by 
pulling the fibration associated to the inclusion
\[i\colon M\times M-\triangle\rightarrow M\times M\] back along the 
graph of $f$.
For this particular case we can give another description  
of the stable cohomotopy Euler class.  Let $\Lambda^fM=
\{\gamma\in M^I|f(\gamma(1))=\gamma(0)\}$.

\begin{prop}\cite[4.1, 5.1]{KW}\mylabel{kwclassidentify} There is an isomorphism
\[\{S_M^0, S_M(\Gamma_f^*(\fibs(i)))\}_M\cong \{S^0,\Lambda^fM_+\}.\]
\end{prop}
 
We will denote the image of $\sect_-\amalg \sect_+$ under this 
isomorphism by $R^{KW}(f)$\nidx{RKW@$R^{KW}$}.

If $M$ is of dimension $n$ then $\Gamma_f^*\fibs(i)\rightarrow M$
is $(n-1)$-connected, see \cite[6.1,6.2]{KW}.  If $n$ is at least three then
the corollary below follows from \myref{firstform}.
\begin{corollary}\cite[10.1]{KW}\mylabel{converse1}
If $M$ is a closed smooth manifold of dimension at least 3, $f$ is 
homotopic to a map with no fixed points if and only if $R^{KW}(f)$ is zero.
\end{corollary}

\chapter{Topological duality and fixed point theory}\label{summary}

In the previous chapters we recalled the definitions of some
classical fixed point theory invariants.  The definition we gave of
the algebraic Reidemeister trace is very similar to the trace in
the symmetric monoidal category of modules over a commutative
ring.  From the definition we gave in the last chapter it is less
clear that the geometric Reidemeister trace resembles the trace in
a symmetric monoidal category.

In this chapter we will give another description of the geometric Reidemeister
trace that will make the similarity with the trace in a symmetric monoidal
category more clear.  We will give definitions of duality and trace that
resemble the definitions for a symmetric monoidal category and are also
similar to the definition of the Hattori-Stallings trace for modules over a
ring.  These constructions will be shown to be examples of
duality and trace in a bicategory with shadows in Chapters \ref{whybicat} and 
\ref{classfpsec}.

We will also describe an invariant, originally defined by Crabb and James, 
that can be identified with Klein and Williams' invariant.  The
definition of this invariant is similar to the geometric Reidemeister trace
but it has two very significant differences.  This invariant does not require
a base point.  We also replace homotopy classes of paths 
with path spaces.  

In this chapter we are mostly interested in the impact of these changes on
the invariants that we have already described, but the changes are even more 
important for fiberwise invariants.
For a fiberwise space choosing a base point would
correspond to choosing a section and sections do not always exist. 
So fiberwise invariants need to be 
unbased.  The change to 
path spaces reflects the greater variety of invariants for fiberwise
spaces.  In the classical case, the invariant defined by Crabb and James 
is only zero when the geometric Reidemeister trace is zero.  These 
invariants have fiberwise generalizations that do not share this property.

Since the techniques of Chapter \ref{whybicat} 
significantly simplify some of the proofs we will delay most of the proofs 
until Chapters \ref{classfpsec} and \ref{classfpsec2}.

\section{Duality for spaces with group actions}\label{summary1}
In this section and the next section we will give an alternate definition of the
geometric Reidemeister trace that is closer to the trace in a
symmetric monoidal category. We first define duality for spaces
with an action by a group $\pi$.  The motivating example of
a space with a group action is the action of the fundamental group
of a manifold on the universal cover by deck transformations.  In the
next section we
use this definition of dual pairs to define a trace.  This
trace is similar to the Hattori-Stallings trace.

There are two important observations about the duality defined in
this section.  First, this duality is more similar in perspective
to the duality defined by Ranicki in \cite[3]{Ranicki} than to
duality in the symmetric monoidal category of $G$-spaces for some
group $G$. Second, despite the action of the group $\pi$, this
duality is used to study classical invariants, not equivariant
ones.

Let $\pi$ be a discrete group.  For a based right $\pi$-space $X$
and a based left $\pi$-space $Y$, let $X\odot Y$\nidx{$\odot $} denote the
based bar complex $B(X,\pi,Y)$\nidx{B@$B(-,\pi,-)$}.

The bar complex $B(X,\pi,Y)$ is the geometric realization of
the simplicial based space 
with $n$ simplices \[X\wedge (\pi^n)_+\wedge Y,\] face maps
\[\partial_0(x,g_1,g_2,\ldots, g_n,y)=(xg_1,g_2,\ldots, g_n,y)\]
\[\partial_i(x,g_1,g_2,\ldots, g_n,y)=(x,g_1,\ldots,g_ig_{i+1},\ldots, g_n,y)\,\,
\mathrm{for}\,\,0< i<n\]
\[\partial_n(x,g_1,g_2,\ldots, g_n,y)=(x,g_1,\ldots, g_{n-1}, g_ny)\]
and degeneracy maps
\[s_i(x,g_1,g_2,\ldots, g_n, y)=(x, g_1,\ldots ,g_i,e,g_{i+1}, \ldots
g_n,y).\]

If $Z$ has left and right actions by $\pi$ then
$B(X, \pi, B(Z,\pi,Y))$ is isomorphic to $B(B(X,\pi,Z),\pi, Y)$.
Also, $B(X,\pi,\pi_+)$ and $X$ are equivalent, but not isomorphic, 
as right $\pi$-spaces.

We think of $B(X, \pi,Y)$ as the homotopy coequalizer
of the maps \[\xymatrix{X\wedge \pi_+\wedge Y
\ar@<.5ex>[r]\ar@<-.5ex>[r]&X\wedge Y
}\] where
the maps $X\wedge \pi_+\wedge Y\rightarrow X\wedge Y$ are the
action of $\pi$ on $X$ and $\pi$ on $Y$.  We will denote the actual 
coequalizer of these maps $X\wedge_\pi Y$\nidx{$\wedge_\pi $}.

We use a homotopy coequalizer to define $\odot$ so that the
result has the correct homotopy type. Alternatively, we could make
assumptions on the actions of $\pi$ so that the bar resolution is 
equivalent to $X\wedge_\pi Y$.  This will be the case in the examples 
we consider.  

\begin{lemma}\cite[8.5]{classfib} If $\pi$ acts principally on $X$ and 
effectively on $Y$
then $B(X,\pi,Y)$ is weakly equivalent to $X\wedge_\pi Y$.
\end{lemma}

Let $K$ be a based space (without an action by $\pi$).  Then
$\vee_{\pi}K$, the wedge product of copies of $K$ indexed by the
elements of $\pi$, has left and right actions of $\pi$ given by
permutations of the factors.

\begin{definition}\mylabel{pindual} We say a based right  $\pi$-space $X$ has
\emph{n-dual}\idx{n-dual} $Y$ if $Y$ is a based left $\pi$-space, there
is a map $\eta\colon S^n\rightarrow X\odot Y$ and a $\pi-\pi$
equivariant map $\epsilon\colon Y\wedge X \rightarrow \vee_{\pi} S^n$
such that the diagrams below commute up to equivariant homotopy
after smashing with $S^m$ for some $m\in\mathbb{N}$.
\[\xymatrix{S^n\wedge X\ar[r]^-{\eta\wedge \id}\ar[dd]_{\gamma}&(X\odot Y)
\wedge X\ar[d]^{\cong} &
Y\wedge S^n\ar[r]^-{\id\wedge \eta}\ar[dd]_{(\sigma\wedge \id)
\gamma}&Y\wedge(X\odot Y)
\ar[d]^{\cong}\\
&X\odot (Y\wedge X)\ar[d]^-{\id \odot
\epsilon}&
&((Y\wedge X)\odot  Y)\ar[d]^-{ \epsilon \odot  \id}\\
X\wedge S^n\ar[r]_-{\cong}& X\odot 
(\vee_{\pi}S^n)
&S^n\wedge Y\ar[r]_-{\cong}& (\vee_{\pi}S^n)\odot 
Y}\]
\end{definition} 

The map $\sigma\colon S^n\rightarrow S^n$ is a map of degree
$(-1)^n$.  If $Y$ is the dual of $X$ we say that $(X,Y)$ is a dual pair.

Not many $\pi$-spaces are dualizable in this sense.  For example, 
Let $\sigma$ be a nontrivial subgroup of $\pi$.  If the $\pi$ space $(\pi/\sigma
)_+$ 
was dualizable with dual $Y$ there would be a $\pi$-$\pi$-equivariant map 
\[\epsilon:Y\wedge (\pi/\sigma)_+\rightarrow \vee_{\pi}S^n\] for some integer 
$n$ such that the diagrams in \myref{pindual} are satisfied.
Since $\epsilon$ is a $\pi$-$\pi$-equivariant map its image must be 
the basepoint of $\vee_{\pi}S^n$.  

Let $M$ be a smooth compact manifold with universal cover
$\tilde{M}$, quotient map $\pi\colon \tilde{M}\rightarrow M$\nidx{pi@$\pi$}, and normal
bundle $\nu$.

\begin{lemma}\mylabel{univcovdual1}
For a closed smooth manifold $M$, $\tilde{M}_+$ is dualizable as a
right $\pi_1M$ space with dual
$T\pi^*\nu$.\end{lemma}

Let $S^\nu$\nidx{snu@$S^\nu$} be the fiberwise one point compactification 
of the normal bundle of $M$ and $\tilde{M}\rightarrow \pi^* S^\nu$ be the 
inclusion as the points at infinity.
Then $T\pi^*\nu$\nidx{Tpinu@$T\pi^*\nu$} 
is the pushout of the maps $\tilde{M}\rightarrow 
\pi^* S^\nu$ and $\tilde{M}\rightarrow*$.   We use
$T$ since we want to suggest the Thom space in analogy with the
duality described in Chapter \ref{reviewfp}. The space
$T\pi^*\nu$ is a left $\pi_1M$ space with action given by $
\alpha \cdot (\gamma,v)= (\gamma\alpha^{-1},v)$ for $\alpha\in
\pi_1M$ and  $(\gamma,v)\in T\pi^*\nu$.
Since the right action of $\pi_1M$ on $\tilde{M}$ is free,
$\tilde{M}_+\odot T\pi^*\nu$ is equivalent to $\tilde{M}\wedge_{\pi_1M}
T\pi^*\nu$.

The coevaluation map for $\tilde{M}$ is the composite
\[\xymatrix{S^n\ar[r]^{\chi}&T\nu\ar[r]&\tilde{M}_+\odot T\pi^*\nu.}\]
The first map is the Pontryagin-Thom map for an embedding of $M$
in $\mathbb{R}^n$.  
The second map is defined by \[v\mapsto
(\gamma_{\rho(v)}, \gamma_{\rho(v)},v)\] where $\rho\colon \nu\rightarrow M$
\nidx{rho@$\rho$} is
the projection and $\gamma_{\rho(v)}$ is any lift of $\rho(v)$ to
$\tilde{M}$.  The second map is independent of the chosen lift
since quotienting by the action of $\pi_1M$ will identify any two
different choices.

Before we define the evaluation map we need the following preliminary 
lemma.
\begin{lemma}\mylabel{localcont}\cite[II.5.2]{CrabbJames}  
Let $K$ be an ENR.  Then there is an open neighborhood $W$ of the
diagonal in $K\times K$ and a homotopy  $H\colon W\times I \rightarrow
K$\nidx{H@$H$} 
such that $H_0(x,y)=x$, $H_1(x,y)=y$ and $H_t(x,x)=x$ for all $(x,y)
\in W$ and $t\in I$. 
\end{lemma}
We fix such a homotopy $H$ and use it when defining all similar evaluation
maps.

The evaluation map for $\tilde{M}_+$,
\[\xymatrix{T\pi^*\nu\wedge\tilde{M}_+ \ar[r]&(\pi_1M)_+ \wedge
S^n=\vee_{\pi_1M}S^n,}\] is defined by
\[(\gamma_m,v,\gamma_n)\mapsto ({\gamma_m^{-1} H(n,m)\gamma_n},
\epsilon(v,n))\] where $\epsilon$ is the evaluation map from the
dual pair $(M_+,T\nu)$. The element \[{\gamma_m^{-1}
H(n,m)\gamma_n}\] of $\pi_1M$ is  the unique $g\in\pi_1M$ such
that $\gamma_m\cdot g$ is contained in a small neighborhood of $\gamma_{n}$.

If $\sect\colon M\rightarrow T\nu$ is the zero section,
points outside of a small neighborhood $N$ of
the image of $(\sect\times \id)\triangle$ are mapped by $\epsilon$ to
the base point of $S^n$.  If necessary we can shrink $N$ so that
for all $(v,m)\in N$, $(\rho(v),m)\in W$.  If $H(n,m)$ is not defined
then the triple $(\gamma_n,\gamma_m,v)$ is mapped to the base
point. The evaluation map is continuous and independent of choices. It is
also compatible with the actions of $\pi_1M$ on $T\pi^*\nu$ and
$\tilde{M}_+$.  We show that these maps make the required diagrams commute
in the proof of \myref{bunivdual}.

\section{The geometric Reidemeister trace as a trace}\label{summary1a}

Before we can define the trace, we need to introduce a little more
structure. This additional structure plays the role of the
symmetry isomorphism in the definition of trace in a symmetric
monoidal category and allows us to compare the target of the
coevaluation with the source of the evaluation for dual pairs of
$\pi$-spaces.

\begin{definition}Let $Z$ be a based $\pi$-$\pi$ space.  
The \emph{shadow}\idx{shadow} of $Z$,
$\sh{Z}$, \nidx{shad@$\protect\sh{-}$}
is the cyclic bar resolution $C(Z,\pi)$.\nidx{C@$C(-,\pi)$}
\end{definition}

The cyclic bar resolution $C(Z,\pi)$ is the geometric realization
of the simplicial based space with $n$ simplices
\[(\pi^n)_+ \wedge Z \] face maps
\[\partial_0(g_1,g_2,\ldots, g_n, z)=(g_1,\ldots,g_{n-1}, g_nz)\]
\[\partial_i(g_1,g_2,\ldots, g_n, z)=(g_1,\ldots,g_{n-i}g_{n-i+1},\ldots, g_n,z
)\,\,
\mathrm{for}\,\,0< i<n\]
\[\partial_n(g_1,g_2,\ldots, g_n,z)=(g_2,\ldots, g_{n}, zg_1)\]
and degeneracy maps
\[s_i(g_1,g_2,\ldots, g_n,z)=(g_1,\ldots ,g_i,e,g_{i+1}, \ldots
g_n,z).\] We think of $\sh{Z}$ as the homotopy coequalizer of the two
actions of $\pi$ on $Z$.

For a homomorphism $\phi\colon \pi \rightarrow \pi$ let 
$\pi^\phi$\nidx{piphi@$\pi^\phi$}  
be $\pi$ as a set.  On the left  $\pi$ acts on $\pi^\phi$
by multiplication and on the right $\pi$ acts by applying $\phi$
and then acting by multiplication. There is a simplicial map from
$C(\pi^\phi,\pi)$ to the constant simplicial set on the set of
semiconjugacy classes of $\pi$ with respect to $\phi$ given by
\[(g_1,g_2,\ldots, g_n, h)\mapsto (g_1g_2\ldots g_nh).\] This  map
is a simplicial homotopy equivalence with inverse given by
\[h\mapsto (e,\ldots, e,\phi(h)).\]  The homotopy between the
identity map and the map
\[(g_1,g_2,\ldots, g_n, h)\mapsto (e,\ldots,
e,\phi(g_1g_2\ldots g_nh))\] is
\[h_0(g_1,g_2,\ldots,g_n,h)=(g_1,g_2,\ldots, g_n,h,e)\]
\[h_i(g_1,g_2,\ldots,g_n,h)=(g_1,\ldots,g_{n-i},g_{n-i+1}
\ldots g_nh,e,\ldots,e )\,\,\mathrm{for}\,\,0<i<n\]
\[h_n(g_1,g_2,\ldots,g_n,h)=(g_1\ldots g_nh,e,\ldots,e ).\]
See \cite[I.5.1]{simplicial} for the definition of a simplicial homotopy.
In particular, $\sh{\pi^\phi}$ is equivalent to the set of semiconjugacy
classes.  This was the definition of $\sh{\pi^\phi}$ in \myref{semiconj}.

Taking the shadow of a $\pi$-$\pi$ space is similar to applying the
functor $\sh{-}$ to an $R$-$R$-bimodule for a ring $R$. In both
cases these solutions seem like an ad hoc resolution to a very
small problem.  In the next chapter we will show that structures
very similar to these are very
important in the definition of trace in a bicategory. 

In this section we will define trace for $\pi$-maps
\[f\colon X\rightarrow \,X^\phi \] where $\phi\colon \pi \rightarrow \pi$ is a
homomorphism, $X$ is a right $\pi$-space and  $X^{\phi}$\nidx{Xphi@$X^\phi$} is the
space $X$ with right action of $\pi$ given by $x\cdot g=x\phi(g)$.
There is a simplicial map from $B(X,\pi,\pi^\phi_+)$ to the
constant simplicial set $X^\phi$ given by the action of $\pi$ on
$X$. This map is a simplicial homotopy equivalence with inverse
given by \[x\mapsto (x,e,\ldots e) \] and so $X\odot 
\pi^\phi_+$ is equivalent to $X^\phi$ as a right $\pi$-space.

\begin{definition} Let $X$ be a right $\pi$-space with $n$-dual $Y$.
Let $\phi\colon \pi\rightarrow \pi$ be a homomorphism.  The \emph{trace}\idx{trace}
of an equivariant map $f\colon  X\rightarrow X^{\phi}$ is the stable
homotopy class of the map \[\xymatrix{S^n \ar[r]^-{\eta}
&X\odot Y\ar[r]^-{f\odot  \id}&
X^\phi \odot Y\cong(X\odot \pi^\phi_+)\odot Y
\cong\sh{Y\wedge X\odot (\pi^\phi)_+}
\ar[r]^-{\sh{\epsilon\odot \id}} &\vee_{\sh{\pi^\phi}}S^n.}\]

\end{definition}

Let $f\colon M\rightarrow M$ be a continuous map of a closed smooth  manifold
and $\phi\colon \pi_1(M)\rightarrow
\pi_1(M)$ be the induced map given by a choice of path $\basepath$ from the base point
to its image under $f$.
Define \[\tilde{f}\colon \tilde{M}\rightarrow \tilde{M}\] by
$\tilde{f}(\gamma) =f(\gamma)\basepath$.\nidx{f@$\tilde{f}$}  
Then $\tilde{f}$ satisfies
\[\tilde{f}(\gamma\alpha)= f(\gamma)f(\alpha)\basepath
=f(\gamma)\basepath\basepath^{-1}f(\alpha)\basepath=\tilde{f}(\gamma)
\phi(\alpha)\] and defines an equivariant map
\[ \tilde{f}\colon \tilde{M}_+ \rightarrow
 \tilde{M}_+^\phi.\]

Recall that the geometric Reidemeister trace of a map $f$ is 
\[\sum_{\mathrm{Fixed\, Point\, Classes}\,
F_k}i(F_k)\cdot F_k\in\mathbb{Z}\sh{\pi_1(M)^\phi}\]
where $i(F_k)$ is the index of the fixed point class $F_k$.

Let $\pi_i^s(X)$ be the $i^{th}$ stable homotopy group of $X$.

\begin{prop}\mylabel{pihiso} There is an isomorphism $\pi_0^s(X)\rightarrow H_0(X)$.
The image of the trace of $\tilde{f}$ under this isomorphism is 
the geometric Reidemeister trace of $f$.
\end{prop}

\begin{proof} The isomorphism $\pi_0^s(X)\rightarrow H_0(X)$ is the 
composite 
\[\xymatrix{ \pi_0^s(X)&\pi_q(\Sigma^q X)\ar[l]\ar[r]&H_q(\Sigma^q X)&
H_0(X)\ar[l]}.\]
The first map is the inclusion.  For sufficiently large $q$ the Freudenthal 
suspension theorem implies this map is an isomorphism.  The second
map is the Hurewicz homomorphism.  It is an isomorphism 
since $\pi_i(\Sigma^qX)$ is trivial for all $i$ less than $q$.  The last 
map is the suspension isomorphism.

By \myref{univcovdual1},  $(\tilde{M}_+,T \pi^*\nu)$ is a dual
pair. The trace of $\tilde{f}$ with respect to this dual pair is
the composite
\[\xymatrix@C=25pt{S^n\ar[r]^-\eta& T\nu\ar[r]& \tilde{M}_+ \odot 
T\pi^*{\nu} \ar[r]^-{\tilde{f} \odot  \id}&\tilde{M}_+
\odot \pi_1M^\phi
\odot T\pi^*{\nu}\ar[r]^-\epsilon&
\vee_{\sh{\pi_1M^\phi}}S^n}\] The trace of $\tilde{f}$ is a map into a
wedge product, so specifying a point in $S^n$ and an element of $\sh{\pi_1M^\phi}$
identifies the image of $v\in S^n$ under
the trace of $\tilde{f}$.  The image of $v\in S^n$ under the trace $\tilde{f}$ in 
$S^n$ is  $\epsilon (\eta(v),f\rho(\eta(v)))$.  The image of 
 $v\in S^n$ under the trace $\tilde{f}$ in $\sh{\pi_1M^\phi}$ is
\[\gamma_{\rho\eta(v)}^{-1} H(f\rho\eta(v),\rho\eta(v))
f(\gamma_{\rho\eta(v)})\basepath.\] 
The maps $\eta$ and $\epsilon$ are the
coevaluation and evaluation for the dual pair $(M_+,T\nu)$ and $H$ is 
as in  \myref{localcont}. At a
fixed point the group element
\[\gamma_{\rho\eta(v)}^{-1} H(f\rho\eta(v),\rho\eta(v))
f(\gamma_{\rho\eta(v)})\basepath
\] is
$\gamma_{\rho\eta(v)}^{-1}f(\gamma_{\rho\eta(v)})\basepath$, the image of
the fixed point under the injection from the fixed point classes
to the semiconjugacy classes described in Section \ref{geotracesec}.
Since the index is local, the trace of $\tilde{f}$ restricted to a
neighborhood of a fixed point is a map of degree equal to the
index of the fixed point.

The trace of $\tilde{f}$ is an element of $\pi_0^s(\sh{\pi_1M^\phi}_+)$
and it is the image of an element in $\pi_n(\Sigma^n \sh{\pi_1M^\phi}_+)$
if $M$ is $n$-dualizable.  The image of this representative under the 
Hurewicz homomorphism is $(\tr(f))_*([S])$.  The projection
\[H_n(\Sigma^n \sh{\pi_1M^\phi}_+)\cong \oplus H_n(S^n)
\rightarrow H_n(S^n)\] to the component corresponding to 
$\alpha\in \sh{\pi_1(M)^\phi}$ takes $(\tr(f))_*([S])$ to the 
index of the fixed point class corresponding to $\alpha$.
\end{proof}

The trace of $\tilde{f}$ and the other invariants we will define here 
are more naturally described as elements of stable homotopy 
groups rather than homology classes.  
We will think of them as homotopy classes of maps, and use the isomorphism
$\pi_0^s(X)\cong H_0(X)$ when it is necessary to make a connection
with homology.  We will use this isomorphism to refer to the trace of $\tilde{f}$
as the geometric Reidemeister trace.

\begin{rmk} By looking at the explicit description of the trace of $\tilde{f}$
for a map $f\colon M\rightarrow M$ we see that $R^{geo}(f)$ does not depend
on the choice of the lift of $f$. Also note that $R^{geo}(f)$ is
an invariant of the homotopy class of $f$ since the trace of $\tilde{f}$
is an invariant of the homotopy class of $\tilde{f}$.
\end{rmk}

In the next chapter we will define duality and trace in a bicategory. In
Chapter \ref{classfpsec} we will show that the duality and trace defined
in this section are examples of duality and trace in a bicategory.

\section{Duality for spaces with path monoid actions}\label{summary3}

In the previous section we gave a description of the geometric Reidemeister
trace.  In this section we will describe an invariant that is similar
to the geometric Reidemeister trace, and coincides with 
the invariant defined by Klein and Williams.

In Section \ref{summary1} the motivating example was the action of
the fundamental group of a topological space on the universal
cover by deck transformations.  In this section the motivating
example is the `action' of the free path space of a topological
space on itself by composition.  From this action we can define 
modules over the path space and then define duality for  modules
and trace for homomorphisms.

We will use this trace to
define the homotopy Reidemeister trace which is a `derived' form
of the geometric Reidemeister trace. The homotopy Reidemeister trace 
is zero only when the
geometric Reidemeister trace is zero.

In the previous section we worked with homotopy classes of maps.  
In this section we will work with free Moore paths.  
The free Moore paths of a space have
composition, but composition is only defined for paths with
compatible endpoints.  This path space also satisfies unit
conditions, but it has many units rather than just one.  We use 
Moore paths rather than regular paths since the composition of Moore
paths is strictly associative.

In the previous section all spaces had base points 
and we had to make adjustments since
the maps were not  based.  In this section 
we will not use base points at all.  This has several advantages.  
The base point is part of the construction of the geometric Reidemeister
trace but the 
invariant does not depend on the choice of base point.  As we mentioned before,
fiberwise spaces don't always have sections, so in that case we will 
need unbased descriptions.

The  free Moore
path space\idx{free Moore paths} \nidx{pm@$\protect\calP M$} 
of a space $M$ is
\[ \calP M=\left\{(\gamma, u)\in \mathrm{Map}
([0,\infty), M)\times [0,\infty) \left|
\gamma(t)=\gamma(u)\,
\mathrm{for\,all}\,t\geq u
\right.\right\}\] 
This space is given the
subspace topology from $\mathrm{Map}
([0,\infty),M)\times [0,\infty)$.   
There are two maps $\sou,\tar\colon \calP M\rightarrow M$, given by
$\sou(\gamma,u)=\gamma(0)$ and $\tar(\gamma,u)=\gamma(u)$.  Later
we will think of $\calP M$ as a category and so $\sou$ and $\tar$ denote
source and target.\nidx{s@$\sou$}\nidx{t@$\tar$}
Two paths
$ (\beta,v),(\alpha, u)\in PM$ can be composed if $\alpha(u)=\beta(0)$
and this composition defines a unital and associative 
product \[\calP M\times_M\calP M=\{((\beta,v),(\alpha,u))
| \beta(0)=\alpha(u)\}\rightarrow
\calP M.\]
If we restrict to paths that start and end at some chosen point $\ast\in M$ the
composition induces the group multiplication in the fundamental
group.

Let $X$ be an ex-space with section and projection
maps $M\stackrel{\sect}{\rightarrow}
X\stackrel{\pro}{\rightarrow} M$ that satisfies the conditions of 
\myref{exspaceconditions}.
We define an ex-space over $M$, $\calP M\boxtimes X$\nidx{$\boxtimes$}, by
imposing additional identifications on the fiber product 
\[\calP M\times_M X=\{((\gamma,u),x)\in \calP 
M\times X|\gamma(0)=\pro(x)\}.\]
We identify $((\gamma,u),x)$ and $((\gamma',u'),x')$ if $x$ and 
$x'$ are both in the image of the section and $\gamma(u)=\gamma'(u')$.
The projection map \[\calP M\boxtimes X\rightarrow X\] is $((\gamma,u),x)
\mapsto \gamma(u)$.  The section \[M\rightarrow \calP M\boxtimes X\]
is $m\mapsto ((c_m,0),\sect(m))$ 
where $c_m$ is the constant path at $m$ in $M$. 
A different description of this product 
is given in Chapter \ref{classfpsec}.  

Similarly,  for two ex-spaces $X$ and $Y$ over $M$ we define a based 
space $X\boxtimes Y$ by
\[\{(x,y)|\pro(x)=\pro'(y)\}/\sim\] where $(x,y)$ is identified with $(x',y')$
if one of element of each pair is in the image of the section.
The base point is 
the point of $X\boxtimes Y$ given by the equivalence relation.

\begin{definition}\mylabel{sketchydef}An ex-space $X$ over $M$
is a \emph{right $ \calP M$-module}
if there is a map over and under $M$
\[\kappa\colon X\boxtimes \calP M\rightarrow X\] that is associative and unital
with respect to the product of $\calP M$.
\end{definition}

The definitions of a left module and a bimodule are similar.  The space
$\calP M$ is a space over $M\times M$ via the map $\tar\times \sou$.  
With a disjoint section added $\calP M$ is a ex-space over $M\times M$.
This ex-space is written $(\calP M, \tar\times \sou)_+$.  Using 
composition of paths $(\calP M,\tar\times \sou)_+$ is a 
$(\calP M,\tar\times \sou)_+$-bimodule.  
Using only one of the maps $\tar$ or $\sou$ we 
can think of $\calP M$ as either a left or right $\calP M$-module.
For example, 
with a disjoint section $(\calP M,\sou)_+$ is a right $\calP M$-module.

If $X$ and $Y$ are ex-spaces over $M$, the external smash product
$Y\bar{\wedge}X$ is an ex-space over $M\times M$.\nidx{$\bar{\wedge}$}  
If $X_n$ is the 
fiber of $X$ over $n$ and $Y_m$ is the fiber of $Y$ over $m$, the fiber 
of $Y\bar{\wedge}X$ over $(m,n)$ is  $Y_m\wedge X_n$. If $X$ 
is a right $\calP M$-module and $Y$ is a left $\calP M$-module then 
$Y\bar{\wedge}X$ is a $\calP M$-$\calP M$-bimodule.

For a right $\calP M$-module $X$ and a left $\calP M$-module $Y$ we
define a space $X\odot Y$\nidx{$\odot $} as the bar resolution
$B(X,\calP M,Y)$.\nidx{B@$B(-,\protect\calP M,-)$} 
The product used to define the bar resolution 
is the $\boxtimes$ product defined
above.  This is analogous to the $\odot $ for two spaces with
an action by a group $\pi$ or to the tensor product of modules
over a ring.  

\begin{definition} We say that a right $\calP M$-module $X$ is
$n$-dualizable\idx{n-dual} if there is a  left $\calP M$-module $Y$, a
continuous  map $\eta\colon S^n\rightarrow X\odot Y$ and a 
map $\epsilon\colon Y\bar{\wedge} X\rightarrow S^n\bar{\wedge}(\calP M, 
\tar\times \sou)_+$ of $\calP M$-$\calP M$-bimodules such that 
\[\xymatrix{S^n\bar{\wedge} X\ar[r]^-{\eta\wedge \id}\ar[dd]_{\gamma}&(X\odot Y)
\bar{\wedge} X\ar[d]^{\cong} &
Y\bar{\wedge} S^n\ar[r]^-{\id\wedge \eta}\ar[dd]_{(\sigma\wedge \id)
\gamma}&Y\bar{\wedge}(X\odot Y)
\ar[d]^{\cong}\\
&X\odot (Y\bar{\wedge} X)\ar[d]^-{\id \odot
\epsilon}&
&((Y\bar{\wedge} X)\odot  Y)\ar[d]^-{ \epsilon \odot  \id}\\
X\bar{\wedge} S^n\ar[r]_-{\cong}& X\odot 
(\vee_{\pi}S^n)
&S^n\bar{\wedge} Y\ar[r]_-{\cong}& (\vee_{\pi}S^n)\odot 
Y}\]
commute stably up to homotopy respecting the action by
$\calP M$.\end{definition}

Let $\sou^*S^\nu$\nidx{ssnu@$\sou^*S^\nu$} denote the pullback of $S^\nu$ along 
$\sou\colon \calP M\rightarrow M$.  Then  
$T_M\sou^*S^\nu$\nidx{Tssnu@$T_M\sou^*S^\nu$} 
is the quotient of $\sou^*S^\nu$ where we identify
$((\gamma,u),v)$ and $((\gamma',u'),v')$  
if $v$ and $v'$ are in the image of the section and 
$\gamma(u)=\gamma'(u')$.  This is an ex-space over $M$ with projection 
given by \[((\gamma,u),v)\mapsto \gamma(u).\]  
We use $T_M$ to denote the quotient 
since $T_M\sou^*S^\nu$ is related to the Thom space.
This is a left $\calP M$-module

\begin{lemma}\mylabel{bunivdual3p}
Let $M$ be a closed smooth manifold.  Then $(\calP M, \sou)_+$ is dualizable
with dual  $T_M\sou^*S^\nu$.
\end{lemma}
We prove this lemma in Chapter \ref{classfpsec} where it is part of
\myref{bunivdual3}.

The coevaluation map \[S^n\rightarrow (\calP M,\sou)_+\odot 
T_M\sou^*S^{\nu}\] is the composite of the Pontryagin-Thom map for
the normal bundle of $M$ with the map that takes an
element $v$ of the normal bundle to $(c_{\rho(v)},c_{\rho(v)},v )$. 
Here $c_x$ is the constant path at $x$.
The evaluation map is more
difficult to describe.   It is closely related to the evaluation
map in \myref{univcovdual1} and it is defined in 
\myref{bunivdual3}.

\section{The homotopy Reidemeister trace as a trace}\label{symmary3a}
Before we define the trace we need to
define the shadow of a space with two actions of $\calP M$.

\begin{definition}Let $Z$ be a $\calP M$-$\calP M$-bimodule. 
The \emph{shadow}\idx{shadow} of $Z$,
$\sh{Z}$,\nidx{shad@$\protect\sh{-}$}
is the cyclic bar resolution $C(Z,\calP M)$.\nidx{C@$C(-,\protect\calP M)$} 
\end{definition}

For a map $f\colon M\rightarrow M$, $(\calP^f M, \tar\times \sou)_+$ 
is the $\calP M$-$\calP M$-bimodule 
\nidx{PfM@${\protect\calP^f M}$} defined by 
\[(\calP^fM, \tar\times \sou)
\coloneqq\{(m,(\gamma,u))\in
M\times \calP M|f(m)=\gamma(0)\}.\] This is a space over $M\times M$ with
projection map $(m,\gamma)\mapsto (\gamma(u),m)$. The actions of $\calP M$
are given by 
composition of paths on the left  and composing with $f$ followed by
composition of paths on the right.  
The shadow of $(\calP^fM, \tar\times \sou)_+$ 
is equivalent to the based space \[\Lambda^fM=\{(\gamma,u)\in
\calP M|f(\gamma(u))=\gamma(0)\}_+.\]  See Section \ref{classhtpysect}
for a description of this equivalence.

For a map $f$ there is an induced map
\[\tilde{f}\colon  (\calP M,\sou)_+\rightarrow (\calP^fM,\sou)_+\]\nidx{f@$\tilde{f}$} 
of  right $\calP M$-modules.

\begin{definition} If $M$ is closed smooth manifold
and $f\colon M\rightarrow M$ is a map, the \emph{trace}\idx{trace} of
$\tilde{f}$ is the stable homotopy class of the map
\[\xymatrix@R=5pt{S^n\ar[r]&(\calP M,\sou )_+\odot T\sou^*S^{\nu}
\ar[r]^-\cong &\sh{T\sou^*S^{\nu}\bar{\wedge} (\calP M,\sou)_+}\\
\,\ar[r]^-{\id\bar{\wedge} \tilde{f}}
&\sh{T\sou^*S^{\nu}\bar{\wedge} (\calP^f M,\sou)_+}\ar[r]&S^n\wedge\sh{(
\calP^fM,
\tar\times \sou)_+}}\]
\end{definition}

\begin{definition} The \emph{homotopy Reidemeister trace},\idx{homotopy 
Reidemeister trace}\idx{Reidemeister trace!homotopy}
$R^{htpy}(f)$, of $f$ is
the trace of $\tilde{f}$.
\end{definition}

The duality and trace defined in this section are also examples of the duality
and trace in bicategories defined in Chapter \ref{whybicat}.  The bicategory
used to define this duality is described in Chapter \ref{classfpsec}.

\chapter{Why bicategories?}\label{whybicat}

In the previous two chapters we have explained some of the reasons why
Dold and Puppe's trace in symmetric monoidal categories cannot describe
the Reidemeister trace.  We have also explained why there should be some
structure, similar to the trace in symmetric monoidal categories, that
does describe the Reidemeister trace.  In this chapter we will describe
that structure.

To achieve the necessary additional generality  we replace
symmetric monoidal categories by bicategories.   
Bicategories have structure
that is very similar to a symmetric monoidal category in the ways
that are important for defining duality and trace.  In particular,
bicategories have `tensor products', or composition, and units.
Bimodules over a ring and spaces with group actions are examples
of bicategories.

Without some additional structure bicategories are not similar
enough to symmetric monoidal categories to have a trace.  To
define a trace we add shadows.  The shadows are closely related to
the bicategory composition and they play a role similar to that of
the symmetry isomorphism in a symmetric monoidal category.

Some results that follow easily from the definitions of duality and trace
significantly simplify proofs of results stated in Chapter \ref{summary}.  We
include those results here and complete the proofs omitted from 
Chapter \ref{summary} in Chapters \ref{classfpsec}, \ref{classfpsec2}, 
\ref{fibfpsec} and \ref{fibfpsec2}.

We omit most proofs in this section since they are diagram chases
from the definitions.

\section{Definitions}\label{bicatdef}
Bicategories can be thought of as monoidal categories with many
objects.  Instead of having objects and morphisms, bicategories
have 0-cells, 1-cells, and 2-cells. Each 1-cell or 2-cell in a
bicategory $\sB$\nidx{B@$\protect\sB$} 
has a source 0-cell and a target 0-cell.  For two
0-cells $A$ and $B$, the 1-cells and 2-cells with source $A$ and
target $B$ form a category with objects the 1-cells and morphisms
the 2-cells.  This category is usually written $\sB(A,B)$.  In
addition, for 0-cells $A,B$, and $C$ there is a  functor\nidx{$\odot$} 
\[\odot\colon \sB(B,C)\times\sB(A,B)\rightarrow \sB(A,C)\]
that acts as
`composition' for 1-cells and 2-cells. For each 0-cell $A$ 
there is a functor
$U_A$\nidx{Ua@$U_A$} from the category with one object and
one morphism to $\sB(A,A)$.  Up to isomorphism 2-cells, the
functors $\odot$ are associative and unital with respect to the
functors $U_A$.\idx{0-cell}\idx{1-cell}\idx{2-cell}

\begin{definition}\cite[1.0]{Leinster} A \emph{bicategory}\idx{bicategory} 
$\sB$ consists of
\begin{enumerate}\item A collection $\ob\sB$.
\item Categories $\sB(A,B)$ for each $A,B\in \mathrm{ob}\sB$.
\item Functors \[\odot \colon \sB(B,C)\times \sB(A,B)\rightarrow
\sB(A,C)\]
\[U_A\colon \ast \rightarrow \sB(A,A)\]
for $A$, $B$ and $C$ in $\mathrm{ob}\sB$.
\end{enumerate}
Here $\ast$ denotes the category with one object and one morphism.
The functors $\odot$ are  required to satisfy unit and
associativity conditions  up to  natural isomorphism 2-cells.
\end{definition}

\begin{rmk}
There are two choices for the convention used for the bicategory composition.  
We follow the convention used in \cite{maclane}, but the 
choice \[\odot\colon  \sB(A,B)\times \sB(B,C)\rightarrow \sB(A,C)\]
is also used.
\end{rmk} 

Some examples of bicategories include:
\begin{itemize}
\item The bicategory with 0-cells rings, 1-cells bimodules, and
2-cells homomorphisms. \item The bicategory with 0-cells
categories, 1-cells functors, and 2-cells natural transformations.
\item The bicategory $\Ex$\nidx{ex@$\Ex$} of ex-spaces with 0-cells spaces; and
for two spaces $A$ and $B$ the category $\Ex(A,B)$ is the category
of ex-spaces over $B\times A$.\footnote{The name $\Ex$ is used to
refer to a different, but related, category in \cite{MS}.}
\end{itemize}

A monoidal category is a bicategory with a single 0-cell.  The
objects of the monoidal category are the 1-cells of the bicategory
and the morphisms are the 2-cells.  The monoidal product of the
monoidal category is the $\odot$ in the bicategory and the unit
object $I$ is used to define the single functor $U_I$.  Chapter
\ref{reviewbicat} contains more examples of bicategories.

\begin{definition}
A \emph{lax functor}\idx{lax functor of bicategories} 
$F\colon \sB\rightarrow \sB'$ between bicategories
consists of 
\begin{enumerate}
\item A function $F$  from the
0-cells of $\sB$ to the 0-cells of $\sB'$;
\item Functors $F\colon \sB(A,B)\rightarrow
\sB'(FA,FB)$ for all pairs of 0-cells $A$ and $B$ of $\sB$;  
\item Natural transformations \[\phi_{X,Y}\colon F(X)\odot F(Y)\rightarrow
F(X\odot Y)\] for
all 1-cells $X$ and $Y$;\nidx{phixy@$\phi_{X,Y}$}
\item Natural transformations $\phi_A\colon 
U'_{F(A)}\rightarrow F(U_A)$ for each 0-cell $A$\nidx{phia@$\phi_A$}
\end{enumerate}
that satisfy some coherence conditions.
\end{definition}

This is analogous
to a monoidal functor and its natural transformations \[F(A)\otimes
F(B)\rightarrow F(A\otimes B)\] and $I'\rightarrow F(I)$.
A \emph{strong functor}\idx{strong functor of bicategories} 
of bicategories is a lax functor where
the natural transformations $\phi_{X,Y}$ and $\phi_A$ are natural
isomorphisms.

The duality of Section \ref{doldsec} can be defined in a monoidal
category, but it is more common to define dual pairs in symmetric
monoidal categories.  Similarly, while we could make all of the
definitions in this chapter for a bicategory, we choose to impose
some additional conditions that will mimic some of the symmetry
conditions in a symmetric monoidal category.

\begin{definition}\cite[16.2.1]{MS}
An {\em involution}\idx{involution} on a bicategory $\sB$ consists of the
following data.
\begin{enumerate}
\item A bijection $t$ on the $0$-cells of $\sB$ such  that $ttA =
A$. 
\item Equivalences of categories $t\colon \sB(A,B)\rtarr
\sB(tB,tA) = \sB^{\text{op}}(tA,tB),$ with the equivalences given
by isomorphism $2$-cells $\xi\colon\id\iso tt$. 
 
\item Natural
isomorphism $2$-cells $\io = \io_A \colon tU_A\rtarr U_{tA}$ for
$0$-cells $A$ and
$$\ga = \ga_{X,Y}\colon (tY)^{\text{op}}\odot^{\text{op}} (tX)^{\text{op}}
\equiv tX \odot tY \rtarr t(Y\odot X)$$ for $1$-cells $X\colon
A\rtarr B$ and $Y\colon B\rtarr C$; the left and right unit
$2$-cells $\la$ and $\rh$ must be related by $t(\la_X)\ga_{X,U_X}
= \rh_{tX}(\id\odot\io)$ or, equivalently, $t(\rh_X)\ga_{U_X,X} =
\la_{tX}(\io\odot\id)$, the appropriate hexagonal coherence
diagram relating $\ga$ to the associativity $2$-cell $\al$ must
commute, and the following diagram relating $\xi$ to $\ga$ must
commute:
$$\xymatrix{
Y\odot X \ar[d]_{\xi\odot \xi} \ar[r]^-{\xi} &  tt(Y\odot X)\\
ttY\odot ttX \ar[r]_{\ga} & t(tX\odot tY). \ar[u]_{t(\ga)}\\}$$
We also require $\xi(U_A)=t(\iota_A)\iota_A$. 
\end{enumerate}
A {\em symmetric bicategory}\idx{bicategory!symmetric}\idx{symmetric 
bicategory} 
$\sB$ is a bicategory equipped with
an involution.
\end{definition}

A bicategory $\sB$ is \emph{closed}\idx{bicategory!closed}
\idx{closed!bicategory} 
if for 0-cells $A,\,B,\,C$ there are two functors
$$ \rc\colon \sB(A,B)^{\text{op}}\times \sB(A,C)\rtarr \sB(B,C) $$
and $$ \lc\colon \sB(A,C)\times \sB(B,C)^{\text{op}} \rtarr \sB(A,B)$$
and natural isomorphisms
\begin{equation}\label{intadj}\nonumber
\sB(A,B)(X,Z\lc Y) \cong \sB(A,C)(Y\odot X, Z) \cong \sB(B,C)(Y,X\rc Z)
\end{equation}
for $1$-cells $X\in\sB(A,B)$, $Y\in\sB(B,C)$, and $Z\in\sB(A,C)$.\nidx{$\rc$}
\nidx{$\lc$}

In a symmetric bicategory we can use the involution to express $\lc$
in terms of $\rc$ or $\rc$ in terms of $\lc$.

For more detailed definitions see 
\cite{Leinster, MS}.

\section{Rings, bimodules, and maps}
Many of the important features of bicategories can be seen in the
example of rings, bimodules, and homomorphisms.  The 0-cells of
this bicategory are rings, the 1-cells with source $A$ and target
$B$ are the $B$-$A$-bimodules. The 2-cells between two $B$-$A$
bimodules are the bimodule homomorphisms.  The $\odot$ is the
usual tensor product of modules over a ring.  The functor $U_A$
associated to a ring $A$ is $A$ regarded as an $A$-$A$-bimodule.

The bicategory of rings, bimodules, and homomorphisms,
denoted by $\Mod$,\nidx{M@$\protect\Mod$} 
is symmetric with the involution that takes a ring
to its opposite and takes an $A$-$B$-bimodule to its opposite, a
$B^{op}$-$A^{op}$-bimodule.  The isomorphism \[X\otimes_BY\cong
(Y^{op}\otimes_{B^{op}}X^{op})^{op}\] for an $A$-$B$-bimodule $X$
and $B$-$C$-bimodule $Y$ is also part of the structure of the
involution.

The bicategory $\Mod$ gives familiar examples of the functors $\rc$
and $\lc$. Let $X$ be a $B$-$A$-bimodule, $Y$ be a $C$-$B$-bimodule,
and $Z$ be a $C$-$A$-bimodule.  Then $X\rc Z$ is $\Hom_A(X,Z)$, the
$C$-$B$-bimodule of right $A$-module maps from $X$ to $Z$.
Similarly $Z\lc Y$ is $\Hom_C(Y,Z)$, the $B$-$A$-bimodule of left
$C$-module maps from $Y$ to $Z$. There are natural isomorphisms
\[\Mod(B,C)(Y,\Hom_A(X,Z))\cong\Mod(A,C)(Y\otimes_BX,Z)
\cong\Mod(A,B)(X,\Hom_C(Y,Z)).\]

To define the algebraic Reidemeister trace we defined the
Hattori-Stallings trace of an endomorphism of a finitely generated
projective module.  There we defined a  map
\[\nu\colon X\otimes_AP\otimes_A\Hom_A(X,A)\rightarrow \Hom_A(X,X
\otimes_AP)\] which is an isomorphism for all $A$-$A$-bimodules $P$
and finitely generated projective right $A$-modules $X$. The map $\nu$
can also be defined using the adjunction isomorphisms above.  
The adjunction isomorphisms 
also give a map \[\eta\colon \mathbb{Z}\rightarrow \Hom_A(X,X)\] for
any right $A$-module $X$.  If $X$ is a finitely generated
projective right $A$-module then the composite of $\eta$ with
$\nu^{-1}$ gives a map \[\mathbb{Z}\rightarrow
X\otimes_A\Hom_A(X,A).\] In fact, the existence of a map
\[\mathbb{Z} \rightarrow X\otimes _A\Hom_A(X,A)\] satisfying some
additional conditions is equivalent to $\nu$ being an isomorphism.

\begin{prop}\mylabel{dualrings} The following are equivalent for a right
$A$-module $X$.
\begin{enumerate}
\item $X$ is a finitely generated projective right $A$-module.
\item The map \[\nu\colon P\otimes_A\Hom_A(X,A)\rightarrow \Hom_A(X,P)\]
is an isomorphism for all $A$-$A$-bimodules $P$. \item  There is a
map $\eta\colon \mathbb{Z} \rightarrow X\otimes_AD_AX$,
$D_AX=\Hom_A(X,A)$, such that
\[\xymatrix{D_AX\ar[r]^-\cong\ar[d]^\id &{D_AX\otimes_\mathbb{Z}\mathbb{Z}}
\ar[r]^-{\id\otimes \eta}
&{D_AX\otimes_\mathbb{Z} (X\otimes_AD_AX)}\ar[d]^\cong\\
D_AX\ar[r]^-\cong&A\otimes_AD_AX&{(D_AX\otimes_{\mathbb{Z}}X)
\otimes_AD_AX}\ar[l]^-{\ev\otimes\id}}\]
and
\[\xymatrix@C=15pt{X\ar[r]^-\cong\ar[d]^\id&{\mathbb{Z}\otimes_{\mathbb{Z}}
X} \ar[r]^-{
\eta \otimes \id}&{(X\otimes_AD_AX)\otimes_{\mathbb{Z}}X}\ar[d]^{\cong}\\
X\ar[r]^-\cong&X\otimes_AA&{X\otimes_A(D_AX\otimes_\mathbb{Z}X)}
\ar[l]^-{\id\otimes \ev}}\] commute.
\end{enumerate}
\end{prop}

Using the map $\eta$ and the functor $\sh{-}$ from
\myref{ringshadow} we can give another
description of the Hattori-Stallings trace.

\begin{prop} If $X$ is a finitely generated projective right $A$-module,
$P$ is an $A$-$A$-bimodule, and $f\colon X\rightarrow X\otimes_AP$ is a
map of right $A$-modules, then the Hattori-Stallings trace of $f$
is the image of 1 under the composite
\[\xymatrix@C=20pt{{\mathbb{Z}}\ar[d]^-{\eta}&&\sh{P}\\
X\otimes_AD_A(X)
\ar[r]^-{f\otimes \id}& X\otimes_A P\otimes_A D_A(X)\ar[r]^-\cong
&\sh{P\otimes_AD_A(X)\otimes_\mathbb{Z}X}\ar[u]^-{\sh{\ev}}.}\]
\end{prop}

In the rest of this chapter we generalize this description of the
Hattori-Stallings trace to a trace in bicategories with shadows.  We
begin by recalling the generalization of
\myref{dualrings} to a bicategory from \cite{MS}.  Then we define shadows which
generalize the functors $\sh{-}$.

\section{Duality}
Duality in a bicategory is similar to  duality in a symmetric
monoidal category.  Many of the differences are already seen in
\myref{dualrings}, which we generalize here.   
This section is based on Chapter 16
of \cite{MS} which contains additional details.

\begin{definition} A 1-cell $X\in \sB(B,A)$ is \emph{right dualizable} 
\idx{right dualizable}\idx{dualizable!right} if
there is a 1-cell $Y\in\sB(A,B)$ and maps $\eta\colon U_A\rightarrow
X\odot Y$, called the \emph{coevaluation},\idx{coevaluation} 
and $\epsilon\colon Y\odot
X\rightarrow U_B$, called the \emph{evaluation},\idx{evaluation}
 such that the following diagrams
commute in $\sB (B,A)$ and $\sB (A,B)$, respectively.
\[\xymatrix@C=22pt{X \ar[d]_{\id} & U_A \odot X \ar[l]_-\cong\ar[r]^-{\eta\odot \id}
& (X\odot Y)\odot X \ar[d]^\cong\\
X & X\odot U_B \ar[l]^-\cong & X\odot (Y\odot X)\ar[l]^-{\id\odot
\epsilon} }
\qquad \xymatrix@C=22pt{Y \ar[d]_{\id}  & Y \odot U_A \ar[l]_-\cong\ar[r]^-{\id\odot \eta} & Y \odot (X\odot Y) \ar[d]^\cong \\
Y & U_B\odot Y \ar[l]^-\cong & (Y\odot X)\odot
Y\ar[l]^-{\epsilon\odot \id}}\]  
\end{definition}

If $X$ is right dualizable with dual $Y$
we say that $Y$ is the \emph{right dual}\idx{right dual}\idx{dual!right}
 of $X$ and that $(X,Y)$ is a
\emph{dual pair}.\idx{dual pair}
We also say that $Y$ is \emph{left dualizable} and that $X$ is the 
\emph{left dual} of $Y$.

\begin{prop}
Let $X$ be a 1-cell in $\sB(B,A)$, $Y$ be a 1-cell in $\sB(A,B)$, 
and $\epsilon\colon Y\odot X\rightarrow U_B$ be a 2-cell in $\sB(B,B)$.
Then the following are equivalent.
\begin{enumerate}
\item $Y$ is the right dual of $X$ with evaluation $\epsilon$.
\item The map $\epsilon/(-)\colon \sB(W, Z\odot Y)\rightarrow \sB(W\odot X,
Z)$ which takes a 2-cell $f\colon W\rightarrow Z\odot Y$ to the 2-cell
\[\xymatrix{W\odot X\ar[r]^-{f\odot \id}&Z\odot Y\odot X\ar[r]^-{\id \odot 
\epsilon}&Z\odot U_B\cong Z}\] is a bijection for all $W\in \sB(A,C)$
and $Z\in \sB(B,C)$.
\end{enumerate}
\end{prop}

There is a similar characterization using the coevaluation. 

As in the symmetric monoidal case, any two right duals of a right
dualizable object are isomorphic.  If the duals of $X$ are $Y$ and
$Y'$ with coevaluation and evaluation maps
\[\xymatrix@R=5pt{\eta\colon U_A\ar[r]& X\odot Y
&{\mathrm{and}} &\epsilon\colon  Y\odot X\ar[r]& U_B\\
\eta'\colon U_A\ar[r]&X\odot Y'&{\mathrm{and}}&\epsilon'\colon Y'\odot X\ar[r]&
U_B}\]
then the map \[\xymatrix{Y\cong Y\odot U_A\ar[r]^{\id\odot \eta'}&
Y\odot X\odot Y'\ar[r]^{\epsilon\odot \id}&U_B\odot Y'
\cong Y'}\]
is an isomorphism with inverse \[\xymatrix{
Y'\cong Y'\odot U_A\ar[r]^{\id\odot \eta}
&Y'\odot X\odot Y\ar[r]^{\epsilon'\odot \id}&U_B\odot Y
\cong Y}.\]

If the bicategory is closed there
are canonical duals\idx{canonical dual} 
given by the $\rc$ and $\lc$ functors.  This
is the case in the bicategory of rings, bimodules, and
homomorphisms.
\begin{prop}
If $\sB$ is a closed bicategory the following are equivalent for a 1-cell $X\in \sB(B,A)$.
\begin{enumerate}
\item $X$ is right dualizable.
\item The map $\nu\colon  X\odot (X\rc U_B)\rightarrow X\rc X$ is an isomorphism.
\end{enumerate}
\end{prop}
There is a similar result for left dualizable 1-cells.

A dual pair in a symmetric monoidal category is also a dual pair
in the corresponding one 0-cell bicategory.  However, in a
bicategory an object that is right dualizable might not be left
dualizable, and conversely.

The following results about composites of dual pairs follow
immediately from the definition of a dual pair.  While they are
both very easy to prove they have significant consequences.

\begin{theorem}\mylabel{dualcomposites1}
Let $X$ be a right dualizable 1-cell in $\sB(B,A)$ with dual $Y$
and $W$ be a right dualizable 1-cell in $\sB(C,B)$ with dual $Z$.

Let $(\eta,\epsilon)$ be coevaluation and evaluation maps for the dual
pair $(X,Y)$, and let $(\zeta,\psi)$ be coevaluation and evaluation maps
for the dual pair $(W,Z)$. Then the composites
\[\xymatrix@C=40pt{U_A \ar[r]^-{\eta}
& X\odot Y \ar[r]^-\iso
& X \odot U_B \odot Y \ar[r]^-{\id\odot \zeta\odot\id}
& (X\odot W)\odot (Z\odot Y)}\]
and
\[\xymatrix@C=40pt{ (Z \odot Y)\odot (X\odot W) \ar[r]^-{\id\odot\epsilon\odot \id}
& Z \odot U_B \odot W \ar[r]^-\iso
& Z \odot W \ar[r]^-\psi& U_C} \]
are coevaluation and evaluation maps that exhibit $(X\odot W, Z\odot Y)$
as a dual pair of $1$-cells.
\end{theorem}

\begin{theorem}\mylabel{dualcomposites2}
Let $(X,Y)$ be a dual pair with  evaluation
$\epsilon\colon Y\odot X\rightarrow U_B$.  Let $Z$ be another 1-cell and
suppose $(X\odot Z,V)$ is a dual pair. If $\epsilon$ is an isomorphism
then $(Z,V\odot X)$ is a dual pair.

Let $(Z,W)$ be a dual pair with coevaluation $\chi\colon U_B\rightarrow Z\odot W$
and $X$ be a 1-cell.  If $(X\odot Z,V)$
is a dual pair and $\chi$ is an isomorphism, then $(X, Z\odot V)$ is a
dual pair.
\end{theorem}

Strong functors of bicategories are compatible with dual pairs,
but weaker hypotheses can also give compatibility between functors
and dual pairs.
\begin{prop}\mylabel{bicatfuntoriality} 
Let $(X,Y)$ be a dual pair in a bicategory $\sB$, $X\in\sB(B,A)$, and
$F\colon \sB\rightarrow \sB'$ be a lax functor of bicategories such that $\phi_{X,Y}$
and $\phi_B$ are isomorphisms.  Then $(FX,FY)$ is a dual pair in $\sB'$.

If the coevaluation and evaluation maps for the dual pair $(X,Y)$ are
$$\eta\colon U_A\rtarr X\odot Y \ \ \
\text{and}\ \ \ \epsilon\colon Y\odot X\rtarr U_B,$$ then the coevaluation
and evaluation maps for
the dual pair $(FX,FY)$ are
\[\xymatrix@C=40pt{U_{FA}'\ar[r]^-{\phi_A}& F(U_A)\ar[r]^-{F(\eta)}& F(X\odot Y)
\ar[r]^-{(\phi_{X,Y})^{-1}}& FX\odot FY}\] and
\[\xymatrix@C=40pt{FY\odot FX\ar[r]^-{\phi_{Y,X}}& F(Y\odot X)\ar[r]^-{F(\epsilon)}& F(U_B)
\ar[r]^-{(\phi_B)^{-1}}&U_{FB}'.}\]

\end{prop}

Let $\Ch$\nidx{ch@$\Ch$} be the bicategory with 0-cells rings, 1-cells chain complexes
of bimodules, and 2-cells maps of chain complexes. 
Let $C$ be a chain complex of left $R$-modules and $D$  be a chain complex
 of right $R$-modules.  Suppose $(D,C)$ is a dual pair. There is a map
\[H_*(D)\odot H_*(C)=H_*(D)\otimes_RH_*(C)\rightarrow H_*(D\otimes_RC).
\]  The K\"unneth Theorem implies this map  is an isomorphism if each  $C_i$
 is a projective module, the boundaries of $C_i$ are projective
and  the homology of $C$ is projective in each degree.  The
natural transformations $\phi_R$ are the identity for all rings
$R$ and so when the hypotheses of the K\"unneth Theorem are satisfied
\myref{bicatfuntoriality}
implies that the homology of a dualizable complex is
dualizable.

\section{Shadows}\label{shadowssection}

In symmetric monoidal categories the symmetry isomorphism
\[X\otimes Y\rightarrow Y\otimes X\] provides a way to compare
the target of the coevaluation with the source of the evaluation
for a dual pair $(X,Y)$. This is an important part of the
definition of the trace. To define trace in a bicategory we will
also need to be able to compare the target of coevaluation with
the source of evaluation. For example, in the bicategory of rings,
bimodules, and homomorphisms it is necessary to compare $X\otimes
_AY$ with $Y\otimes _BX$ for an $B$-$A$-bimodule $X$ with dual $Y$.

The comparisons we need are not automatically part of the
structure of a bicategory, as we can see in the bicategory $\Mod$.  In 
$\Mod$ we introduced the functors $\sh{-}$ to define the
Hattori-Stallings trace.  In this section we describe how to
generalize these functors to other bicategories.

\begin{definition} A bicategory $\sB$ has \emph{shadows}\idx{shadow} 
if there is a 0-cell
$I$, functors\nidx{shad@$\protect\sh{-}$} \[\sh{-}\colon \sB(A,A)\rightarrow \sB(I,I)\]
for all
0-cells $A$, and natural isomorphisms\nidx{theta@$\theta$} \[\theta\colon  \sh{X\odot
Y}\cong \sh{Y\odot X} \] for all pairs of 1-cells $X\in\sB(B,A)$
and $Y\in\sB(A,B)$. We also require that for $X\in \sB(I,I)$,
$\sh{X}\cong X$ and the following diagrams, relating the
isomorphisms $\theta$ to the unit and associativity isomorphisms
in the bicategory,  commute.\footnote{The following diagrams
 are due to Michael Shulman.}
\[\xymatrix{\sh{(X\odot Y)\odot Z}\ar[r]^\theta\ar[d]&\sh{Z\odot (X\odot Y)}
\ar[r]&\sh{(Z\odot X)\odot Y}\\
\sh{X\odot(Y\odot Z)}\ar[r]_\theta&\sh{(Y\odot Z)\odot X}\ar[r]&
\sh{Y\odot (Z\odot X)}\ar[u]_\theta
}\]
\[\xymatrix{\sh{Z\odot U_A}\ar[r]^\theta\ar[dr]&\sh{U_A\odot Z}\ar[d]
\ar[r]^\theta&\sh{Z\odot U_A}\ar[dl]\\&\sh{Z}}\]
\end{definition}

Shadows can be thought of as `cyclic
tensor products' since the natural isomorphisms will allow cyclic
permutations of 1-cells and 2-cells.

As we noted before, from a symmetric monoidal category we can
define a bicategory with a single 0-cell, 1-cells the objects of
the category, and 2-cells the morphisms.  The identity functor is
a shadow for this bicategory and the isomorphism $\theta$ is the
symmetry isomorphism.

The bicategory $\Ch$ is a bicategory with shadows.  The shadows are
given by applying the shadows of $\Mod$ levelwise.  The
isomorphism $\sh{X\odot Y}\rightarrow \sh{Y\odot X}$ is the usual
exchange of elements and adds a sign determined by degree.

Since shadows are not automatically part of the structure of a bicategory,
it is not surprising that additional hypotheses will be needed before a lax
functor is considered compatible with shadows.

\begin{definition} Let $\sB$ and $\sB'$ be bicategories with shadows and
$F\colon \sB\rightarrow \sB'$ be a lax functor of bicategories.  Then
$F$ is \emph{compatible with shadows}\idx{compatible with shadows} 
if for each 0-cell $A$ there
is a natural transformation \[\psi_A\colon \sh{F(-)}\rightarrow F\sh{-}\]\nidx{psia@$
\psi_A$}
such that
\[\xymatrix{\sh{FX\odot FY}\ar[r]^-\theta\ar[d]_{\phi_{X,Y}}&\sh{FY\odot FX}
\ar[d]^-{\phi_{Y,X}}\\
\sh{F(X\odot Y)}\ar[d]_{\psi_A}&\sh{F(Y\odot X)}\ar[d]^{\psi_B}\\
F\sh{X\odot Y}\ar[r]_\theta&F\sh{Y\odot X}}\] commutes for all
1-cells $X\in\sB(B,A)$ and $Y\in \sB(A,B)$.
\end{definition}

Homology is a lax functor of bicategories that is compatible with
shadows. The lax functor is the identity on 0-cells and the usual
homology functor on 1-cells and 2-cells. Define $\psi_R$ by the
coequalizer
\[\xymatrix{R\otimes H_*(C_*)\ar[r]\ar[d]&H_*(R\otimes C_*)\ar@<.5ex>[r]^-
{H_*(\kappa)}\ar@<-.5ex>[r]_-{H_*(\kappa\gamma)}&H_*(C_*)\ar@{=}[d]\ar[r]
&\sh{H_* (C_*)}\ar@{.>}[d]^{\psi_R}\\ H_*(R\otimes
C_*)\ar@<.5ex>[rr]^{H_*(\kappa)}\ar@<-.5ex>[rr]
_{H_*(\kappa\gamma)}&&H_*(C_*) \ar[r]&H_*(\sh{C_*}).}\]

\begin{rmk}
The definitions of shadows and a functor that is compatible with
shadows given here can be generalized while retaining the same
spirit, see \cite{PS1}.  Here we require that the shadow is a
functor from the hom categories to one chosen hom category in our
bicategory.  We could also define a shadow as a functor from the
hom categories to some symmetric monoidal category that has
associated symmetry isomorphisms that satisfy coherence
conditions.  We could similarly generalize the definition of a
functor that is compatible with shadows.
\end{rmk}

\section{Trace}\label{trace}
Motivated by the definition of trace in a symmetric monoidal category and
the Hattori-Stallings trace we can now use duality in a bicategory to define
trace in a bicategory with shadows.

In this section $\sB$ is a bicategory with shadows and $X\in \sB
(B,A)$ is a 1-cell
with right dual $Y\in\sB(A,B)$.  Let $\eta \colon U_A\rightarrow X\odot Y$
and $\epsilon \colon Y\odot X \rightarrow U_B$ be the coevaluation and evaluation
for the  dual pair $(X,Y)$.

\begin{definition}
For 1-cells $P\in\sB(B,B)$ and  $Q\in \sB(A,A)$ and a 2-cell \[f\colon 
Q\odot X\rightarrow X\odot P\] the \emph{trace}\idx{trace} 
of $f$ is the composite
\[\xymatrix{\sh{Q}\ar[d]^\cong
&&& \sh{P}\\
\sh{Q\odot U_A}
\ar[r]^-{\sh{\id\odot\eta}}&\sh{Q\odot X\odot Y} 
\ar[r]^-{\sh{f\odot\id}}&\sh{X\odot P\odot Y}\cong\sh{Y\odot X\odot
P} \ar[r]^-{\sh{\epsilon \odot \id}}&\sh{U_B\odot P}\ar[u]^\cong.}\]
\end{definition}

In a symmetric monoidal category, trace was defined only for endomorphisms
of dualizable objects. In a bicategory we add the 1-cells $P$
and $Q$ since we want to use this trace in fixed point theory
applications.  For these examples the maps we want to take the trace of
are not endomorphisms of 1-cells. Rather, they are 2-cells of the form
$X\rightarrow X\odot P$.  This can be seen in the definition of
the algebraic Reidemeister trace where $X=C_*(\tilde{M};
\mathbb{Q})$ for some compact manifold $M$ and
$P=\mathbb{Z}\pi_1(M)^\phi$.  While $Q$ is not needed in our applications,
we add it to the definition for symmetry.

We define the trace of a 2-cell $g\colon Y\odot Q\rightarrow P\odot Y$ similarly.

The following lemmas describe basic properties of the trace.  All of these 
lemmas are easy to prove. 

\begin{lemma}\mylabel{traceind}  
If $(X,Y)$ and $(X,Y')$ are dual pairs, the trace of $f$ with respect to
$(X,Y)$ is equal to the trace of $f$ with respect to $(X,Y')$.
\end{lemma}

Let $f'$ be the composite
\[\xymatrix@C=35pt{Y\odot Q\ar[r]^-{\cong}&
Y\odot Q\odot U_A\ar[r]^-{\id \odot \id\odot \eta}&
Y\odot Q\odot X\odot Y\ar[r]^-{\id\odot f\odot \id}&
\,}\]
\vspace{-12pt}\hspace{-50pt}
\[\xymatrix@C=35pt{
&Y\odot X\odot P\odot Y\ar[r]^-{\epsilon\odot \id\odot \id}&U_B\odot P
\odot Y\ar[r]^-\cong & P\odot Y.
}\]
The 2-cell $f'$ is the \emph{dual} of $f$.

\begin{lemma}\mylabel{tracedual} For $f$ and $f'$ as above,
\[\tr(f)=\tr (f').\]
\end{lemma}

One of the defining properties of a trace function on matrices is
commutativity.  The trace in bicategories is also commutative.

\begin{lemma}\mylabel{tracecyclic} If $X$ and $Z$ are right dualizable 1-cells, 
$g\colon R\odot Z\rightarrow X\odot S$, $f\colon Q\odot
X\rightarrow Z\odot P$ are 2-cells, and the composites 
\[\xymatrix@R=5pt{Q\odot R\odot Z\ar[r]^-{\id_Q\odot g}&Q\odot X\odot S
\ar[r]^-{f\odot \id_S}&Z\odot P\odot S \\
R\odot Q\odot X\ar[r]^-{\id_R\odot f}&R\odot Z\odot P\ar[r]^-{g\odot \id_P}&
X\odot S\odot P}\] are defined, then \[\tr((f\odot \id_S)(\id_Q\odot g))=\tr ((g\odot \id_P)
(\id_R\odot f))\]
\end{lemma}

The trace respects the $\odot$ structure.

\begin{lemma}\mylabel{tracemult} If $X$ and $Y$ are right dualizable 1-cells, 
$f\colon Q\odot X\rightarrow X$ and $g\colon Z\rightarrow Z\odot P$
are 2-cells, and \[g\odot f\colon Z\odot Q\odot X\rightarrow Z\odot P
\odot X\] is defined, then 
\[\protect\tr (g\odot f)=\tr(g)\tr(f).\]
\end{lemma}

Some of the dual pairs that we will consider later have much more structure
than is required by the definitions.  The additional structure gives more information
about the traces.

\begin{lemma}\mylabel{tracegpd}
\begin{enumerate}
\item Let $(X,Y)$ be a dual pair such that the evaluation $\epsilon$ is
an isomorphism and let $Z$ be another 1-cell such that $X\odot Z$
is dualizable.  For a 2-cell $g\colon Q\odot Z\rightarrow Z\odot P$ 
let $g^\star$ be the composite \[\xymatrix{&X\odot Q\odot Y\odot X\odot
Z\ar[rr]^-{\id\odot \id\odot \epsilon \odot \id}&&X\odot Q\odot
U_A\odot Z\ar[d]^\cong\\& && X\odot Q\odot Z\ar[r]^-{\id\odot g}
&X\odot Z\odot P.}\]

Then \[\xymatrix{{\sh{X\odot Q\odot Y}}\ar[r]^-\cong&
{\sh{Y\odot X\odot Q}}\ar[r]^-{\sh{\epsilon\odot \id}}
&{\sh{U_A\odot Q}\cong \sh{Q}}\ar[r]^-{\tr (g)}&{\sh{P}}}\]
is the trace of $g^\star$.

\item Let $(Z,W)$ be a dual pair such that the coevaluation $\chi$ is an
isomorphism and let $X$ be another 1-cell such that $X\odot Z$ is dualizable.
If $f\colon Q\odot X\rightarrow X\odot P$ is a 2-cell let $f^\star$ be the
composite
\[\xymatrix{&Q\odot X\odot Z\ar[r]^-{f\odot \id}&X\odot P\odot Z
\ar[d]^\cong\\& &X\odot U_A\odot P\odot Z\ar[rr]^-{\id\odot
\chi\odot \id\odot \id}&& X\odot Z\odot W\odot P\odot Z.}\]

Then
\[\xymatrix{{\sh{Q}}\ar[r]^-{\tr (f)}&{\sh{P}\cong
\sh{P\odot U_A}}
\ar[r]^-{\sh{\id\odot \chi}}&{\sh{P\odot Z\odot W}}\ar[r]^{\cong}&
{\sh{W\odot P\odot Z}}}\]
is the trace of $f^\star$.

\end{enumerate}
\end{lemma}

Strong symmetric monoidal functors preserve dual pairs and trace
in a symmetric monoidal category, as do lax symmetric
monoidal functors that satisfy some additional hypotheses. 
Strong functors of bicategories, and lax
functors of bicategories where some of the coherence natural transformations
are isomorphisms, preserve dual pairs  in a bicategory.  Strong functors
that are compatible with shadows almost preserve the trace.

\begin{prop}\mylabel{functortrace} Let $F$ be a lax functor
compatible with shadows and $(X,Y)$ a dual pair such that
\[\phi_{X,Y}\colon F(X)\odot F(Y)\rightarrow F(X\odot Y)\]
and \[\phi_B\colon U'_{F(B)}\rightarrow F(U_B)\] are isomorphisms.  If
$f\colon Q \odot X \rightarrow X\odot P$ is a 2-cell, $\phi_{Q,X}$ is an
isomorphism and $\hat{f}$\nidx{f@$\hat{f}$} is the composite
\[\xymatrix{FQ\odot FX\ar[r]^{\phi_{Q,X}^{-1}}&F(Q\odot X)\ar[r]^{F(f)}
&F(X\odot P)\ar[r]^{\phi_{X,P}}&FX\odot FP
}\]
then the following diagram commutes.
\[\xymatrix@C=50pt
{\sh{FQ}\ar[r]^{\tr (\hat{f})}
\ar[d]_{\psi_A}&\sh{FP}\ar[d]^{\psi_B}\\
F\sh{Q}\ar[r]_-{F(\tr (f))}&F\sh{P}}\]
\end{prop}

For the homology functor, the natural transformations $\phi_A$ are
all the identity and so the conditions of \myref{functortrace}
are all consequences of the K\"unneth Theorem.  

\begin{corollary} Let $C$ be a finitely generated chain complex of projective
right $R$-modules such that the boundaries and homology of $C$ are projective
in each degree.  If  $f\colon C\rightarrow C\otimes_R P$ is  map of chain complexes
then
\[\xymatrix@C=60pt{{\mathbb{Z}}\ar@{=}[d]\ar[r]^-{\tr(\phi\circ H_*(f))}
&\sh{H_*(P)}
\ar[d]^{\psi_A}\\ {\mathbb{Z}}\ar[r]_-{H_*(\tr(f))}&H_*(\sh{P})}\]
commutes.\end{corollary}

In particular, if $M$ is a finite CW complex and
$H_*(\tilde{M};\mathbb{Z})$ is projective as a right module over
$\mathbb{Z}\pi_1M$ then the algebraic Reidemeister trace computed
using the chains on $M$ is the same as the trace of the induced
map on homology.  Compare this observation with
\cite[1.4]{Husseini} and  \cite[4.3.b]{Dold}.

\chapter{Duality for parametrized modules}\label{classfpsec}

In this chapter we give several examples of the duality defined in the 
previous chapter. 
We will first describe the bicategory of ex-spaces and one particular 
example of duality in this bicategory.

From the bicategory of ex-spaces we can define a bicategory that is a
topological analogue of the bicategory of rings, bimodules, and
homomorphisms.  After defining the bicategory we give several
examples of dual pairs.  
These examples are similar to those
in Chapter \ref{summary}, but now we use the formal results from
Chapter \ref{whybicat} to simplify many proofs.

The results from Chapter \ref{whybicat} that do the most to simplify
the proofs here are the results about composites of dual pairs.
The first of these results shows that the composite of two dual
pairs is a dual pair.  This result, along with a particular dual
pair for a compact smooth manifold, will produce many of the dual
pairs we described in Chapter \ref{summary}. 

In the next chapter we use these dual pairs to show that several
forms of the Reidemeister trace are examples of trace in bicategories
and we use functoriality of the trace to relate these invariants.

\section{Costenoble-Waner Duality}

We first define the bicategory $\Ex$ of ex-spaces.   The 0-cells of
$\Ex$ are spaces. A 1-cell in $\Ex(A,B)$ is an ex-space $X$ over
$B\times A$, a space $X$ with maps \[B\times A
\stackrel{\sect}{\rightarrow} X\stackrel{\pro}{\rightarrow}B\times A\]
such that $\pro\circ \sect=\mathrm{id}$. The 2-cells of $\Ex$ are maps of
total spaces that commute with the section and projection.
If $Y$ is an ex-space over $B$ we think of it as an object
of $\Ex(B,\ast)$.

Recall from \myref{exspaceconditions} that for homotopical control
we will usually consider parametrized spaces $X$ over
$A\times B$ where the map $X\rightarrow A\times B$ is a fibration
and the map $A\times B\rightarrow X$ is a
fiberwise cofibration.  
This assumption implies that maps
in the homotopy category will correspond to fiberwise homotopy classes of maps.
While this is a restrictive assumption, in many of the examples we are
interested in this condition is satisfied.  When this does not hold we choose an 
equivalent replacement that does satisfy these conditions.
See \cite[9.1.2]{MS} for further details.

The external smash product $\bar{\wedge}$ of an ex-space
$X$ over $A$ with an ex-space $Y$ over $B$ is a
parametrized space over $A\times B$. The fiber of the external
smash product over $(a,b)$ is the fiber of $X$ over $a$ smashed
with the fiber of $Y$ over $b$.

If $ {X}$ is a parametrized space over $ {A}\times  B$ and $
Y$ is a parametrized space over $ B\times  C$ then we
define $ {X}\boxtimes  Y$,\nidx{$\boxtimes$} a parametrized space over $
{A}\times C$, as the pullback along $\triangle\colon 
B\rightarrow
 B\times  B$ and then push forward along
$r\colon  B \rightarrow \ast$ of $ {X}\bar{\wedge} Y$.
\[\xymatrix{ {A}\times C\ar[d]& {A}\times  B
\times  C\ar[l]_-{\id\times r\times \id}\ar[r]^-{\id\times
\triangle\times \id}\ar[d]& {A}\times  B\times  B
\times  C\ar[d]\\
 {X}\boxtimes Y\ar[d]&(\id\times \triangle\times \id)^*( {X}
\bar{\wedge} Y )\ar[r]\ar[l]\ar[d] & {X}\bar{\wedge}
 Y\ar[d]\\
 {A}\times C& {A}\times  B
\times  C\ar[l]^-{\id\times r\times \id}\ar[r]_-{\id\times
\triangle\times \id} & {A}\times  B\times  B \times
 C}\]
Following \cite{MS}, we write this as $X\boxtimes Y=r_!\triangle^*(X\bar{
\wedge}Y)$ where $(-)^*$\nidx{$(-)^*$} indicates pull back and 
$(-)_!$\nidx{$(-)_{"!}$} indicates
push forward.
This is the bicategory composition in $\Ex$.  For more details on these
definitions see Chapter 17 of \cite{MS}.

With the assumption that the projection maps are fibrations 
and the sections are fiberwise cofibrations
$X\boxtimes Y$ will have the correct homotopy type. 

For each 0-cell $B$, 
$(B,\triangle)_+\in \Ex(B,B)$\nidx{B@$(B,\triangle)_+$} 
denotes the ex-space with projection map the
diagonal map $\triangle\colon B\rightarrow B\times B$ and a
disjoint section.   This is the unit for $\boxtimes$ and so it will be denoted $U_B$.

More generally, if $p\colon X\rightarrow B$ is a continuous map, then
$(X,p)_+$\nidx{X@$(X,p)_+$} is the parametrized space with projection $p$ and a
disjoint section.  We regard $(X,p)_+$ as an object of $\Ex(B,
\ast)$.

The bicategory $\Ex$ has a particularly simple involution.
The bijection $t$ on 0-cells is the identity and the pullback along
the interchange map
can be used to define an equivalence of categories
\[\Ex(A,B)\rightarrow \Ex(B,A).\]  For a map $p\colon X\rightarrow B$,  
$t(X,p)_+$ is the same space as $(X,p)_+$, but is thought of 
an object of $\Ex(\ast,B)$.

Costenoble-Waner duality \cite[Chapter 18]{MS} for parametrized
spaces is an example of duality in a more sophisticated stable version of
the bicategory $\Ex$ but it  also has an interpretations
in terms of $n$-duality in $\Ex$.  

\begin{definition}\cite[18.3.1]{MS}
An ex-space $X$ over $B$ is \emph{Costenoble-Waner  n-dualizable}
\idx{Costenoble-Waner duality} 
if there is an ex-space $Y$ over $B$ and maps 
\[\xymatrix{S^n\ar[r]^-\eta&X\boxtimes tY&{\mathrm{and}}&tY\boxtimes X\ar[r]^-\epsilon
&\triangle_!S^n_B}\]
such that \[\xymatrix{S^n\boxtimes X\ar[r]^-{\eta\boxtimes \id}
\ar[dd]_{\gamma}&(X\boxtimes tY)
\boxtimes X\ar[d]^{\cong} &
tY\boxtimes S^n\ar[r]^-{\id\boxtimes \eta}\ar[dd]_{(\sigma\boxtimes \id)
\gamma}&tY\boxtimes(X\boxtimes tY)
\ar[d]^{\cong}\\
&X\boxtimes (tY\boxtimes X)\ar[d]^-{\id \boxtimes
\epsilon}&
&((tY\boxtimes X)\boxtimes  tY)\ar[d]^-{ \epsilon \boxtimes  \id}\\
X\boxtimes S^n\ar[r]_-{\cong}& X\boxtimes 
\triangle_!S^n_B
&S^n\boxtimes tY\ar[r]_-{\cong}& \triangle_!S^n_B\boxtimes
tY}\] commute up to fiberwise homotopy.
\end{definition}

Note that $\triangle_!(S_B^n)\in \Ex(B,B)$ is the pushforward of $S^n\times 
B$ along the diagonal map of $B$. 

\begin{theorem}\cite[18.6.1]{MS}\mylabel{dualM} Let $M$ be a closed
smooth manifold with an embedding in $\bR^n$.  
Then  $(S^0_M,tS^{\nu} )$ is a Costenoble-Waner 
$n$-dual pair.\end{theorem}

The ex-space $S^0_M\in \Ex(M,\ast)$\nidx{SM@$S^0_M$} 
has total space two copies of $M$.  
The projection
map is the identity map on both components.  The ex-space
$S^\nu\in \Ex(M,\ast)$\nidx{snu@$S^\nu$} is the 
fiberwise one point compactification of the normal bundle of $M$.  The 
section is the inclusion of $M$ as the points added by the compactification. 
It is a space over $M$ via the projection 
map $\rho\colon S^{\nu}\rightarrow M$.

The coevaluation map
\[\eta\colon S^n\rightarrow  S^0_M\boxtimes tS^{\nu} \cong T\nu\] is
the Pontryagin-Thom map for the normal bundle of the
embedding  $M\rightarrow
S^n$.

The diagonal gives an inclusion of $M$ into $\nu \times M$.
Let $e$ be an identification of 
a neighborhood $V$ of $M$ in $\nu\times M$
with the trivial bundle $\mathbb{R}^n\times M$.
We can define a map
\[E\colon V\rightarrow \mathrm{Map}(I,M)\times(\mathbb{R}^n\times M)\] by
\[E(v,m)=(H(\rho(v),m),e(v,m))\] where $H(\rho(v),m)$ is a path from 
$\rho(v)$ to $m$ as in \myref{localcont}. 

The evaluation  map $\epsilon$ is the composite of the
Pontryagin-Thom map for the embedding $M\rightarrow \nu\times
M$ with the map $E$.
This is related to the evaluation map described for the dual pair in  
\myref{univcovdual1}.

Since Costenoble-Waner dual pairs are examples of dual pairs in 
a bicategory there are other characterizations of Costenoble-Waner
duals.  Let $\{-,-\}$ denote stable homotopy classes of maps and $\{-,-\}_B$
\nidx{$\{-,-\}_B$}
denote fiberwise stable homotopy classes of maps over $B$.

\begin{corollary}\mylabel{CWdualmaps}
If $X$ is Costenoble-Waner $n$-dualizable with dual $Y$,
\[\{Z\boxtimes X, W\}_B\cong \{S^n\wedge Z, W\boxtimes tY \}\]
for $Z\in \Ex(\ast,\ast)$ and $W\in \Ex(B,\ast)$,
\end{corollary}

In particular, for a closed smooth manifold $M$
\[\{S^0_M, U\}_M\cong \{S^0, U\boxtimes tS^\nu\}\]
for $U\in \Ex(M,\ast)$.

For any space $B$ the parametrized spaces $(B,\id)_+\in
\Ex(B,\ast)$ and $t(B,\id)_+\in \Ex(\ast,B)$ form a dual pair.  The
coevaluation is the diagonal map \[\triangle\colon B\rightarrow
B\times B.\]  If $r\colon B\rightarrow \ast$ is the map to a point, the
evaluation map is \[r_+\colon B_+\rightarrow S^0.\]  For a manifold $M$, 
the dual pairs $((M,\id)_+,t(M,\id)_+)$  and $(S_M^0, tS^\nu)$ 
can be composed to give the dual pair $(M_+,T\nu)$ described in
\myref{topexdual}.

\section{A bicategory of bimodules over parametrized monoids}
\label{randualpara}
In this section we define the bicategory that describes
the topological dual pairs we will use later.  This is a
bicategory of monoids, bimodules, and maps of bimodules and its
construction is similar to the construction of the bicategory of
rings, bimodules, and homomorphisms from the category of abelian
groups.

The bicategory in this chapter is a special case of the bicategory
described in Section \ref{catex4} with one exception.  In Section
\ref{catex4} we make frequent use of colimits.  In this section we
will use homotopy colimits. Here we are primarily interested in
homotopical information and homotopy colimits will give the right
homotopy types.  Also, we must use homotopy colimits to be able to
connect our invariants with classical invariants, especially the 
invariant defined by Klein and Williams.

\begin{definition} A \emph{monoid}\idx{monoid} in $\Ex$ is a parametrized space
$\sA\in \Ex(A,A)$ with parametrized maps
\[\xymatrix{\mu\colon  {\sA\boxtimes
\sA}\ar[r]&{ \sA}&{\mathrm{ and}}&\iota\colon U_A\ar[r]&{ \sA}}\]
such that
\[\xymatrix{{ \sA\cong U_A\boxtimes  \sA}
\ar[r]^-{\iota\boxtimes \id}& {\sA\boxtimes \sA}\ar[r]^-\mu&
{\sA}}\] and
\[\xymatrix{ {\sA\cong  \sA\boxtimes U_A}\ar[r]^-{\id
\boxtimes \iota}& {\sA\boxtimes \sA}\ar[r]^-\mu& {\sA}}\] are the
identity and
\[\xymatrix{{ \sA\boxtimes \sA\boxtimes \sA}\ar[r]^-{\mu\boxtimes
\id}\ar[d]_{\id\boxtimes \mu}& {\sA\boxtimes \sA}\ar[d]^\mu\\
 {\sA\boxtimes \sA}\ar[r]_-\mu&{ \sA}}\] commutes.
\end{definition}

We think of $\mu$ as composition and $\iota$ as the unit.

\begin{definition} Let $\sA$ and $\sB$ be two monoids in $\Ex$.   An 
$\sA$-$\sB$-\emph{bimodule}\idx{bimodule} 
is an object $ \sX\in \Ex( B, A)$ and two
parametrized maps \[\kappa\colon  \sA\boxtimes  \sX\rightarrow  \sX\]
and \[\kappa'\colon  \sX\boxtimes \sB\rightarrow  \sX\] that are unital
and associative with respect to the monoid structure of $\sA$ and
$\sB$. We also require that the actions $\kappa$ and $\kappa'$
commute.
\end{definition}

A monoid $\sA$ defines an $\sA$-$\sA$ bimodule with left and
right actions given by the monoid multiplication $\mu$. We will denote
this bimodule by $U_{\sA}$ since it is the unit in a bicategory.

A parametrized space $\sX$ over $A$ is trivially a bimodule. 
Thought of as a space over $\ast\times A$, $\sX$ has a left
action by $U_\ast$ using the obvious isomorphism. It also
has a right action by $U_A$ using the  unit
isomorphism \[\sX\boxtimes U_A \rightarrow  \sX.\]

\begin{definition} Let $\sX$ and $\sY$ be $\sA$-$\sB$-bimodules.
A \emph{map of bimodules}\idx{map of bimodules} 
is a parametrized map $f\colon  \sX\rightarrow
\sY$ such that
\[\xymatrix{{ \sA\boxtimes \sX}\ar[r]^-\kappa\ar[d]_{\id\boxtimes f}
& {\sX}\ar[d]^f &{\mathrm{and}}& {\sX\boxtimes
 \sB}\ar[r]^-{\kappa'}\ar[d]_{f\boxtimes \id}
& {\sX}\ar[d]^f\\
 {\sA\boxtimes  \sY}\ar[r]_-{\kappa}& {\sY}&&
 {\sY\boxtimes \sB}\ar[r]_-{\kappa'}& {\sY}}\]
commute.
\end{definition}

\begin{definition}\mylabel{exmonoidsodot}Let $ \sX$ be an $\sA$-$\sB$-bimodule 
and $ \sY$ a $\sB$-$\sC$-bimodule. Then $\sX\odot
\sY$\nidx{$\odot$} is  the bar resolution $B(\sX,\sB,\sY)$. This is 
an $\sA$-$\sC$-bimodule.
\end{definition}

The bar resolution $B(\sX,\sB,\sY)$\nidx{B@$B(-,\protect\sB,-)$} 
is the geometric realization
of the simplicial ex-space over $C\times A$ 
with $n$ simplices \[\sX\boxtimes (\sB)^n\boxtimes
\sY,\] face maps \[\partial_0=\kappa'\boxtimes \id_{\sB}^{n-1}\boxtimes
\id_\sY\]
\[\partial_i=\id_\sX\boxtimes \id_{\sB}^{i-1}\boxtimes \mu\boxtimes \id_{\sB}^{n-i-1}
\boxtimes \id_\sY\,\,\mathrm{for}\,\,0<i<n\]
\[\partial_n=\id_\sX\boxtimes \id_{\sB}^{n-1}\boxtimes \kappa\]
and degeneracy maps
\[s_i=\id_\sX\boxtimes \id_{\sB}^{i}\boxtimes \iota\boxtimes \id_{\sB}^{n-i}
\boxtimes \id_\sY.\]

We think of $\sX\odot \sY$ as the homotopy coequalizer
\[\xymatrix{ {\sX\boxtimes \sB\boxtimes \sY }\ar@<.5ex>[r]^-{\kappa'\boxtimes
\id}\ar@<-.5ex>[r]_-{\id\boxtimes \kappa}& {\sX\boxtimes
\sY}\ar[r]& {\sX \odot \sY}}\] as an ex-space.  

The bar resolution is associative up to isomorphism.  To see this recall
that the geometric realization is a tensor product of functors, 
see \cite{shulmanholim}.  Then the comparison of $B(\sX,\sB, B(\sY,\sC,\sZ))$
with $B(B(\sX,\sB,\sY),\sC,\sZ)$ is a comparison of coequalizers.  
The product $\sX\odot \sB$ is homotopy equivalent to $\sX$ using a simplicial
homotopy and the extra degeneracy in $\sB$.

This defines a bicategory $\sM_{\Ex}$\nidx{mex@$\protect{\sM}_{\Ex}$} 
with 0-cells monoids, 1-cells
bimodules, and 2-cells homotopy classes of maps of bimodules.  The
$\odot$ of \myref{exmonoidsodot} is the bicategory
composition.  The unit associated to a monoid $\sA$ is that monoid
regarded as a $\sA$-$\sA$-bimodule.  
The involution on this
bicategory is very similar to the involution on the bicategory of
rings, bimodules and homomorphisms.

\section{Ranicki duality for parametrized bimodules}\label{randualpara2}
Since the bicategory $\sM_{\Ex}$ is defined using spaces instead of
spectra, the definition of duality has to be modified a little
from the definition of duality in a bicategory. We imitate
the definition of $n$-duality for parametrized spaces.

If $ \sX $ is a $ \sA$-$\sB$-bimodule then $ \sX \bar{\wedge}S^n$
and $S^n\bar{\wedge}\sX $ are also $\sA$-$\sB$-bimodules.

\begin{definition} Let $ \sX $ be an $\sA$-$ \sB$-bimodule.
Then $ \sX $ is $n$-\emph{dualizable}\idx{n-dual} if there is a $\sB$-$\sA
$-bimodule  $ \sY$ and maps of bimodules
\[\xymatrix{\eta\colon S^n\bar{\wedge}  U_{\sA}\ar[r]&{\sX \odot \sY}&
{\mathrm{and}}&\epsilon\colon  \sY\odot \sX \ar[r]& S^n\bar{\wedge}
U_\sB}\] such that the following diagrams commute stably up to
parametrized homotopy respecting the module structure.
\[\xymatrix{{S^n\bar{\wedge}  \sX} \ar[r]^-\cong \ar[d]_{\gamma}
&{(S^n\bar{\wedge}U_\sA ) \odot  \sX} \ar[r]^{\eta\odot \id}& {\sX
\odot \sY\odot  \sX} \ar[d]^{\id\odot \epsilon}\\ {\sX
\bar{\wedge} S^n}&&{ \sX \odot (S^n\bar{\wedge} U_\sB)}
\ar[ll]^-\cong}\]

\[\xymatrix{ {\sY\bar{\wedge} S^n}\ar[r]^-\cong
\ar[d]_{(\sigma\bar{\wedge}\id)\gamma} & {\sY\odot(S^n\bar{\wedge}
 U_\sA )}\ar[r]^{\id\odot \eta}& {\sY \odot
 \sX \odot  \sY}\ar[d]^{\epsilon\odot \id}\\
{ S^n\bar{\wedge} \sY}&&{(S^n\bar{\wedge}  U_\sB) \odot  \sY}\ar[ll]^-\cong.}\]
As before, $\sigma$ is a map of degree $(-1)^n$.
\end{definition}

By neglect of structure any $\sA$-$\sB$-bimodule $\sX$ defines an
$\sA$-$U_{B}$-bimodule denoted $L(\sX)$\nidx{L@$L$} and a $U_A$-$\sB$-bimodule
denoted $R(\sX)$\nidx{r@$R$}.

\begin{lemma}\mylabel{topmonodidual} 
Let $\sA$ be a monoid.  Then $(R(U_\sA),L(U_\sA))$ is a
dual pair.  

The coevaluation map \[U_A\rightarrow R(U_\sA)\odot
L(U_\sA)\] is the unit map. The evaluation map \[L(U_\sA)\odot
R(U_\sA)\rightarrow U_\sA\] is the monoid multiplication.
\end{lemma}

The simplicial ex-space used to define $R(U_\sA)\odot L(U_\sA)$ has an `extra'
degeneracy given by regarding an element of $R(U_\sA)$ or
$L(U_\sA)$ as an element of $\sA$.  This means that
$R(U_\sA)\odot L(U_\sA)$ is equivalent to $N(U_\sA)$,\nidx{N@$N$} the monoid
$\sA$ regarded as an $U_A$-$U_A$-bimodule.

Costenoble-Waner duality 
is an example of duality in the bicategory $\sM_{\Ex}$. 
The
Costenoble-Waner dual of an ex-space $X$ over $B$ is the dual of $X$ 
as a $U_\ast$-$U_B$-bimodule in $\sM_{\Ex}$.

Costenoble-Waner duality only uses monoids defined using 
unit isomorphisms.  The Ranicki dual pair described in Section \ref{summary1}
requires less trivial monoids. With the discrete topology,
$(\pi_1M)_+\in \Ex(\ast,\ast)$ is a monoid. The monoid multiplication
\[(\pi_1M)_+\boxtimes (\pi_1M)_+
=(\pi_1M\times\pi_1M )_+\rightarrow (\pi_1M)_+\] is the group
multiplication and  the unit $\iota\colon S^0\rightarrow (\pi_1M)_+$ is
the inclusion of the identity element of $\pi_1M$.  The universal cover
$\tilde{M}_+\in \Ex(\ast,\ast)$ is a right $(\pi_1M)_+$
module with the right action given by the usual action of $\pi_1M$
on $\tilde{M}$.  With this interpretation, a Ranicki dual for
$\tilde{M}_+$ is a dual for the module $\tilde{M}_+$ in the
bicategory $\sM_{\Ex}$. 

We can use the quotient map to regard $\tilde{M}$ as a space over $M$.
In contrast with the convention above, we choose to regard $(\tilde{M},\pi)_+$
as an element of $\Ex(\ast,M)$.  This choice 
is consistent with our later convention for path spaces since we think of 
$\tilde{M}$ as the homotopy classes of paths in $M$ that start at a
base point.
Then $(\tilde{M},\pi)_+$ is a right $(\pi_1M)_+$-module.
Note that
$\tilde{M}_+$ is equivalent to $S^0_M\odot (\tilde{M},\pi)_+$.
Let $\,_*\tilde{M}$ denote the universal cover of $M$ regarded as homotopy 
classes of paths ending
at the base point rather than starting at the base point.  
Then $(\,_*\tilde{M},\pi)_+$ is a left $\pi_1M$-module.

\begin{lemma}\mylabel{bunivdual}
For a closed smooth  manifold
$M$ we have the following dual pairs.
\begin{enumerate}\item $((\tilde{M},\pi)_+,(\,_*\tilde{M},\pi)_+)$.
\item $(\tilde{M}_+,T\pi^*S^{\nu})$
\end{enumerate}\end{lemma}

\begin{proof}
As before, $T\pi^*S^\nu$ is the pushout of the maps
$\tilde{M}\rightarrow \pi^*tS^\nu$ and $\tilde{M}\rightarrow *$.
This is equivalent to $(\,_*\tilde{M},\pi)_+ \odot S^\nu$.  The ex-space
\[(\tilde{M},\pi)_+\odot (\,_*\tilde{M},\pi)_+\] is equivalent to the 
coequalizer of the two actions of $\pi_1M$ on 
\[(\tilde{M}\times\tilde{M},\pi\times\pi)_+\] since $\pi_1M$
acts freely on $\tilde{M}$.  Similarly, \[\tilde{M}_+\odot
T\pi^*S^\nu\] is equivalent to the coequalizer of the 
two actions of $\pi_1M$ on \[ \tilde{M}_+\wedge T\pi^*S^\nu.\]

The coevaluation map for $((\tilde{M},\pi)_+,(\,_*\tilde{M},\pi)_+)$
\[U_M\rightarrow (\tilde{M},\pi)_+\odot (\,_*\tilde{M},\pi)_+
\simeq (\tilde{M}\times_{\pi_1M}\tilde{M},\pi\times \pi)_+\] is given by
$m\mapsto (\gamma_m,\gamma_m)$ for  any lift $\gamma_m$ of $m$.
This map is well defined since the action of $\pi_1M$ identifies
all possible choices. The evaluation map \[(\,_*\tilde{M},\pi)_+\odot
(\tilde{M},\pi)_+=
\{(\alpha,\beta)\in\tilde{M}\times\tilde{M}|\pi(\alpha)=\pi(\beta)\}_+
\rightarrow U_{\pi_1M}\] is given by $(\alpha,\beta)\mapsto
\alpha\beta$.

Note that the required diagrams for this dual pair commute
strictly and without needing to stabilize.  This is closely
related to the  observation above that monoids produce dual pairs.

The second dual pair is the composite of the dual pairs
$((\tilde{M},\pi)_+,(\,_*\tilde{M},\pi)_+)$ and $(S^0_M,tS^{\nu})$. The
coevaluation map
\[\xymatrix{S^n\ar[r]^-{\eta}&T\nu\ar[r]&\tilde{M}_+\wedge_{\pi_1M}T\pi^*
S^{\nu}\simeq \tilde{M}_+\odot  T\pi^*S^{\nu}}\]
is the composite of the
coevaluation, $\eta$,  for the dual pair $(S^0_M,tS^{\nu})$ and the
map \[v\mapsto ( \gamma_{p(v)}, \gamma_{p(v)},v).\] The
evaluation map
\[\xymatrix{T\pi^*S^{\nu}\odot \tilde{M}_+\simeq T\pi^*S^{\nu}
\wedge\tilde{M}_+ \ar[r]&(\pi_1M)_+ \wedge
S^n=\vee_{\pi_1M}S^n}\] is given by $(\gamma,v,\alpha)\mapsto
( \gamma H(\alpha(1),\gamma(0))\alpha, \epsilon(v,\alpha(1)))$ where $\epsilon$
is the evaluation map for the dual pair $(S^0_M,tS^{\nu})$. The path
$H(\alpha(1),\gamma(0))$ is defined in  \myref{localcont}.
\end{proof}

This lemma completes the proof of \myref{univcovdual1}.

\section{Moore loops and bicategories}\label{mlab}
In the previous section we defined duality in the bicategory of
monoids and bimodules in parametrized spaces and gave examples of
dual pairs.  With the exception of the first dual pair in 
\myref{bunivdual}, the dual pair for the universal cover of a
manifold regarded as a space over that manifold, we haven't used
the flexibility the bicategory $\Ex$ offers.  In this section we
will begin to exploit this greater flexibility.

One undesirable aspect of using Ranicki duality to describe
duality for universal covers was the need to choose a base point. 
There are two ways of dealing
with this problem.  The first is to verify that different choices
of the base point give ``the same'' dual pairs.  The second 
is to use all possible choices of base point. In other words, use
objects like the fundamental groupoid rather than the fundamental
group. The first
approach is used in \cite{brown, jiang} and the second in \cite{coufal}. We
will use the second approach here since it will also be useful
when defining fiberwise dual pairs.

For our topological applications, the fundamental groupoid is not
exactly the right object.  First, we would rather have a space of
objects and a space of morphisms rather than sets.  Second, for
fiberwise applications we would rather consider all paths than
homotopy classes of paths.  Instead of the fundamental groupoid we
will consider a topologized version of the fundamental groupoid
and the space of Moore paths.

Let $M$ be a compact manifold and $\Pi M$\nidx{pim@$\Pi M$}  
the space of homotopy
classes of paths in $M$ with endpoints fixed.  Topologize this
space using the quotient topology from the usual compact open
topology on $\mathrm{Map}(I,M)$.  
There is a Hurewicz fibration $\tar\times
\sou\colon \Pi M\rightarrow M\times M$\nidx{t@$\tar$}\nidx{s@$\sou$} 
given by $\sou(\gamma)=\gamma(0)$ and
$\tar(\gamma)=\gamma(1)$. 
The fiber product
\[\Pi M\times_M\Pi M=\{(\gamma_2,\gamma_1)\in \Pi M\times
\Pi M|\gamma_1(1)=\gamma_2(0)\} \] is a space with a map to
$M\times M$ given by $(\gamma_2,\gamma_1)\mapsto
(\gamma_2(1),\gamma_1(0))$.  Composition gives a strictly
associative map $\mu\colon  \Pi M\times_M\Pi M\rightarrow \Pi M$.
This is a map over $M\times M$.

The inclusion $\iota$\nidx{iota@$\iota$} of $M$ into $\Pi M$ by constant paths is
also a map over $M\times M$ if $M$ is regarded as a space over
$M\times M$ using the diagonal map\nidx{$\triangle$}  
\[\triangle\colon M\rightarrow M\times M.\] 
We regard $\Pi M\times_MM$ as a space over $M\times M$ by the
map $(\gamma, m)\mapsto (\gamma(1),m)$.  Then $\Pi M\times_MM$ is
homeomorphic to $\Pi M$ as a space over $M\times M$.  The map
$\iota$ acts as the unit for $\mu$ in the sense that the
following maps are the identity
\[\Pi M\cong \Pi M\times_MM\rightarrow \Pi M\times\Pi M\rightarrow
\Pi M\]
\[\Pi M\cong M\times_M\Pi M\rightarrow \Pi M\times \Pi M
\rightarrow \Pi M.\]
With a disjoint section, $\Pi M$ is a monoid in
$\Ex$.

In contrast with the previous sections, we will 
not use a different notation for a path space monoid and that monoid regarded 
as a bimodule.  

Recall that for an $\sA$-$\sB$-bimodule $\sX$, $R(\sX)$
is $\sX$ regarded
as a $U_A$-$\sB$-bimodule, $L(\sX)$ is $\sX$ regarded as a $\sA$-$U_B$-bimodule,
and $N(\sX)$ is $\sX$ regarded as a $U_A$-$U_B$-bimodule.

The parametrized space  $(\Pi M,\sou)_+$\nidx{pim@$(\Pi M,\sou)_+$}
has a right
action of $(\Pi M,\tar\times \sou)_+$ by composition of paths.
Recall that $S^\nu$ is the fiberwise one point compactification of
the normal bundle of $M$.  Then $T_M\sou^*S^\nu$\nidx{tmssnu@$T_M\sou^*S^\nu$} 
is defined to be 
$L(\Pi M,\tar\times \sou)_+\odot tS^\nu$.  This is an ex-space over
$M$, and it has a left action by $\Pi M$.

The dual pairs in \myref{bunivdual2} are the unbased versions of the dual
pairs in \myref{bunivdual}. We make
this comparison explicit in \myref{topdualcompare}.

\begin{lemma}\mylabel{bunivdual2}
For a smooth compact  manifold $M$ we have the following
dual pairs.
\begin{enumerate}
\item $(R(\Pi M,\tar\times \sou)_+,L(\Pi M,\tar\times \sou)_+)$ \item
$((\Pi M,\sou)_+,T_M\sou^*S^{\nu})$
\end{enumerate}
\end{lemma}

\begin{proof}
As noted before $R(\Pi M,\tar\times \sou)_+\odot L(\Pi M,\tar\times
\sou)_+$ is equivalent to  $N(\Pi M,\tar\times \sou)_+$.  

The first dual pair is a dual pair arising from a
monoid and so this dual pair follows from \myref{topmonodidual}.   
The
second dual pair is the composite of \[(R(\Pi M,\tar\times
\sou)_+,L(\Pi M,\tar\times \sou)_+)\] with the dual pair
$(S^0_M,tS^{\nu})$.

\vspace{5pt}
In $(i)$, the coevaluation map
\[U_M\rightarrow R(\Pi M,\tar\times \sou)_+\odot L(\Pi M,
\tar\times \sou)_+\] is given by $m\mapsto(c_m,c_m)$ where $c_m$ is the
constant path at $m$. The evaluation map
\[L(\Pi M,\tar\times \sou)_+\odot R(\Pi M,\tar\times \sou)_+
\rightarrow (\Pi M,
\tar\times \sou)_+\] is given by $(\alpha,\beta)\mapsto
\alpha\beta$.

\vspace{5pt} In $(ii)$,  note that \[((\Pi M,\sou)_+,T_M\sou^*S^{\nu})
\simeq(S^0_M\odot  R(\Pi M,\tar\times \sou)_+,  L(\Pi M,\tar\times \sou)_+\odot
tS^{\nu}).\]  The coevaluation map
\[S^n\rightarrow
T\nu\rightarrow (\Pi M,\sou)_+\odot T_M\sou^*S^{\nu}\] is given by
$v\mapsto ( c_{\rho(\eta(v))}, c_{\rho(\eta(v))},\eta(v))$, where $\eta$ is
the Pontryagin-Thom map for the normal bundle of $M$. The
evaluation map
\[T_M\sou^*S^{\nu}\odot (\Pi M,\sou)_+ \rightarrow S^n\bar{\wedge}
L(\Pi M,\tar\times \sou)_+\odot R(\Pi M,\tar\times \sou)_+\rightarrow
S^n\bar{\wedge} (\Pi M,\tar\times \sou)_+\] is given by $(\alpha,
v,\beta)\mapsto ( \epsilon(v,\beta(1)),\alpha H(\beta(1),\alpha(0))
\beta)$ where
$H(\beta(1),\alpha(0))$ is as in \myref{localcont} and $\epsilon$ is
the evaluation map for $(S^0_M,tS^{\nu})$.

\end{proof}

\begin{lemma}\mylabel{topdualcompare} $\tilde{M}_+$ is dualizable as a
$\pi_1M$-space if and only if
$(\Pi M,\sou)_+$ is dualizable as a $(\Pi M,\tar\times \sou)_+$-module.
\end{lemma}

\begin{proof} This result follows from Theorems \ref{dualcomposites1}
and \ref{dualcomposites2}. Let $x$ be a point in $M$ and $(\Pi M_x,\tar)_+$
\nidx{pimx@$\Pi M_x$}
be the universal cover of $M$ thought of as homotopy classes of paths in
$M$ that start at $x$.  This has a right action of $\pi_1(M,x)$ and a
left action of $(\Pi M,\tar\times \sou)_+$.  This space
 is dualizable with dual $(\,_x\Pi M,\sou)_+$, \nidx{pimx@$_x\Pi M $}the
universal cover thought of as homotopy classes of paths in $M$ ending
at $x$ with a right action by  $(\Pi M,\tar\times
\sou)_+$ and a left action by $\pi_1(M,x)$.  This dual pair satisfies 
the additional
condition that the evaluation map is an isomorphism.

Then $(\Pi M,\sou)_+\odot (\Pi M_x,\tar)_+$ is equivalent to
$\tilde{M}_+$ regarded as a right $\pi_1(M,x)$-space.
By Theorems \ref{dualcomposites1} and
\ref{dualcomposites2},
 $\tilde{M}_+$ is dualizable as a $\pi_1M$-space if and only if
$(\Pi M,\sou)_+$ is dualizable as a $(\Pi M,\tar\times \sou)_+$-space.
\end{proof}

Recall that $\calP M=\{(\gamma,u)\in M^{[0,\infty)}\times [0,\infty)|
\gamma(t)=\gamma(u)\,\mathrm{for}\, t\geq u\}$\nidx{pm@$\protect\calP M$} 
is the space of free Moore paths in $M$.  
With a disjoint section, $\calP M$ is a monoid over $M\times M$ in $\Ex$.
The parametrized space $(\calP M,\sou)_+$
has a right action of $(\calP M, \tar\times \sou)_+$
by composition of paths.

\begin{lemma}\mylabel{bunivdual3}
For a compact smooth manifold $M$ we have the following
dual pairs.
\begin{enumerate}
\item $(R(\calP M,\tar\times \sou)_+, L(\calP M,\tar\times \sou)_+)$ 
\item $((\calP M,\sou)_+,T_M\sou^*S^{\nu})$
\end{enumerate}
\end{lemma}

\begin{proof}
$T_M\sou^*S^\nu$\nidx{tmssnu@$T_M\sou^*S^\nu$} 
is defined to be $L(\calP M,\tar\times \sou)_+\odot S^\nu$.

The first dual pair is a dual pair arising from a
monoid as in \myref{topmonodidual}.  
The second dual pair is the composite of $(R(\calP
M,\tar\times \sou)_+, L(\calP M,\tar\times \sou)_+)$ with the dual pair
$(S^0_M,tS^{\nu})$.

\vspace{5pt} In $(i)$, the coevaluation map
\[U_M\rightarrow R(\calP M,\tar\times \sou)_+
\odot L(\calP M,\tar\times \sou)_+
\simeq(\calP M,\tar\times \sou)_+\] is the inclusion of $M$ into $\calP
M$ as constant paths.  The evaluation map \[(\calP
M\times_M\calP M,\tar\times \sou)_+ \simeq L(\calP M,\tar\times \sou)_+
\odot
R(\calP M,\tar\times \sou)_+ \rightarrow (\calP M,\tar\times \sou)_+\] is
given by composition of paths.

\vspace{5pt} In $(ii)$,  note that \[(S^0_M\odot R(\calP
M,\tar\times \sou)_+, L(\calP M,\tar\times \sou)_+\odot S^{\nu})
\simeq ((\calP
M,\sou)_+,T_M\sou^*S^{\nu}).\]
  The coevaluation map
\[S^n\rightarrow  S^0_M\boxtimes S^{\nu}\simeq  T{\nu}\rightarrow
(\calP M,\sou)_+\odot T_M\sou^*S^{\nu}\] is given by  $v\mapsto
(c_{\rho\eta (v)}, c_{\rho\eta(v)},\eta (v) )$ where $\eta$ is the
Pontryagin-Thom for the embedding of $M$ in $\bR^n$.  The evaluation map
\[\xymatrix{T_M\sou^*S^{\nu}\odot (\calP M,\sou)_+\ar[r]& L(\calP
M,\tar\times \sou)_+\odot ((M,
\triangle)_+\bar{\wedge} S^n) \odot R(\calP M,\tar\times \sou)_+\ar[d]\\&
S^n\bar{\wedge} (\calP M,\tar\times \sou)_+}\] is given by  $(\alpha, v,
\beta)\mapsto(\epsilon (v,\beta(1)),\alpha H(\beta(1),\alpha(0))\beta)$
where $H$ is as in  \myref{localcont}.
\end{proof}

This lemma completes the proof of \myref{bunivdual3p}.

\myref{bunivdual3} is very similar to \myref{bunivdual2}.  In both 
lemmas a dual pair defined using a monoid is composed with the 
dual pair $(S^0_M,S\nu)$.  Let $c:\calP M\rightarrow \Pi M$ be the 
map that takes a path to its homotopy class with end points 
fixed.  This map induces a map \[N(\calP M,\tar\times \sou)_+)\rightarrow 
N(\Pi M, \tar\times \sou)_+\] and similarly for the corresponding
left and right modules.  Functoriality implies 
that the following diagrams commute
\[\xymatrix{U_M\ar[r]\ar[dr]&N(\calP M, \tar\times \sou)_+\ar[r]^-\sim\ar[d]& 
R(\calP M,\tar\times \sou)_+\odot L(\calP M,\tar\times \sou)_+\ar[d]\\
&N(\Pi M, \tar\times \sou)_+\ar[r]^-\sim& 
R(\Pi M,\tar\times \sou)_+\odot L(\Pi M,\tar\times \sou)_+}\]
\[\xymatrix{
L(\calP M,\tar\times \sou)_+\odot R(\calP M,\tar\times \sou)_+\ar[r]\ar[d]
&(\calP M,\tar\times \sou)_+\ar[d]\\
L(\Pi M,\tar\times \sou)_+ \odot R(\Pi M,\tar\times \sou)_+\ar[r]
&(\Pi M,\tar\times \sou)_+}\]
Composing with the dual pair $(S^0_M,S^\nu)$ gives similar diagrams 
showing compatibility of
the dual pairs $((\calP M,\sou)_+,T_M\sou^*S^{\nu})$ 
and $((\Pi M,\sou)_+,T_M\sou^*S^{\nu})$.   

\section{Shadows and traces for Ranicki dualizable bimodules}
\label{randualpara3}
The shadows in $\sM_{\Ex}$ are very similar to $\odot$ in $\sM_{\Ex}$.  
In special cases 
they are a `derived form' of the semiconjugacy classes of the fundamental 
group.   They also relate to the target of the Hattori-Stallings trace.  

\begin{definition} Let $ \sZ $ be an $\sA$-$\sA$-bimodule. 
Then $\sh\sZ$\idx{shadow}\nidx{shad@$\protect\sh{-}$} 
is the cyclic bar resolution $C(\sZ,\sA)$.

The maps $\theta$\nidx{theta@$\theta$} are induced by
\[\xymatrix{{r_!\triangle^*(\sX\boxtimes \sY)}\ar[r]\ar[d]&
{r_!\triangle^*(\sX\odot
\sY)}\ar[r]&{\sh{\sX\odot \sY}}\ar@{.>}[d]^{\theta}\\
{r_!\triangle^*(\sY\boxtimes \sX)}\ar[r]&{r_!\triangle^*(\sY\odot
\sX)}\ar[r]& {\sh{\sY\odot \sX}}}\]

\end{definition}

The target of the shadow functors is 
the category of $U_\ast$-$U_\ast$-bimodules.  This category is also 
the category of based spaces.

The cyclic bar resolution $C(\sZ,\sA)$\nidx{C@$C(-,\protect\sA)$} 
is the geometric
realization of the simplicial based space with $n$ simplices
\[r_!\triangle^*((\sA)^n\boxtimes\sZ ),\] face maps
\[\partial_0=\id_{\sA}^{n-1}\odot \kappa\]
\[\partial_i=\id_{\sA}^{n-i-1}\odot \mu\odot \id_{\sA}^{i-1}\odot \id_\sZ
\,\,\mathrm{for}\,\,0<i<n\]
\[\partial_n=(\id_{\sA}^{n-1}\odot \kappa')\gamma\]
where $\gamma\colon (\sA)^n\boxtimes\sZ \rightarrow (\sA)^{n-1}\boxtimes
\sZ\boxtimes \sA$ is the twist map, and degeneracy maps
\[s_i=\id_{\sA}^{i}\odot \iota\odot \id_{\sA}^{n-i}\odot \id_\sZ.\]

As before, we use $n$-duality to define the trace of a map.

\begin{definition} Let $ \sX $ be an $n$-dualizable $
 \sA$-$\sB$-bimodule with
dual $ \sY$, coevaluation and evaluation
\[\xymatrix{\eta\colon S^n\bar{\wedge}  U_\sA\ar[r]& {\sX \odot \sY} &{\mathrm{and}}
&\epsilon\colon  \sY\odot \sX \ar[r]&S^n\bar{\wedge} U_\sB.}\]
Suppose $ {\sQ}$ is a $ \sA$-$ \sA$-bimodule, $\sP$ is an 
$\sB$-$\sB$-bimodule 
and $f\colon  \sQ\odot \sX
\rightarrow {\sX}\odot \sP$ is a map of bimodules.  Then the
\emph{trace}\idx{trace} of $f$ is the stable homotopy class of the composite
\[\xymatrix@C=25pt{{\sh{\sQ\bar{\wedge} S^n}}
\ar[r]^-{\sh{\id \odot \eta}} &\sh{\sQ\odot \sX \odot \sY} 
\ar[d]^-{\sh{f\odot \id}}\\&{\sh{\sX \odot \sP \odot \sY} }
\ar[r]^-\cong&{\sh{\sY\odot \sX \odot \sP}}
\ar[r]^-{\sh{\epsilon\odot \id}}&{\sh{S^n\bar{\wedge} \sP}} }\]
\end{definition}

We give  examples of this trace in the next section.

\chapter{Classical fixed point theory}\label{classfpsec2}\label{classfpinv}
This chapter implements the plan described in the introduction for
proving \myref{reidemeister1}, the converse of the Lefschetz
fixed point theorem that uses the Reidemeister trace.  
We use the dual pairs from Chapter \ref{classfpsec} 
to interpret the fixed point
invariants we defined in Chapter \ref{reviewfp2} as examples of the trace in
bicategories with shadows.  
Then we use functoriality to identify the algebraic, geometric, and homotopy
Reidemeister traces. 

In Chapter \ref{classfpsec} we used the results on composites of dual
pairs to produce new dual pairs.  In this chapter we will use the 
corresponding 
results for the compatibility of composites of dual pairs and traces
to compare the based and unbased versions of different forms 
of the Reidemeister trace.
We also show how to use properties of
the trace in bicategories with shadows to recover some standard
fixed point theory results.  

\section{The geometric Reidemeister trace}

In this section we define  the
 unbased geometric Reidemeister trace using
trace in a bicategory.  For this invariant we use the topologized
fundamental groupoid.

Let $f\colon M\rightarrow M$ be an endomorphism of a compact manifold and
\nidx{pifm@$\Pi^fM$}
\[\Pi^fM=\{(\gamma,x)\in \Pi M\times M|
\gamma(0)=f(x)\}.\]  There is a  Hurewicz fibration $\tar\times
\sou\colon \Pi^fM\rightarrow M \times M$ given by $\sou(\gamma,x)=x$,
$\tar(\gamma,x)=\gamma(1)$.  This defines a $(\Pi M,\tar\times \sou)_+
$-$(\Pi M,\tar\times \sou)_+$-bimodule $(\Pi^fM,\tar\times \sou)_+$ with
the usual left action of $\Pi M$ on itself and the right action
given by first composing with $f$ and then composing paths. Then
$(\Pi M,\sou)_+\odot (\Pi^fM,\tar\times \sou)_+$ is equivalent to 
the right $(\Pi M,\tar\times
\sou)_+$-module $(\Pi^fM,\sou)_+$.   The map $f$ induces a map
of right $(\Pi M,\tar\times \sou)_+$-modules
\[f_*\colon 
(\Pi M,\sou)_+\rightarrow (\Pi M,\sou)_+\odot (\Pi^fM,\tar\times
\sou)_+.\]\nidx{f@$f_*$}

If $M$ is $n$-dualizable, the trace of $f_*$ is the stable 
homotopy class of
the  map
\[S^n\rightarrow \sh{S^n\bar{\wedge} \,(\Pi^fM, \tar\times \sou)_+}
\cong S^n\wedge \,\sh{(\Pi^fM,\tar\times \sou)_+}
\]
given by $v\mapsto (\epsilon[\eta(v),f(\rho\eta(v))], H[f\rho(\eta(v)),\rho(\eta
(v))])$ where
$H(f\rho(\eta(v)),\rho(\eta
(v)))$ is a path from $f\rho\eta(v)$ to $\rho\eta(v)$ as in \myref{localcont}.

\begin{definition} The \emph{unbased geometric Reidemeister trace}, 
\idx{unbased geometric Reidemeister trace}\idx{Reidemeister trace!unbased
geometric}$R^{U,
geo}(f)$,\nidx{rgeou@$R^{U,geo}$}  is the element 
$\tr(f_*)\in \pi_0^s(\sh{\Pi^fM,\tar\times\sou)_+})$.
\end{definition}

\begin{prop}\mylabel{geotracecompare} A choice of base point $\ast\in M$ 
determines an isomorphism
\[\sh{S^n\bar{\wedge} 
(\Pi^fM,\tar\times \sou)_+}\rightarrow \vee_{\sh{\pi_1M^{\phi}}}S^n.\]
Under this identification $R^{U,geo}(f)=R^{geo}(f)$.
\end{prop}

\begin{proof} In \myref{topdualcompare} we compared the dual
pairs \[((\Pi M,\sou)_+,T_M\sou^*S^\nu)\] and \[(\tilde{M}_+,T\pi^*S^\nu)\]
 using a third dual pair, $((\Pi M_x, \tar)_+,
(\,_x\Pi M, \sou)_+).$  The dual pair \[((\Pi M_x, \tar)_+,
(\,_x\Pi M, \sou)_+)\] has the property that the evaluation
is an isomorphism.  Then the result follows from 
\myref{tracegpd}.
\end{proof}

\begin{rmk}
The classical definition  of the geometric Reidemeister trace
described in Chapter \ref{reviewfp2} suggests that composites of dual pairs
might be relevant since the geometric Reidemeister trace is defined using
the index and information about the fundamental group.  The index
can be defined using the classical dual pair of a compact manifold
and the fundamental group information can be described using a
standard dual pair related to the fundamental groupoid regarded 
as a monoid.
\end{rmk}

\section{The homotopy Reidemeister trace}\label{classhtpysect}

The Moore paths monoid  can also be used to define a trace. Given a 
map $f\colon M\rightarrow M$, let \nidx{PfM@${\protect\calP^f M}$}
\[\calP^fM=\{(\gamma,u,x)\in \calP M\times M| \gamma(0)
=f(x)\}.\]
There are maps $\sou,\tar\colon \calP^fM\rightarrow M$ given by
$\sou(\gamma,u,x)=x$ and $\tar(\gamma,u,x)=\gamma(u)$ and these define a
$(\calP M,\tar\times \sou)_+$-$(\calP M,\tar\times \sou)_+$-bimodule
$(\calP^fM, \tar\times \sou)_+$
with the usual left action of $\calP M$ and the
right action of $\calP M$ given by first composing with $f$ and 
then composing paths.
The right $(\calP M,\tar\times \sou)_+$-module 
$(\calP M,\sou)_+\odot  (\calP^fM,\tar\times \sou)_+$ is equivalent
to the right
$(\calP M,\tar\times \sou)_+$-module $(\calP^fM, \sou)_+$. The map  $f$
induces a map of right $(\calP M,\tar\times \sou)_+$-modules
\[\tilde{f}\colon (\calP M,\sou)_+\rightarrow  (\calP M,\sou)_+\odot (\calP^fM,
\tar\times \sou)_+.\]

The shadow of $(\calP^fM, \tar\times\sou)$ is homotopy
equivalent to \nidx{lambdafm@$\Lambda^fM$} 
\[\Lambda^fM\coloneqq \{(\gamma,u)\in
\calP M|\gamma(0)=f(\gamma(u))\}\simeq\{\alpha\in M^I|\alpha(0)=
f (\alpha(1))\}.\]  
We can see this by an explicit construction of
a simplicial homotopy equivalence.  On the level of the $n$-simplices
the homotopy equivalence is given by  
\[\calP M\boxtimes\ldots \boxtimes \calP M\boxtimes \calP^fM\rightarrow
\Lambda^fM\]
\[(\alpha_1,\alpha_2,\ldots ,\alpha_n,\gamma)\mapsto
(\alpha_1\alpha_2\ldots \alpha_n\gamma)\] with inverse
\[\gamma \mapsto (c,c,\ldots ,c,f(\gamma))\] where $c$ is a constant
path.  The homotopy between
the identity map and the map \[((\alpha_1,\alpha_2,\ldots
,\alpha_n,\gamma)\mapsto (c,c, \ldots, c, f(\alpha_1\alpha_2\ldots
\alpha_n\gamma))\] is similar to the homotopy used in the
comparison of $C(\pi^\phi,\pi)$ with $\sh{\pi^\phi}$. 

The trace of $\tilde{f}$\nidx{f@$\tilde{f}$} is the stable homotopy 
class of the map
\[S^n\rightarrow  S^n\wedge \Lambda^fM_+\]
given by $v\mapsto ( \epsilon(\eta(v),f\rho\eta(v)),H(f(\rho\eta(v)),\rho\eta
(v)))$, where
$\eta$ and $\epsilon$ are the coevaluation and coevaluation from
the dual pair $(S^0_M,S^{\nu})$ and $H(f(\rho\eta(v)),\rho\eta
(v))$ is a path from
$f\rho\eta(v)$ to $\rho\eta(v)$ as in \myref{localcont}.

\begin{definition} The \emph{homotopy Reidemeister trace},\idx{homotopy
Reidemeister trace}\idx{Reidemeister trace!homotopy}
 $R^{htpy}(f)$,\nidx{rhtpy@$R^{htpy}$}
is the trace of $\tilde{f}$.
\end{definition}

\begin{prop}\mylabel{geohtpycompare} The map
\[\calP^fM\rightarrow \Pi^fM\] that takes a path to its homotopy class
with end points fixed induces an isomorphism 
\[\pi_0^s(\sh{(\calP^fM,\tar\times \sou)_+})\rightarrow
\pi_0^s(\sh{(\Pi^fM,\tar\times \sou)_+}).\]   The image of 
$R^{htpy}(f)$ under this isomorphism is $R^{U,geo}(f)$.
\end{prop}

\begin{proof}
In \myref{pihiso} we defined an isomorphism $\pi_0^s(X)\cong H_0(X)$.
The maps that define this isomorphism are all natural and so the following 
diagram commutes.
\[\xymatrix{\pi_0^s(\sh{\calP^fM}_+)\ar[d]&\pi_q(\Sigma^q 
\sh{\calP^fM}_+)\ar[r]\ar[l]\ar[d]&H_q(\Sigma^q \sh{\calP^fM}_+)\ar[d]
&H_0(\sh{\calP^fM}_+)\ar[d]\ar[l]\\
\pi_0^s(\sh{\Pi^fM}_+)&\pi_q(\Sigma^q 
\sh{\Pi^fM}_+)\ar[r]\ar[l]&H_q(\Sigma^q \sh{\Pi^fM}_+)
&H_0(\sh{\Pi^fM}_+)\ar[l]}\]
Here $q$ is chosen so that the Freudenthal suspension theorem implies 
that the first horizontal maps are isomorphisms.  The second horizontal
maps are the Hurewicz maps.  The third maps are the homology isomorphism.
The map $H_0(\sh{\calP^fM}_+)\rightarrow H_0(\sh{\Pi^fM}_+)$ is 
an isomorphism since the map $\sh{\calP^fM}\rightarrow \sh{\Pi^fM}$ 
is the map to components.  Then \[\pi_0^s(\sh{\calP^fM}_+)
\rightarrow \pi_0^s(\sh{\Pi^fM}_+)\] is also an isomorphism. 

At the end of Section \ref{mlab} we showed that the diagrams 
\[\xymatrix{S^0\ar[r]&S^0_M\odot U_M\odot S^\nu \ar[r]\ar[dr]& 
(\calP M,\sou)_+ \odot T_M\sou^*S^\nu\ar[d]
\\
&&(\Pi M,\sou)_+ \odot T_M\sou^*S^\nu}\]
and 
\[\xymatrix@C=7pt{T_M\sou^*S^\nu\odot (\calP M,\sou)_+\ar[r]\ar[d]&
S^n\odot L(\calP M,\tar\times \sou)_+
\odot R(\calP M,\tar\times \sou)_+\ar[r]\ar[d]&S^n\odot (\calP M,\tar\times \sou)_+
\ar[d]\\
T_M\sou^*S^\nu\odot (\Pi M,\sou)_+\ar[r]&S^n\odot L(\Pi M,\tar\times \sou)_+
\odot R(\Pi M,\tar\times \sou)_+\ar[r]&S^n\odot (\Pi M,\tar\times \sou)_+}\]
commute.  The diagram 
\[\xymatrix{(\calP M,\sou)_+\ar[r]\ar[d]&
(\calP^fM,\sou)_+\ar[d]\\
(\Pi M,\sou)_+\ar[r]&
(\Pi^fM,\sou)_+}\]
also commutes.  The compatibility of shadows and these diagrams
imply that the image of $R^{htpy}(f)$ under the map to 
components is $R^{geo}(f)$.
\end{proof}

\section{The Klein and Williams invariant as a trace}

To identify $R^{KW}(f)$ with $R^{htpy}(f)$ we need  to fill in some of the 
details we omitted in Section \ref{usuconverse}.  We will still not 
give a complete proof here.  All of the details can be found in 
Section \ref{fibconvers}.

Recall that $S_B$ is the unreduced fiberwise suspension.

\begin{prop}[\myref{kwclassidentify}]\mylabel{kwclassidentify2}
There is an isomorphism
\[\{S_M^0, S_M(\Gamma_f^*(\fibs(i)))\}_M\cong \{S^0,\Lambda^fM_+\}.\]
\end{prop}

\begin{proof}[Sketch of proof]
Since $M$ is a closed smooth manifold \myref{CWdualmaps} implies
there  is an isomorphism 
\[\{S_M^0,S_M(\Gamma_f^*(\fibs(i)))\}_M\cong \{S^0,tS^{\nu}
\odot S_M(\Gamma_f^*(\fibs(i)))\}\]
where $\nu$ is the normal bundle of $M$.

It remains to identify $S_M(\Gamma_f^*(\fibs(i)))$.  Let ${\nmalbdl}$
\nidx{tau@$\protect\nmalbdl$}
be the normal bundle of the inclusion of the diagonal of $M$ into 
$M\times M$.   
The ex-space
$S_{M\times M}(\fibs(i))$ is weakly equivalent to \[\triangle_!S^{\nmalbdl}
\boxtimes (\calP M,\tar\times \sou)_+.\]  Taking pullbacks of both 
sides gives a weak equivalence
\[ S_M(\Gamma_f^*(\fibs(i)))\simeq \triangle_!S^{\nmalbdl}\boxtimes 
\Lambda^fM_+.\]
Combining this with the isomorphism above, we have an isomorphism 
\[\{S_M^0, S_M(\Gamma_f^*(\fibs(i)))\}_M\cong \{S^0,S^\nu\boxtimes \triangle_!
S^{\nmalbdl}\boxtimes \Lambda^fM_+\}.\]
By construction the $\boxtimes$ product $S^\nu\boxtimes \triangle_!
S^{\nmalbdl}$ is equivalent to $S^\nu\wedge_MS^{\nmalbdl}$ and this bundle
is trivial.  So we have an isomorphism 

\[\{S_M^0, S_M(\Gamma_f^*(\fibs(i)))\}_M\cong \{S^0,S^n\bar{\wedge}
\Lambda^fM_+\}.\]

\end{proof}

\begin{theorem}\mylabel{KWclassicfinal}
Let $M$ be a closed smooth manifold and $f\colon M\rightarrow M$ a
continuous map.  Under the identification in \myref{kwclassidentify2}
\[R^{KW}(f)=R^{htpy}(f).\]\end{theorem}

\begin{proof}
To define the stable homotopy Euler class we used the map
\[\sect_+\amalg \sect_-\colon S^0_M\rightarrow f^*S_{M\times M}N(M\times M-\triangle).\]
The corresponding map ${\secttwo}\colon S^0_M\rightarrow \triangle_!S^{\nmalbdl}\odot
\Lambda^fM_+$ under the identification in 
\myref{kwclassidentify2} takes
the section of $S^0_M$ to the section of $\triangle_!S^{\nmalbdl}\odot
\Lambda^fM_+$ 
and on the other copy of $M$,
\[\secttwo(m)=((m,f(m)),H(m,f(m)))\]  where 
$H(f(m),m)$ is a path from $f(m)$ to $m$ as in \myref{localcont}.

Let $\phi$ denote the equivalence in \myref{kwclassidentify2}.
Then we have the following isomorphisms.
\[\xymatrix{[S^0_M,S_Mf^*N(M\times
M-\triangle)]_M\ar[d]_F\ar[r]^-{\phi_*}&[S^0_M,\triangle_!
S^{\nmalbdl}\odot \Lambda^fM_+]_M\ar[d]^F\\
\{S^0_M,S_Mf^*N(M\times M-\triangle)\}_M\ar[r]^-{\phi_*}\ar[d]_D&
\{S^0_M,\triangle_!S^{\nmalbdl}\odot \Lambda^fM_+\}_M\ar[d]^D\\
\{S^0,S^{\nu}\odot S_Mf^*N(M\times M-\triangle)\}\ar[r]_-{(1\odot
\phi)_*}& \{S^0,S^{\nu}\odot \triangle_!S^{\nmalbdl} \odot \Lambda^fM_+\} }\]
The map  $F$ is the stabilization isomorphism.  The map $D$ is the isomorphism
that defines the dual pair $(S_M^0,tS^\nu)$ as in \myref{CWdualmaps}.

In the top left corner we have the stable cohomotopy
Euler class and in the bottom right corner the corresponding  map is
\[\xymatrix{S^0\ar[r]^-\eta& S^{\nu}\odot S_M^0\ar[r]^-{\id\odot \secttwo}&
S^{\nu}\odot\triangle_!S^{\nmalbdl}\odot \Lambda^fM_+\simeq
(S^\nu\wedge_MS^{\nmalbdl})\odot \Lambda^fM_+\simeq S^n\times \Lambda^fM}\] This map is
the trace of the map induced by $f\colon M\rightarrow M$ on the space of
free Moore paths in $M$.\end{proof}

\section{Duality for unbased bimodules enriched in chain complexes}
\label{algrtrace}
There is an unbased version of the algebraic Reidemeister trace
which can be defined using the trace in a bicategory with shadows.
This invariant was defined by Coufal in \cite{coufal2,coufal} using a
different approach.

The unbased algebraic Reidemeister trace is an example of the trace in
the bicategory of categories, bimodules and natural
transformations enriched in chain complexes of modules over a
commutative ring $R$. In this bicategory the 0-cells are categories 
enriched in chain
complexes of modules over $R$. If $\sA$ and $\sB$ are categories
enriched in $\Ch_R$ then $\sA\otimes \sB$ is the category with objects
$\ob \sA\times \ob \sB$ and morphisms chain complexes given by the tensor
product of the morphism chain complexes of $\sA$ and $\sB$.
The 1-cells are enriched functors of the form
\[\sA\otimes \sB^{op}\rightarrow \Ch_R.\] We 
refer to functors of this form as $\sA$-$\sB$-bimodules. The 2-cells
are enriched natural transformations.  We denote this bicategory 
$\sE_{\Ch}$.

The bicategory composition $\odot$ is the enriched tensor product of 
functors.  If $\sX\colon \sA\otimes \sB^{op}\rightarrow \Ch_R$ and 
$\sY\colon \sB\otimes \sC^{op}\rightarrow \Ch_R$ are enriched functors
$\sX\odot \sY$ is the coequalizer
\[\xymatrix@C=70pt{{\displaystyle\coprod_{b,b'\in \ob(\sB)}\sX(a,b')\otimes
\sB(b,b')
\otimes \sY(b,c)}\ar@<.5ex>[r]^-{
\coprod \kappa_\sB
\otimes \id}
\ar@<-.5ex>[r]_-{\coprod \id\otimes \kappa_\sB}& {\displaystyle
\coprod_{b\in \ob(\sB)}
\sX(a,b)\otimes \sY(b,c)}.}\]

This bicategory is
the `many object' generalization of the bicategory of rings, chain
complexes of bimodules and maps of chain complexes. A complete
definition of this bicategory can be found in Section
\ref{catex2}.

Let $M$ be a connected CW complex
and let $\mathbb{Z}\Pi M$\nidx{zpim@$\protect\bZ\Pi M$} be the
category with objects the points of $M$ and $\bZ\Pi M(x,y)$ the free
abelian group on the homotopy classes of paths in $M$ from $x$ to
$y$. We have forgotten the topology on $\Pi M$.  

We think of $\bZ\Pi M$
as a category enriched in chain complexes of abelian groups
concentrated in degree 0. 
Define a right $\mathbb{Z}\Pi M$-module \nidx{C@$\protect\calC$} 
\[\calC M\colon \mathbb{Z}\Pi M\rightarrow \Ch_\mathbb{Z}\]
where $\calC M(x)$ is the cellular chains on the universal cover of
$M$ based at $x$. The action of a homotopy class of paths is the
chain map that is induced by the action on the universal cover.

Note that \[\mathbb{Z}\Pi M(x,x)\cong\mathbb{Z}\pi_1(M,x)\]
and \[\calC M(x)\cong C_*(\tilde{M}).\]

\begin{lemma}\mylabel{covtransdual}
The right $\mathbb{Z}\Pi M$-module $\calC M$ is dualizable if and only
if the right $\mathbb{Z}\pi_1(M,x)$-module $ C_*(\tilde{M})$
is dualizable.

Let $M$ be a  finite, connected CW complex.  Then $\calC M(x)$ is
dualizable as a right $\bZ\Pi M(x,x)$-module for any $x\in M$ and $\calC M$
is dualizable as a right $\bZ\Pi M$-module.
\end{lemma}

\begin{proof}

For any $x\in M$, the groupoid $\bZ\Pi M$ defines a
$\bZ\Pi M$-$\bZ\Pi M(x,x)$-bimodule $\bZ\Pi M(x,-)$\nidx{zpimx@$\protect\bZ
\Pi M(x,-)$} 
and a 
$\bZ\Pi M(x,x)$-$\bZ\Pi M$-bimodule
$\bZ\Pi M(-,x)$.\nidx{zpimx@$\protect\bZ\Pi M(-,x)$}  
These form a dual pair \[(\bZ\Pi M(-,x),
\bZ\Pi M(x,-)).\]  The
coevaluation map \[\eta\colon \bZ\Pi M\rightarrow \bZ\Pi M(x,-)\odot \bZ\Pi M(-,x)\]
takes a representative $\alpha$ of a homotopy class to
$(\alpha|_{[\frac{1}{2},1]}\circ \beta^{-1}  ,\beta \circ\alpha|_{[0,\frac{1}{2}]} )$
 where $\beta$ is any path from
$\alpha(\frac{1}{2})$ to $x$.  
This is independent of the choice of $\beta$ since
different choices are identified by the $\odot$ product.  The
evaluation map \[\epsilon\colon \bZ\Pi M(-,x)\odot \bZ\Pi M(x,-) \rightarrow
\bZ\Pi M(x,x)\] is composition.

Note that $\calC M\odot \bZ\Pi M(x,-)$ is isomorphic to $\calC M(x)$ for any
$x\in M$ and that the coevaluation and evaluation maps for the
dual pair $(\bZ\Pi M(x,-),\bZ\Pi M(-,x))$ are isomorphisms. The first 
statement then follows from Theorems \ref{dualcomposites1} and
\ref{dualcomposites2}.

For the second part of the lemma it is enough to show 
that one of the modules
is dualizable. For any $x\in M$, $\calC M(x)$ is a finitely generated
chain complex of free $\mathbb{Z}\pi_1(M,x)$-modules and so is
dualizable with dual $\Hom_{\mathbb{Z}\pi_1(M,x)}
(\calC M(x),\mathbb{Z}\pi_1(M,x))$.

Composing this dual pair with the dual pair  $(\bZ\Pi M(x,-),
\bZ\Pi M(-,x))$
gives a dual pair for $\calC M$. 
\end{proof}

This lemma is a special case of 
\myref{enrichedadjunction}.

\begin{rmk}
In the previous section we assumed that $M$ was a closed
smooth manifold. In this section we assumed that $M$ is a
connected finite  CW complex.  The different assumptions only
reflect that these categories have different `practical
generalities' that we want to work in.  
\end{rmk}

\section{The unbased algebraic Reidemeister trace}\label{algrtrace2}
Before we can define the unbased algebraic Reidemeister trace we need
to define the shadow in the bicategory of enriched categories, bimodules
and homomorphisms.  Let $\sZ$ be an $\sA$-$\sA$-bimodule.  Then $\sh{\sZ}$
is the coequalizer
\[\xymatrix@C=70pt{{\displaystyle\coprod_{a,a'\in \sA} \sA(a,a')\otimes
\sZ(a,a')}
\ar@<.5ex>[r]^-{\coprod
\kappa_\sA}
\ar@<-.5ex>[r]_-{\coprod(\kappa_\sA\circ \gamma)}
&{\displaystyle\coprod_{a\in \sA}\sZ(a,a)}
.}\]\nidx{shad@$\protect\sh{-}$}

Let $f\colon M\rightarrow M$ be a continuous map and $f_*$\nidx{f@$f_*$} 
the induced
map on $\bZ\Pi M$. Let
$U_{\bZ\Pi M}^{f_*}$\nidx{zpifm@$\protect U_{\bZ\Pi M}^{f_*}$} 
be the $\bZ\Pi M$-$\bZ\Pi M$-bimodule 
defined by
$U_{\bZ\Pi M}^{f_*}(x,y)=\bZ\Pi M(f(y),x)$ with action of $\bZ\Pi
M\otimes \bZ\Pi M^{op}$
 given by $(\gamma,\alpha,\beta)\mapsto (\gamma\alpha f(\beta))$.
The map $f\colon M\rightarrow M$ defines a natural transformation
\nidx{cf@$\protect\calC^{f_*}$}
\[\tilde{f}_*\colon \calC M\rightarrow \calC^{f_*}M\coloneqq 
\calC M\odot U_{\bZ\Pi M}^{f_*}.\]
The trace of $\tilde{f}_*$ in the bicategory of enriched
categories, bimodules, and natural transformations is a map
\[\mathbb{Z} \rightarrow \sh{U_{\bZ\Pi M}^{f_*}}.\]

\begin{definition} The \emph{unbased algebraic Reidemeister
trace}\idx{unbased algebraic Reidemeister trace}\idx{Reidemeister 
trace!unbased algebraic} of $f$, $R^{U,alg}(f)$,\nidx{ralgu@$R^{U,alg}$} is 
the image of 1 under
 the trace of $\tilde{f}_*$.  \end{definition}

\begin{rmk} The definition of the unbased algebraic Reidemeister
trace is not the same as the definition of the unbased algebraic
generalized Lefschetz number in \cite{coufal}, but it is not 
difficult to see the invariants  are the same. 
An equivalent characterization of a dualizable 
1-cell in a closed bicategory
implies that the dual of $\calC M$ is $\Hom_{\bZ\Pi M}(\calC M,\bZ\Pi M)$ 
and the coevaluation is 
\[\xymatrix{{\mathbb{Z}}\ar[r]&\Hom_{\bZ\Pi M}(\calC M,\calC M)\ar[r]^-{\nu^{-1}}&
\protect\calC M \odot \Hom_{\bZ\Pi M}(\calC M,\bZ\Pi M) .}\]
The unbased algebraic Reidemeister trace is the image of 1 under
the map
\[\xymatrix{{\mathbb{Z}}\ar[d]&
\sh{U_{\bZ\Pi M}^{f_*}}\\
\Hom_{\bZ\Pi M}(\calC M,\calC M)\ar[d]_{\tilde{f}_*}&
\sh{\calC^{f_*}M\odot \Hom_{\bZ\Pi M}(\calC M,\bZ\Pi M)}
\ar[u]^-\epsilon\\
\Hom_{\bZ\Pi M}(\calC M,\calC^{f_*}M)\ar[r]^-{\nu^{-1}}&
\Hom_{\bZ\Pi M}(\calC M,\bZ\Pi M)\odot 
\calC^{f_*} M
\ar[u]^\cong
}.\]

The invariant defined in \cite{coufal} is the image of $\tilde{f}_*$
under the map 
\[\Hom_{\bZ\Pi M}(\calC M,\calC^{f_*}M)\stackrel{\nu^{-1}}{\rightarrow}
\calC^{f_*}M\odot \Hom_{\bZ\Pi M}(\calC M,\bZ\Pi M) \stackrel{\ev}{\rightarrow} 
\sh{U_{\bZ\Pi M}^{f_*}}.\]
Therefore the unbased generalized Lefschetz number 
coincides with the unbased algebraic Reidemeister trace.
\end{rmk}

The map $\tilde{f}_*$ is of the form described in
\myref{tracegpd} and the coevaluation for the dual pair
$(\bZ\Pi M(x,-),\bZ\Pi M(-,x))$  is an isomorphism so we have the following
corollary.

\begin{corollary}\mylabel{algtracecompare}
A choice of base point $\ast\in M$  determines an isomorphism
\[\sh{U_{\bZ\Pi M}^{f_*}}\rightarrow \sh{\pi_1(M,\ast)^\phi} \] which
takes $R^{U,alg}(f)$ to $R^{alg}(f)$.

\end{corollary}

\section{The proof of  \myref{reidemeister1} and some properties of 
the trace}
We can now give a conceptual proof of \myref{reidemeister1}.  We will
start by connecting the geometric Reidemeister trace with the 
algebraic Reidemeister trace.  Here we use the functoriality of the 
trace in bicategories with shadows.

The rational cellular chains functor is a symmetric monoidal
functor which commutes with trace. We can use this functor to
define  a lax functor of bicategories
$C_*(-;\mathbb{Q})\colon \sM_{\Ex}\rightarrow \sE_{\Ch}$.\nidx{cq@$C_*(-; \protect\bQ)$} 
A monoid $\sA$
in $\Ex$ with projection map $\tar\times \sou\colon \sA\rightarrow A\times A$
is taken to the category with objects the set $A$ (forget
the topology).  For $a,a'\in A$, the morphism chain complex is
\[C_*((\tar\times \sou)^{-1}(a,a');\mathbb{Q}),\] the rational cellular
chains on the inverse image of $(a,a')$. An
$\sA$-$\sA'$-bimodule $\sX$ with projection $\pro\times\pro'$
is taken to the bifunctor which on a pair
of objects $(a,a')\in A\times A'$ is the chain complex
$C_*((\pro\times \pro')^{-1}(a,a');\mathbb{Q})$.  A map of modules is
taken to the induced map on cellular chains.

For any monoid $\sA$ the map $\phi_{\sA}$ is the identity.  For
bimodules $\sX$ and $\sY$ the map $\phi_{\sX,\sY}$ is induced by the
inclusion maps
\[(\pro_\sX\times \pro'_\sX)^{-1}(a,a')\wedge (\pro_\sY\times \pro'_\sY)^{-1}(a',a'')
\rightarrow (\pro_\sX)^{-1}(a)\odot (\pro'_\sY)^{-1}(a'').\]

\begin{lemma}\mylabel{alggeocompare} The functor $C_*(-;\mathbb{Q})$
defines an isomorphism \[\mathbb{Z}\sh{\Pi^{f_*}M}\rightarrow
\sh{t({\mathbb{Z}\Pi^{f_*}M})}\] and under this isomorphism
$R^{U,geo}(f)= R^{U,alg}(f)$.  The same functor defines an
isomorphism \[\mathbb{Z}\sh{ \pi_1M^\phi}\rightarrow
\sh{\mathbb{Z}\pi_1M^\phi}.\]  Under this isomorphism
$R^{geo}(f)=R^{alg}(f)$.
\end{lemma}

We can now complete the proof of \myref{reidemeister1} from 
the introduction.

\begin{proof}[Proof of \myref{reidemeister1}]
By \myref{alggeocompare} there is an isomorphism
\[\mathbb{Z}\sh{\pi_1M^\phi}\cong \sh{\mathbb{Z}\pi_1M^\phi}\] and
under this isomorphism \[R^{alg}(f)=R^{geo}(f).\]  
\myref{geohtpycompare} implies $R^{geo}(f)$ is zero if and only if
$R^{htpy}(f)$ is zero. By  \myref{KWclassicfinal}
there is an equivalence \[S_Mf^*N(M\times M-\triangle)\simeq 
\triangle_!S^{\nmalbdl}\odot \Lambda^fM_+\] and under this 
equivalence
\[R^{KW}(f)=R^{htpy}(f).\]  By \myref{converse1}
$R^{KW}(f)$ is zero if and only if $f$ is homotopic to
a map with no fixed points.
\end{proof}

In addition to demonstrating the compatibility of various forms of
the Reidemeister trace, we can use the trace in bicategories with
shadows to give new proofs of various basic results in fixed point theory. 
These results are applications of \myref{tracemult} and \myref{tracecyclic}. 

Since these results follow from properties of trace in bicategories
they hold for any of the forms of the Reidemeister trace.

\begin{prop} For a product of continuous maps \[f_M\times f_N\colon M\times N
\rightarrow M\times N\]  of closed smooth manifolds, one of which is
simply connected,
\[R(f_M\times f_N)=R(f_M)\times R(f_N).\]
\end{prop}

This follows from \myref{tracemult} and 
is a very special case of results in
\cite{Dold3, Heath, HKW, No}.

According to \cite[I.5.2]{jiang} the Nielsen number satisfies a
commutativity property.  Let $X$ and $Y$ be compact connected
ENR's and $f\colon X\rightarrow Y$, $g\colon Y\rightarrow X$ be continuous
maps.  Then $N(g\circ f)= N(f\circ g)$.  We can recover this result from 
\myref{tracecyclic}.

\begin{corollary} If $M$ and $N$ are closed smooth 
manifolds and $f\colon M\rightarrow
N$ and $g\colon N\rightarrow M$ are continuous maps, there is a bijection
between the fixed point classes of $f\circ g$ and $g\circ f$ and
under this identification
\[R(f\circ g)=R(g\circ f).\]
\end{corollary}
In particular, $N(f\circ g)=N(g\circ f)$.

\section{The Reidemeister trace for regular covering spaces}

In addition to the Reidemeister trace defined using the 
universal cover, there is a
Reidemeister trace for all regular covering spaces.  The theory
for regular covers is very similar to the theory for universal covers, except
maps do not always lift to regular covers.  To resolve this
problem, we will restrict attention to those maps that do have
lifts.  This means that for a normal subgroup $K$ of $\pi_1M$ we will
only consider maps $f\colon M\rightarrow M$ such that $\phi(K)\subset
K$.  Here $\phi$ is the same as in Chapter \ref{reviewfp2}; it is
the map induced on $\pi_1M$ by $f$ after choosing a base point and
a path $\basepath$  from that base point to its image under $f$.

\begin{definition}\cite[III.2.1]{jiang}
Two fixed points $x$ and $y$ of $f\colon M\rightarrow M$
are in the same \emph{mod $K$ fixed point class}\idx{mod K fixed point
class}\idx{fixed point class!mod K} if there exists a lift of $f$
to $\tilde{f}/K\colon \tilde{M}/K\rightarrow \tilde{M}/K$\nidx{mk@$\tilde{M}/K$} 
and lifts of $x,y$ to
$\tilde{x},\tilde{y}\in\tilde{M}/K$ such that $\tilde{f}/K
(\tilde{x})=\tilde{x}$
and $\tilde{f}/K(\tilde{y})=\tilde{y}$.
\end{definition}

If $K$ is the trivial subgroup of $\pi_1M$, this is the usual definition
of fixed point classes.

\begin{lemma}\cite[III.2.2]{jiang} Two fixed points $x$ and $y$ are in the same mod $K$ fixed
point class if and only if there is path $\gamma$ in $M$ from $x$ to $y$ such
that $\gamma f(\gamma^{-1})$ is in $K$.\end{lemma}

Since $\phi(K)\subset K$, $\phi$ induces a map $\phi\colon  \pi_1M/K
\rightarrow \pi_1M/K$.  Let $\sh{(\pi_1M/K)^\phi}$ be the
semiconjugacy classes of $\pi_1M/K$ with respect to the induced
homomorphism $\phi$.\nidx{shad@$\protect\sh{-}$}  For each fixed point 
$x$ pick a path $\gamma_x$ in $M$ from the base
point $\ast$ to $x$.  Then there
is a well defined injection from the mod $K$ fixed point classes of $f$ to
$\sh{(\pi_1M/K)^\phi}$\nidx{pimphi@$\protect\sh{(\pi_1M/K)^\phi}$} 
that takes a fixed point $x$ to the homotopy
class of the path $\gamma_x^{-1}f(\gamma_x)\basepath$.

\begin{definition} The \emph{mod $K$ geometric 
Reidemeister trace}\idx{mod K geometric Reidemeister trace}
\idx{Reidemeister trace!mod K geometric} of $f$, 
$R^{geo}_K(f)$,\nidx{rgeok@$R^{geo}_K$}  is
\[\sum_{\mathrm{mod\, K\, fixed\, point\, classes}\,F_j}i(F_j)\cdot F_j\in
\mathbb{Z}\sh{(\pi_1M/K)^\phi}\]
\end{definition}

For spaces with a universal cover there is a bijection between
regular covers and normal subgroups of the fundamental group.
These regular covers provide more examples of dual pairs.

\begin{lemma}\mylabel{modKdual}Suppose $M$ is a space with 
a universal cover $\tilde{M}$ and 
$\tilde{M}_+$ is
dualizable as a $\pi_1M$ space.  If $K\lc\pi_1M$, then
$(\tilde{M}/K)_+$ is dualizable as a $(\pi_1M)/K$ space.\end{lemma}

\begin{proof} This proof uses a composite of dual pairs.  The group $\pi_1M/K$
has actions by $\pi_1M$ and $\pi_1M/K$ on both the left and the
right. We think of $\pi_1M/K$ as a $\pi_1M$-$\pi_1M/K$-bimodule and
$t\pi_1M/K$ as a $\pi_1M/K$-$\pi_1M$-bimodule.  Then
\[(\pi_1M/K,t(\pi_1M/K))\] is a dual pair.  The coevaluation
\[\pi_1M\rightarrow \pi_1M/K\odot t(\pi_1M/K)\]
is the quotient map.  The evaluation
\[t(\pi_1M/K)\odot\pi_1M/K \rightarrow \pi_1M/K\] is
composition.

Note that $\tilde{M}_+\odot (\pi_1M/K)$ is a cover of $M$
corresponding to the subgroup $K\subset \pi_1M$.  Since 
both $\tilde{M}_+$ and $\pi_1M/K$ are dualizable 
\myref{dualcomposites1} implies the 
composite $\tilde{M}/K$ is dualizable with dual $(t\pi_1M/K)\odot T\pi^*\nu$.
\end{proof}

The proof of \myref{pihiso} also implies the following result.

\begin{lemma} The map induced on homology by the 
trace of $\tilde{f}/K\colon \tilde{M}/K\rightarrow \tilde{M}/K$\nidx{fk@$\tilde{f}/K$}
is the Mod $K$ geometric Reidemeister trace of $f$.\end{lemma}

\begin{lemma}\mylabel{modKtrace} The map $\sh{\pi_1M^\phi}\rightarrow
\sh{(\pi_1M/K)^\phi}$ that takes a semiconjugacy class to the
corresponding mod $K$ semiconjugacy class takes the Reidemeister
trace of $f$ to the mod $K$ Reidemeister trace of $f$.  In
particular, the sum of the coefficients of the Reidemeister trace
is the Lefschetz number.
\end{lemma}

\begin{proof}The map used to
define the Reidemeister trace is of the form described in
\myref{tracegpd} so this result follows from \myref{modKdual}
and \myref{tracegpd}. If $K$ is $\pi_1M$ the mod $K$ 
Reidemeister trace is the Lefschetz number, so  
the sum of the coefficients of
$R^{geo}(f)$ is the Lefschetz number of $f$.
\end{proof}

\chapter{Duality for fiberwise parametrized modules}
\label{fibfpsec}

In this chapter we describe fiberwise generalizations of some of
the results from the previous two chapters.  Unfortunately, since
we are now interested in fiberwise maps not all invariants defined
in Chapter \ref{classfpsec} make sense.  For example, it is no
longer possible to choose a base point and so we will now only use
unbased invariants. Another challenge, and benefit, of the
fiberwise generalization is that the invariant that gives a
converse to the fiberwise Lefschetz fixed point theorem, a
generalization of the homotopy Reidemeister trace, is much richer
than the classical invariant.  One consequence of this is that it
isn't clear what invariants, if any, deserve to be called the
fiberwise geometric Reidemeister trace or the fiberwise algebraic
Reidemeister trace.  We describe one candidate invariant for the
fiberwise geometric Reidemeister trace. This invariant was defined
by Scofield.  It does not give a converse to the fiberwise
Lefschetz fixed point theorem.

We define the fiberwise homotopy Reidemeister trace using the
approach of the previous chapters. This invariant  can be
identified with the fiberwise invariant defined by Klein and
Williams.  The definitions of these invariants, their comparison,
and even the proof that these invariants give a converse to the
fiberwise Lefschetz fixed point theorem are almost identical to
the approach in the classical case.  

\section{Fiberwise Costenoble-Waner duality}

The bicategory we use to study fiberwise spaces is closely related
to the bicategory $\Ex$ we used to study classical fixed point
theory.  The 0-cells are spaces over $B$.  The 1-cells are spaces over 
and under the 0-cells.  The 2-cells are maps of total spaces that commute
with the section and projection.  This bicategory was introduced in
\cite{MS}, where a more sophisticated stable version was also studied.

More formally, the 0-cells of $\Ex_B$\nidx{Exb@$\Ex_B$} are spaces over
$B$. That is, a space $C$ with a map $C\rightarrow B$. A 1-cell
from $C\rightarrow B$ to $D\rightarrow B$ is a space $X$ and maps
\[\xymatrix{D\times_BC\ar[r]^-{\sect}& X\ar[r]^-{\pro} &D\times_BC}\]
such that the composite $\pro\circ \sect$ is the identity map of $D\times_BC$.
For two 1-cells $X$ and $Y$ from $C$ to $D$, a 2-cell from $X$ to
$Y$ is a map $f\colon X\rightarrow Y$ such that
\[\xymatrix{C\times_BD\ar[r]\ar@{=}[d]&X\ar[r]\ar[d]^f
&C\times_BD\ar@{=}[d]\\
C\times_BD\ar[r]&Y\ar[r]&C\times_BD}\] commutes.

As in \myref{exspaceconditions}, for a 1-cell $X$ over 
$C$ and $D$ we require that $X$ and $C\times_BD$ are 
of the homotopy types of CW-complexes,  
the projection $X\rightarrow C\times_BD$ is a Hurewicz
fibration, and the section $C\times_BD\rightarrow X$ 
is a fiberwise cofibration.  When these conditions are
not satisfied 
we implicitly use the model structures and approximation
techniques from \cite{MS} to maintain homotopical control. 

The bicategory composition in $\Ex_B$ is very similar to the
bicategory composition in $\Ex$.  The external smash product of two
1-cells in $\Ex_B$, written $\bar{\wedge}$\nidx{$\bar{\wedge}$}, is
defined by taking the fiberwise smash product over the 0-cells.
This is not the fiberwise smash product over $B$. If $X$ is a
1-cell from $C$ to $D$ and $Y$ is a 1-cell from $D$ to $E$ then we
define $X\boxtimes Y$, a 1-cell from $C$ to $E$, as the pullback
along $\triangle\colon D\rightarrow D\times_BD$ and then pushforward
along $r\colon D\rightarrow B$ of $X\bar{\wedge}Y$\nidx{$\boxtimes$}
 \[\xymatrix{C\times_B E\ar[d]&C\times_B D
\times_B E\ar[l]_-{\id\times r\times \id}\ar[r]^-{\id\times
\triangle\times \id}\ar[d]&C\times_B D\times_B D
\times_B E\ar[d]\\
X\boxtimes Y\ar[d]&(\id\times \triangle\times \id)^*(X
\bar{\wedge}Y )\ar[r]\ar[l]\ar[d] &X\bar{\wedge}
Y\ar[d]\\
C\times_B E&C\times_B D \times_B E\ar[l]^-{\id\times r\times
\id}\ar[r]_-{\id\times \triangle\times \id} &C\times_B D\times_B D
\times_B E}\]
The unit 1-cell associated to a 0-cell $C\rightarrow B$ is $(C,\triangle)_+$
and we  will denote this $U_C$.

\begin{definition}
A 1-cell $X$ in $\Ex_B$ over $C$ is \emph{fiberwise Costenoble-Waner  n-dualizable}
\idx{Costenoble-Waner duality} 
if there is a 1-cell $Y$ over $C$ and maps 
\[\xymatrix{S^n_B\ar[r]^-\eta&X\boxtimes tY
&{\mathrm{and}}&tY\boxtimes X\ar[r]^-\epsilon
&\triangle_!S^n_C}\]
such that
\[\xymatrix{S^n_B\boxtimes X\ar[r]^-{\eta\boxtimes \id}
\ar[dd]_{\gamma}&(X\boxtimes tY)
\boxtimes X\ar[d]^{\cong} &
tY\boxtimes S^n_B\ar[r]^-{\id\boxtimes \eta}\ar[dd]_{(\sigma\boxtimes \id)
\gamma}&tY\boxtimes(X\boxtimes tY)
\ar[d]^{\cong}\\
&X\boxtimes (tY\boxtimes X)\ar[d]^-{\id \boxtimes
\epsilon}&
&((tY\boxtimes X)\boxtimes  tY)\ar[d]^-{ \epsilon \boxtimes  \id}\\
X\boxtimes S^n\ar[r]_-{\cong}& X\boxtimes 
\triangle_!S^n_C
&S^n\boxtimes tY\ar[r]_-{\cong}& \triangle_!S^n_C\boxtimes
tY}\] 
 commute up to fiberwise homotopy over $C$.
\end{definition}

To give examples of this kind of duality we will need to consider
equivariant Costenoble-Waner duality.\idx{Costenoble-Waner duality!equivariant}\idx{equivariant
Costenoble-Waner duality}
A bundle construction gives a connection between dual pairs in the
bicategory $G\Ex$ and dual pairs in the bicategory $\Ex_B$.

Let $G$ be a compact Lie group.  There is a bicategory
$G\Ex$\nidx{gex@$G\Ex$} with 0-cells $G$-spaces.  The 1-cells in $G\Ex$
are ex-spaces $X$ with an action by $G$ such that the section and
projection maps are equivariant.  The 2-cells are equivariant maps
of total spaces that commute with the section and projection maps.
The bicategory composition is induced from that in $\Ex$.  The
group $G$ acts by the diagonal action.

This bicategory also has a stable version and duality in that bicategory
has an interpretation as $V$-duality in $G\Ex$.\idx{v dual@$V$-dual}
\idx{v dualizable@$V$-dualizable}

Let $S^V$\nidx{sv@$S^V$} denote the one point compactification of a
representation $V$ of $G$.
\begin{definition}\cite[18.3.1]{MS}
A 1-cell $X$ in $G\Ex$ is $V$-dualizable for a
representation $V$ of $G$ if there is a 1-cell $Y$ in $G\Ex$ and
maps
\[\xymatrix{S^V\ar[r]^-\eta&X\boxtimes tY&{\mathrm{and}}&
tY\boxtimes X\ar[r]^-\epsilon&\triangle_!S^V_B}\]
such that 
\[\xymatrix{S^V \wedge X \ar[d]_{\gamma}\ar[r]^-{\eta\wedge \id}
& (X\wedge Y)\wedge X \ar[d]^\cong\\
 X\wedge S^V & X\wedge (Y\wedge X)\ar[l]^-{\id\wedge \epsilon} }
\qquad
\xymatrix{ Y \wedge S^V \ar[d]_{(\sigma\wedge \id)\gamma}\ar[r]^-{\id\wedge \eta} 
& Y \wedge (X\wedge Y) \ar[d]^\cong \\
 S^V\wedge Y & (Y\wedge X)\wedge Y\ar[l]^-{\epsilon\wedge \id}}\]
commute stably up to equivariant fiberwise homotopy.
\end{definition}

\myref{dualM} has a generalization to the bicategory $G\Ex$.
\begin{theorem}\mylabel{GdualM}\cite[18.6.1]{MS}
Let $M$ be a closed smooth  manifold
embedded in a representation $V$.  Then $(S^0_M, S^\nu)$ is a
Costenoble-Waner $V$-dual pair.
\end{theorem}

Let $P$ be a principal $G$-bundle and $B$ be $P/G$.  Then there is a lax
functor \nidx{p@$\protect\bP$} \[\bP\colon G\Ex\rightarrow \Ex_B.\]
This functor takes a $G$-space $F$ to $P\times_GF$, where $P\times_GF$
is $P\times F$ quotiented by the diagonal action of $G$.  On 1-cells
and 2-cells $\bP$ is also given by the functor $P\times_G (-)$, which
converts an ex-$G$-space $E$ over a $G$-space $F$ into a 1-cell in $\Ex_B$.
The section and projection
maps of $E$ over $F$ induce section and projection maps for $\bP(E)$
over $B$.

\begin{theorem}\cite[19.4.4]{MS} If $(X,Y)$ is a dual pair in
$G\Ex$,  $(\bP(X),\bP(Y))$ is a dual pair in $\Ex_B$.
\end{theorem}

Combining this theorem with \myref{GdualM} we have the following
result.

\begin{corollary}\mylabel{GfibdualM}
If $p\colon M\rightarrow B$ is a fiber bundle with compact manifold fibers,
then the dual of $(M,\id)_+\in \Ex_B(M,\ast)$ is the fiberwise one point
compactification of the fiberwise normal bundle.
\end{corollary}

The coevaluation and evaluation maps for this dual pair are similar to
those for the dual pair in \myref{dualM}.  The evaluation
is defined using the following generalization of \myref{localcont}.

\begin{lemma}\cite[II.5.17]{CrabbJames} \mylabel{fiblocalcont}
Let $L\rightarrow B$ be a fiberwise ENR.  Then there is an open
neighborhood $W$ of the diagonal in $L\times_BL$ and a fiberwise
homotopy $H_t\colon W\rightarrow L$ such that $H_0(x,y)=x$, $H_1(x,y)=y$
and $H_t(x,x)=x$.
\end{lemma}\nidx{H@$H$}

For a 0-cell $M$ in $\Ex_B$ there is a dual pair $((M, \id)_+,t(M,\id)_+)$.
The coevaluation map is the diagonal map
\[\triangle\colon  M\rightarrow M\times_BM.\]
If $p\colon M\rightarrow B$ is the map from $M$ to $B$, the evaluation map is
\[p_+\colon M_+\rightarrow S^0_B.\]
The dual pair $(M_+,T\nu)$ in \myref{fibdualclas} is the composite
of $((M,\id)_+,t(M,\id)_+$ with the dual pair in \myref{GfibdualM}.

Since fiberwise Costenoble-Waner duality is duality in a bicategory
there are other descriptions of dual pairs.  The only other
characterization we will need is given in the following
corollary.\idx{fiberwise Costenoble-Waner duality}\idx{Costenoble-Waner
duality!fiberwise}

\begin{corollary}\mylabel{CWfibdualmaps}
If $(X,Y)$ is a Costenoble-Waner dual pair in $\Ex_B$,
then the coevaluation map of the dual pair induces an isomorphism
\[\{W\odot X,Z\}_M\rightarrow \{W,Z\odot tY\}_B\]
for $W\in \Ex_B(B,B)$ and $Z\in \Ex_B(M,B)$.
\end{corollary}

\section{Ranicki duality for fiberwise spaces}

The bicategory $\sM_{\Ex_B}$ of monoids,
bimodules, and maps in $\Ex_B$ is defined exactly as the bicategory
$\sM_\Ex$ is defined.  The bicategory composition and
shadow are defined in analogy with the composition and shadow in
Chapter \ref{classfpsec}.  We also have equivalences of the bar
resolutions and cyclic bar resolutions with colimits in some
special cases.
Like dual pairs of spaces and dual pairs in $\Ex$, the dual pairs  in 
$\Ex_B$ and dual pairs of modules in $\Ex_B$ are
defined using $n$-duality.

If $M$ is a space over $B$, $\pro\colon M\rightarrow B$, instead of
considering the topologized fundamental groupoid or the free Moore
path space we will use the fiberwise free Moore paths,\idx{fiberwise
free Moore paths}\idx{free Moore paths!fiberwise} \nidx{pbm@$\protect\calP
_BM$}
\[\calP _BM=\{(\gamma,u)\in \mathrm{Map}_B(B\times [0,\infty),M)\times [0,\infty)|
\gamma(t)=\gamma(u)\,\mathrm{for \, all}\, t\geq u\}.\]  This is a
space over $B$ with the map to $B$ given by $(\gamma,u)\mapsto
\pro\gamma(0)$. The space $\calP_BM$ is  the free Moore paths in
$M$ that are each contained in a single fiber over $B$.
This space is given the subspace topology from $M^{[0,\infty)}\times [0,\infty)$.
For more details on $\mathrm{Map}_B$ see \cite[1.3.7]{MS}.

\begin{lemma}\mylabel{fibrationforfiberproduct}
 If $p\colon M\rightarrow B$ is a fibration
the map 
\[\tar\times \sou\colon P _BM\rightarrow M\times_BM\] given by $(\tar\times
\sou)(\gamma,u) =(\gamma(u),\gamma(0))$  is a fibration.
\end{lemma}

To minimize notation we will use paths rather than Moore paths
in this proof.

\begin{proof}
We must show that all diagrams
\[\xymatrix{ X\ar[r]^-f\ar[d]_{i_0}&{\notcalP}_BM\ar[d]^{(\tar,\sou)}\\
X\times I \ar[r]_-H\ar@{.>}[ur]^k&M\times_BM}\] have a lift
$k$.

The map $f$ has an adjoint $\bar{f}\colon X\times I\rightarrow M$
which satisfies $p(f(x, t))=p(f(x,0))$ for all $x\in X$, $t\in
I$. Let $H=H_1\times_B H_0$. Since the diagram commutes $\bar{f}$
must also satisfy $\bar{f}(x,0)=H_0(x, 0)$ and $\bar{f}(x,1)
=H_1(x,0)$.
Let $J$ be the subset $(0,I)\cup (I,0) \cup (1,I)$
of $I\times I$.  Then a lift $k$ in the
diagram above corresponds to a lift $\bar{k}$ in the
diagram
\[\xymatrix{ X\times J\ar[r]^g\ar[d]_\iota &M\ar[d]^{p_2}\\
X\times I \times I\ar[r]_-{\tilde{H}}\ar@{.>}[ur]^{\bar{k}}&B}\]
where $g\colon J\times X\rightarrow M$ is defined by
\[\begin{array}{rcl}
g(x,0,s)&=&H_0(x,s) \\
g(x,t,0)&=&\bar{f}(x,t) \\
g(x,1,s)&=& H_1(x,s)
\end{array}\]
and $\tilde{H}(x,t,s)=pH_1(x,s)$

Let $\phi\colon X\times J\times I\rightarrow X\times I\times I$ be a
homeomorphism such that \[\xymatrix{X\times
J\ar[rr]^\iota\ar[dr]_{i_0}&&X\times I\times I
\\&X\times J\times I\ar[ru]_\phi}\] commutes.
Then there is a lift $\tilde{k}$ in the diagram
\[\xymatrix{X\times J\ar[rr]^g\ar[dr]^{i_0}
\ar[dd]_\iota &&M\ar[dd]^{p_2}\\
&X\times J \times I \ar[dr]^-{\tilde{H}\circ \phi}\ar[dl]_\phi
\ar@{.>}[ur]^{\tilde{k}}\\
X\times I \times I\ar[rr]_{\tilde{H}}&&B}\] since $p_2$ is a
fibration, and $\tilde{k}\circ\phi^{-1}$ defines the lift
$\bar{k}$.

\end{proof}

Composition of paths gives a
strictly associative product $\calP_BM\times_M\calP_BM \rightarrow
\calP_BM$.  The inclusion of $M$ into $\calP_BM$ is the unit.
Adding a disjoint section to $\calP_BM$ gives a monoid in
$\Ex_B$.

Recall that for a monoid $\sA$, $R(\sA)$ is $\sA$ regarded as a
right $\sA$-module and $L(\sA)$ is $\sA$ regarded as a left $\sA$-module.

\begin{corollary}
Let $M\rightarrow B$ be a fiber bundle with compact
manifold fibers.  Then we have the following dual pairs.
\begin{enumerate} \item $(R(\calP_BM,\tar\times \sou)_+,
L(\calP _BM, \tar\times \sou)_+)$
\item $( (\calP_BM,\sou)_+, T_M\sou^*S_M^{\nu_B})$.\end{enumerate}
\end{corollary}

\begin{proof}
Here $T_M\sou^*S_M^{\nu_B}$\nidx{tmssmnub@$T_M\protect\sou^*S_M^{\nu_B}$}
is defined to be $L(\calP_BM,\tar\times \sou )_+\odot S_M^{\nu_B}$
where $\nu_B$ is the fiberwise normal bundle of $M$ over $B$ and
$S_M^{\nu_B}$ is the fiberwise one point compactification of
$\nu_B$ over $M$.\nidx{nub@$\nu_B$}\nidx{smnub@$S_M^{\nu_B}$}

The first of these dual pairs comes from the monoid $(\calP
_BM,\tar\times \sou)_+$ as in \myref{topmonodidual}.
The second is the composite of the dual pairs
\[( S_M^0,  tS_M^{\nu_B})\] and \[(R(\calP _BM,\tar\times\sou)_+,
L(\calP _BM,\tar\times \sou)_+).\] \end{proof}

\chapter{Fiberwise fixed point theory}\label{fibfpsec2}

In Chapter  \ref{classfpinv}, the corresponding chapter for classical
invariants, we described many invariants and gave several applications
of trace in bicategories.  This section is much
shorter.  One of the reasons is that based invariants cannot be
defined for fiberwise space.  Another is that is not clear 
what invariants should be the generalization of the algebraic and geometric
Reidemeister traces.

In contrast to the algebraic and geometric Reidemeister traces,
the homotopy Reidemeister trace has a straightforward fiberwise
generalization. The invariant defined by Klein and Williams also
has a fiberwise generalization and their proof of the converse to
the Lefschetz fixed point theorem easily generalizes.

\section{Fiberwise fixed point theory invariants}

The fiberwise homotopy Reidemeister trace is based on the
fiberwise free Moore paths monoid.  This is the
Nielsen-Reidemeister invariant defined by Crabb and James
in  \cite[II.6]{CrabbJames} and in Section \ref{fibkwidsec}
it is identified with the invariant defined by Klein
and Williams in \cite{KW}.

Let $M\rightarrow B$ be a fiber bundle with dualizable fibers and
$f\colon M\rightarrow M$ a fiberwise map.  Then $f$ can be used to
define a monoid $(\calP_B^fM,\tar\times \sou)_+$\nidx{pbfm@$\protect\calP_B^fM$}
in $\Ex_B$.  This is analogous to the definition of $(\calP^f M,
\tar\times \sou)_+$. The map $f$ also defines a map of right $(\calP
_BM,\tar\times \sou)_+$ modules
\[\tilde{f}\colon (\calP _BM,\sou)_+\rightarrow
(\calP_BM,\sou)_+\odot (\calP^f_BM,\tar\times \sou)_+.\]\nidx{f@$\tilde{f}$}

\begin{definition} The \emph{fiberwise homotopy Reidemeister trace},
$R^{htpy}_B(f)$, is the
trace of $\tilde{f}$.
\end{definition}\idx{fiberwise homotopy Reidemeister trace}\idx{Reidemeister
trace!fiberwise homotopy}\nidx{rhtpyb@$R^{htpy}_B$}

The fiberwise homotopy Reidemeister trace is a fiberwise stable map over $B$
\[S^n\times B\rightarrow \sh{S^n\wedge (\calP^f _BM, \tar\times
\sou)_+}.\]
The shadow of $(\calP^f_BM,\tar\times \sou)$ is equivalent to
\[\Lambda^f_BM=\{\gamma\in M^I|f(\gamma(1))=
\gamma(0), p(\gamma(t))=p(\gamma(0)) \,\mathrm{for\,all}\,t\in I\}.
\]\nidx{lambdafb@$\Lambda_B^fM$}

The fiberwise homotopy Reidemeister trace
is a very rich invariant and it should be possible to use this
invariant to define other, simpler, invariants.  One invariant we can
extract from the fiberwise homotopy Reidemeister trace is the
fiberwise Nielsen number defined by Scofield.

The fiberwise
homotopy Reidemeister trace is a map
\[S^n\times B\rightarrow S^n\wedge\Lambda_B^fM_+.\]
Each connected component of $\Lambda^f_BM$ has a map to $B$ and so we get a
map
\[\xymatrix@C=15pt{H_0(B_+)\cong H_n(S^n\wedge (B_+))\ar[r]&
H_n(S^n\wedge \Lambda^f_{B}M_+)
\ar[r]^-\phi&\oplus H_n(S^n\wedge (B_+))\cong \oplus H_0(B_+)}\]
where $\phi$ is induced by the decomposition of $\Lambda_B^fM_+$
into path components followed by the map to $B$.
The image of this map on one component $H_0(B)$ in $\oplus H_0(B)$  is the
fiberwise index of the corresponding
`fiberwise fixed point class'.  The fiberwise
Nielsen number is the number of fiberwise fixed point classes with
nonzero index.\idx{fiberwise fixed point classes}\idx{fiberwise Nielsen
number}\idx{Nielsen number!fiberwise}

Scofield showed this invariant does not give a converse to the
fiberwise Lefschetz fixed point theorem.

\begin{ex}\cite[V.3.16]{scofield} Let $f\colon S^3\times S^3\rightarrow S^3\times S^3$
be the map $f(b,z)=(b,b^{24}z)$.  If $S^3\times S^3$ is a space
over $S^3$ via the first coordinate projection, $f$ is a fiberwise
map.  All maps fiberwise homotopic to $f$ have a fixed point, but
the fiberwise Nielsen number of $f$ is zero.
\end{ex}

\section{The converse to the fiberwise Lefschetz fixed point theorem}

\label{fibconvers}

In this section we describe the fiberwise generalization of Klein
and Williams' proof of the converse to the Lefschetz Fixed Point
Theorem.  The intuition and general structure here are identical
to that in Section \ref{usuconverse}.

\begin{prop}\mylabel{fibfptosec}\cite{KW,FadellHPara}
Let $M\rightarrow B$ be a continuous map. Then fiberwise
homotopies of a fiberwise map $f\colon M\rightarrow M$ to a fixed point
free map correspond to liftings which make the diagram below
commute up to fiberwise homotopy.
\[\xymatrix{&M\times_B
M-\triangle\ar[d]\\M\ar@{.>}[ur]\ar[r]^{\Gamma_f}& M\times_B M.}\]
Here $\Gamma_f$ is the graph of $f$.\end{prop}

\begin{proof}
Let $(\Top^*/B)(M,M)$\nidx{topb@$\Top/B$}\nidx{topb@$\Top^*/B$} be the
fiberwise maps $M\rightarrow M$ that are fixed point free.  Let
${\mathrm{proj}}_1\colon M\times_B M\rightarrow M$ be the first
coordinate projection. We have the following map of fibration
sequences.
\[\xymatrix{ (\Top^*/B)(M,M)\ar@{^(->}[r]\ar[d]^{\mathrm{graph}}&(\Top/B)(M,M)
\ar[d]^{\mathrm{graph}}\\
(\Top/B)(M,M\times_B M-\triangle)\ar@{^(->}[r]\ar[d]^{{\mathrm{proj}}_{1*}}&(\Top/B)
(M,M\times_B M)
\ar[d]^{{\mathrm{proj}}_{1*}}\\
(\Top/B)(M,M)\ar@{=}[r]&(\Top/B)(M,M)}\]  The fibers are taken over the
identity map.

The top square is homotopy cartesian and so the homotopy fibers of
the inclusions \[(\Top^*/B)(M,M)\rightarrow (\Top/B)(M,M)\] and
\[(\Top/B)(M,M\times_B M-\triangle)\rightarrow (\Top/B)(M,M\times_B M)\]
coincide up to homotopy.
\end{proof}

We can convert this lifting question into a question
about the existence of sections.  For a fiberwise map
$f\colon X\rightarrow Y$, let $\fibm_B(f) \colon \fibs_B(f)\rightarrow Y$
\nidx{rb@$\protect\fibm_B$}\nidx{nb@$\protect\fibs_B$} be a
Hurewicz fiberwise fibration such that
\[\xymatrix{X\ar[rr]\ar[dr]_f&&\fibs_B(f)\ar[dl]^{\fibm_B(f)}\\&Y}\]
commutes and $X\rightarrow \fibs_B(f)$is an equivalence. Liftings
up to fiberwise homotopy in a diagram \[\xymatrix{&X\ar[d]^f\\
Z\ar[r]_g\ar@{.>}[ur]&Y}\] correspond to sections of the fiberwise fibration
$g^*\fibs_B(f)\rightarrow Z$.

Suppose $\pro\colon E\rightarrow M$ is a fiberwise Hurewicz fibration over a
space $B$. The unreduced fiberwise suspension\idx{unreduced fiberwise
suspension} of $E$ over $M$ is
the double mapping cylinder
\[S_ME\coloneqq M\times \{0\}\cup_{\pro} E\times [0,1]\cup_{\pro} M\times \{1\}.\]
\nidx{sm@$S_M$} This has a map to $M$. Let
\[\sect_-,\sect_+\colon M\rightarrow S_ME\] \nidx{sigma@$\protect\sect_-$}
\nidx{simga@$\protect\sect_+$} be the sections of $S_ME\rightarrow M$
given by the inclusions of $M\times \{0\}$ and $M\times \{1\}$
into $S_ME$.

\begin{prop}
\mylabel{firstform2}
If $E\rightarrow M$ admits a section then $\sect_-$
and $\sect_+$ are homotopic over $M$.

Conversely, assume  $M\rightarrow B$ is a fibration,
$\pro\colon E\rightarrow M$ is $(r+1)$-connected, and $M$ is a cell complex
over $B$ of dimension less than or equal to $2r+1$. If $\sect_-$ and
$\sect_+$ are homotopic over $M$, then $\pro$ has a section.
\end{prop}

For the definition of a cell complex over $B$, see
\cite[24.1]{MS}.

\begin{proof} If $E\rightarrow M$ admits a section $\secttwo$ \nidx{sigma@
$\secttwo$}
this section defines a map \[S_M\secttwo\colon S_MM=M\times I\rightarrow
S_ME\] which is a homotopy over $M$ between $\sect_-$ and
$\sect_+$.

Let $X_1=M\cup_{\pro}(E\times [0,1/2])$ and $X_2=M\cup_{\pro}(E\times [1/2,1])$.
Then \[S_ME=X_1\cup_{E\times \{1/2\}}X_2.\]  Since
$p\colon E\rightarrow M$ is $(r+1)$-connected the pairs $(X_1,E)$ and $(X_2,E)$
are also $(r+1)$-connected.  By the Blakers-Massey Theorem, for any choice
of base point,
\[\pi_i(X_1,E)\rightarrow \pi_i(S_ME,X_2)\] is an isomorphism for $i<2r+2$
and a surjection for $i=2r+2$.

From the pairs $(X_1,E)$ and $(S_ME,X_2)$ we get two long exact sequences
of homotopy groups for any choice of base point in $E$
\[\ldots \rightarrow \pi_i(E)\rightarrow \pi_i(X_1)\rightarrow \pi_i(X_1,E)
\rightarrow \pi_{i-1}(E)\rightarrow \ldots\]
 \[\ldots \rightarrow \pi_i(X_2)\rightarrow \pi_i(S_ME)\rightarrow \pi_i(S_ME
,X_2)\rightarrow \pi_{i-1}(X_2)\rightarrow \ldots\]
Diagram chasing and the isomorphism from the Blakers-Massey Theorem
give an exact sequence
\[\pi_{2r+1}(E)\rightarrow \pi_{2r+1}(X_1)\oplus \pi_{2r+1}(X_2)\rightarrow
\pi_{2r+1}(S_ME)
\rightarrow \pi_{2r}(E)\rightarrow \ldots.\]
The exact sequence continues to the right but does not continue further
to the left.  Using the retractions of $X_1$ and $X_2$ to $M$ we get
an exact sequence
\[\pi_{2r+1}(E)\rightarrow \pi_{2r+1}(M)\oplus \pi_{2r+1}(M)\rightarrow
\pi_{2r+1}(S_ME)
\rightarrow \pi_{2r}(E)\rightarrow \ldots.\]

We would like to compare $E$ to the homotopy pullback of the maps
$\sect_-$ and $\sect_+$.  The homotopy pullback $P$ is the pullback in the
diagram
\[\xymatrix{P\ar[r]\ar[d]_-{\sect_-^*\fibm(\sect_+)}
&{\fibs M}\ar[d]^{\fibm(\sect_+)}
\\M\ar[r]_-{\sect_-}&S_ME}\]
where $\fibm(\sect_+)\colon \fibs M\rightarrow S_ME$ is the map $\sect_+\colon M\rightarrow S_ME$
converted into a fibration (not necessarily a fiberwise
fibration). The map ${\sect_-^*\fibm(\sect_+)}\colon 
P\rightarrow M$ is also a  fibration
with the same fiber.

For any choice of base point in $P$
we get two long exact sequences associated to these  fibrations
\[\ldots \rightarrow \pi_i(F)\rightarrow \pi_i(\fibs M)\rightarrow \pi_i(
S_ME)\rightarrow \pi_{i-1}(F)\rightarrow \ldots\]
\[\ldots \rightarrow \pi_i(F)\rightarrow \pi_i(M)\rightarrow \pi_i(
P)\rightarrow \pi_{i-1}(F)\rightarrow \ldots\] where $F$ is the
fiber of ${\fibm(\sect_+)}$.
The same diagram chase as above gives a long
exact sequence
\[\ldots \rightarrow \pi_i(P)\rightarrow \pi_i(M)\oplus\pi_i(M)\rightarrow
\pi_i(S_ME)\rightarrow \pi_{i-1}(P)\rightarrow \ldots \]

The diagram
\[\xymatrix{E\ar[r]^\pro\ar[d]_\pro&M\ar[d]^{\sect_-}\\
M\ar[r]_{\sect_+}&S_ME}\] is commutative up to preferred fiberwise
homotopy given by the homotopy from $\sect_-$ to $\sect_+$ and so there is
a map $q\colon E\rightarrow P$ such that
\[\xymatrix{E\ar[rr]^-q\ar[dr]_-p&&P\ar[ld]^-{\sect_-^*\fibm(\sect_+)}\\ &M}\]
commutes.  In particular, $q$ is a fiberwise map over $B$. By
comparing our exact sequences we see that for any choice of base
point in $E$, $q_*\colon \pi_i(E)\rightarrow \pi_i(P)$ is a bijection
for $i<2r+1$ and a surjection for $i=2r+1$.

If $[-,-]_B$\nidx{$[-,-]_B$}
denotes (unsectioned) fiberwise homotopy classes of
fiberwise maps, the fiberwise Whitehead theorem \cite[24.1.2]{MS}
and the fact that the map $E\rightarrow P$ is $(2r+1)$-connected
imply that
\[q_*\colon [M,E]_B\rightarrow [M,P]_B\] is a surjection.
The fiberwise homotopy between the sections $\sect_-$ and $\sect_+$
defines a fiberwise map $h\colon M\rightarrow P$ and so there is a
fiberwise map $\secttwo\colon M\rightarrow E$ such that $q\secttwo\colon 
M \rightarrow P$ is
fiberwise homotopic to $h$. Then
\[\mathrm{id}_M\simeq_B {\sect_-^*\fibm(\sect_+)} \circ h\simeq_B {\sect_-^*\fibm(\sect_+)}
\circ q\circ \secttwo
\simeq_B \pro\circ \secttwo\] We can use the fiberwise homotopy lifting
property of $\pro$ to deform $\secttwo$ into an actual section.

\end{proof}

We now assume that all spaces over $M$ are sectioned and all maps over $M$
preserve this section.  In particular, $S_ME$ is an ex-space over $M$ with
section $\sect_-$ and the notation $[-,-]_B$ is now used for the sectioned
fiberwise homotopy classes of maps.

There is a fiberwise
fibration replacing the inclusion
\[i\colon M\times_BM-\triangle\rightarrow M\times_BM\] that is also a
Hurewicz fibration of spaces.
Since $i$ can be
replaced with a map that is both a fibration and a fiberwise fibration
 \[[S^0_M,S_M\Gamma_{f*}(\fibs_M(M\times_BM-\triangle))]_M\] is both the fiberwise
homotopy classes of maps and the maps in
the homotopy category for a model structure, see \cite[9.1]{MS}.
This connects the
geometric description above with duality in $\Ex$.

The candidate for the fiberwise fibration replacing $i$ is the map
\[\xymatrix{
(M\times_BM-\triangle)\times_{M\times_BM}\notcalP_B(M\times_B M) \ar[r]&
M\times_BM}\]
where  $(M\times_BM-\triangle)\times_{M\times_BM}\notcalP_B(M\times_B M)$ is
\[\{( (m_1,m_2),\gamma )\in (M\times_BM-\triangle)\times \notcalP_B(M\times_B M)|
\gamma(1)=(m_1,m_2)\}\] and the map to $M\times_BM$ is given by
evaluation at 0.  For this result we use the path space rather
than the space of Moore paths.

We need a preliminary lemma.

\begin{lemma}\mylabel{fibration}
Suppose $\pro_1\colon E_1\rightarrow B$ and $\pro_2\colon E_2\rightarrow B$ are
fibrations and $f\colon E_1\rightarrow E_2$ is a map over $B$. Then the
map \[\sou\colon E_1\times_{E_2}\notcalP_BE_2=\{(e,\gamma)|\gamma(1)=f(e)\}
\rightarrow E_2\] given by $(e, \gamma)\mapsto \gamma(0)$ is a
fibration.\end{lemma}

\begin{proof}
Recall that $\notcalP_BE_2$ is the subspace of paths
in $E_2$ consisting of paths $\gamma $ such that
$p_2(\gamma(t))=p_2(\gamma(0))$ for all $t\in I$.

The space $E_1\times_{E_2}\notcalP_BE_2$ is the pullback of
\[\xymatrix{&{\notcalP}_BE_2\ar[d]^{\tar\times \sou}
\\ E_1\times_BE_2\ar[r]^{f\times \id}&E_2\times_B
E_2}\] and the map $E_1\times_{E_2}\notcalP_BE_2\rightarrow E_2$ is
the composite of \[\mathrm{id}\times \sou\colon 
E_1\times_{E_2}\notcalP_BE_2\rightarrow
E_1\times_BE_2\] and  the second coordinate projection
\[{\mathrm{proj}}_2\colon E_1\times_BE_2\rightarrow E_2.\] So it
is enough to show that both of these maps are fibrations.

The projection map is a fibration since it
is the pullback
\[\xymatrix{E_1\times_BE_2\ar[r]\ar[d]_{{\mathrm{proj}}_2}&E_1\ar[d]^{p_1}\\
E_2\ar[r]_{p_2}&B}\]
along a fibration.  \myref{fibrationforfiberproduct} shows that 
$\tar\times \sou\colon  \notcalP_B E_2
\rightarrow E_2\times_B E_2$ is a fibration.  This implies 
$E_1\times_{E_2}\notcalP_BE_2\rightarrow E_1\times_BE_2$ is a
fibration.
\end{proof}

\begin{lemma}\mylabel{bundle}
Let $p\colon M\rightarrow B$ be a fiber bundle.  Then there is a fiberwise
fibration replacing the inclusion
\[M\times_BM-\triangle \rightarrow M\times_BM\] that is also
a Hurewicz fibration.
\end{lemma}

\begin{proof}

Since $p\colon M\rightarrow B$ is a bundle there is a cover $\{U_i\}$ of
$B$ and homeomorphisms $f_i\colon p^{-1}(U_i)\rightarrow F\times U_i$.
The maps $f_i$ are maps to a product so projection to $F$ gives
maps 
\[f_{i,F}\colon p^{-1}(U_i)\rightarrow F.\]

The fiber product $q\colon M\times_BM\rightarrow B$ is locally trivial
with respect to this open cover and the trivialization
homeomorphisms are
\[f_{i,F}\times f_{i,F}\times p\colon q^{-1}(U_i)\rightarrow F\times F\times U_i\]
These homeomorphisms restrict to give a local trivialization of
\[M\times_B M-\triangle\rightarrow B\] since the maps $f_{i,F}$
are injective.  So $M\times_BM-\triangle \rightarrow B$ is a fiber
bundle with fiber $F\times F-\triangle$.  Then \myref{fibration}
completes the proof.

\end{proof}

Under the assumptions in \myref{firstform2}, the fiberwise Freudenthal
suspension theorem gives the following isomorphism:
\[[S_M^0,S_ME]_M\cong \{S_M^0,S_ME\}_M.\]
If we further assume that $M\rightarrow B$ is a space over $B$
such that $S_M^0$ is Costenoble-Waner dualizable in $\Ex_B$ with
dual $T_{M,B}$,\nidx{Tmb@$T_{M,B}$} then we have an isomorphism
\[\{S_M^0,S_ME\}_M\cong \{S^0_B,T_{M,B}\odot S_ME\}_B.\]

\begin{definition}\cite[4.3]{KW} Let $E$ and
$M$ be as in \myref{firstform2} and assume $S_M^0$ is
Costenoble-Waner dualizable in $\Ex_B$. The \emph{fiberwise stable
homotopy Euler characteristic}\idx{fiberwise stable homotopy Euler
characteristic}\idx{stable homotopy Euler characteristic!fiberwise}
of $p\colon E\rightarrow M$ is the class
\[\chi_B(p)\in \{S^0_B,T_{M,B}\odot  S_ME\}_B\]\nidx{chib@$\chi_B$}
which corresponds to the
map \[\sect_+\amalg \sect_-\colon S_M^0\rightarrow S_ME\] via the isomorphisms above.
\end{definition}

If $M$ is a space over $B$, $f\colon M\rightarrow M$ is a fiberwise map, and
\[i\colon M\times_BM-\triangle\rightarrow M\times_BM\] is the inclusion
we denote the fiberwise stable homotopy Euler characteristic
of $\Gamma_f^*\fibm_B(i)$ by $R^{KW}_B(f)$.\nidx{rkwb@$R^{KW}_B$}

\begin{corollary}\mylabel{KWfibconverse9}
Let $f\colon M\rightarrow M$ be a fiberwise map of a smooth
fiber bundle with compact manifold fibers $F$. If $\mathrm{dim}(B)\leq
\mathrm{dim}(F)-3$, the map $f$ is fiberwise homotopic to a fixed point
free map if and only if $R^{KW}_B(f)$ is trivial.
\end{corollary}

\begin{rmk}
In \cite{FadellHPara} Fadell and Husseini defined a different fiberwise invariant
using obstruction theory.   For this invariant they require that
the dimension of the fiber is at least three.\end{rmk}

\section{Identification of $R^{KW}$ with $R^{htpy}$}\label{fibkwidsec}
Let $\nmalbdl_B$\nidx{taub@$\nmalbdl_B$} be the fiberwise  normal
bundle of the inclusion of the diagonal into $M\times_B M$. Regard
the sphere bundle $S(\nmalbdl_B)$\nidx{s@$S$} and the disk bundle
$D(\nmalbdl_B)$\nidx{d@$D$} as spaces over $M\times_B M$ by
inclusion. Let $\fibs _BS(\nmalbdl_B)$ and $\fibs _BD(\nmalbdl_B)$
be the total spaces of the fiberwise fibrations corresponding to
the inclusions. Also using this notation, let $\fibs _B(M\times_B
M-\triangle)$ be the total space of the fiberwise fibration
corresponding to $M\times_B M-\triangle \rightarrow M\times_B M$.

\begin{lemma}\mylabel{fibidentify}
As an ex-space over $M\times_B M$, $S_{M\times_B M}\fibs _B(M\times_B
M-\triangle)$ is weakly equivalent to $\triangle_! S^{\nmalbdl_B}\odot
(\calP_B M,\tar\times \sou)_+$.
\end{lemma}

\begin{proof}
There is a fiberwise homotopy cocartesian square
\[\xymatrix{S(\nmalbdl_B)\ar[r]\ar[d]&D(\nmalbdl_B)\ar[d]\\
M\times_B M-\triangle\ar[r]&M\times_B M}\] This is a diagram of
inclusions over $M\times_B M$.  Replacing all of the maps to $M\times_B
M$ by fibrations, we get the following fiberwise homotopy cocartesian
square.
\[\xymatrix{\fibs _BS(\nmalbdl_B)\ar[r]\ar[d]&\fibs _BD(\nmalbdl_B)\ar[d]\\
\fibs _B(M\times_B M-\triangle)\ar[r]&M\times_B M}\]

The fiberwise homotopy cofiber \cite[II.2.1]{CrabbJames} of the bottom
arrow is \[S_{M\times_B M}\fibs _B( M\times_B M-\triangle)\] This
is weakly equivalent to the fiberwise homotopy cofiber of the top
arrow.

The top arrow is a fiberwise cofibration.  To see this, observe
that the inclusion $S(\nmalbdl_B)\rightarrow D(\nmalbdl_B)$ is a cofibration.
Pulling back
along \[\sou\colon  \calP_B(M\times_BM)\rightarrow M\times_B M\]
preserves cofibrations and in this case
converts a cofibration into a fiberwise cofibration. The
fiberwise homotopy cofiber of $\fibs _BS(\nmalbdl_B)\rightarrow \fibs _BD(\nmalbdl_B)$ is
weakly equivalent to its fiberwise cofiber. The fiberwise cofiber
is
\[\fibs_BD(\nmalbdl_B)/\sim=
\{(\gamma,u)\in \calP_B(M\times_BM)|\gamma(u)\in D(\nmalbdl_B)\}/\sim.\]
Here
$\gamma_1\sim\gamma_2$ if $\gamma_1(u_1),\gamma_2(u_2)\in S(\nmalbdl_B)$ and
$\gamma_1(0)=\gamma_2(0)$.

There is a map
\[\fibs _BD(\nmalbdl_B)\rightarrow (\calP_BM,\tar\times \sou)
\times_M(D(\nmalbdl_B),\triangle\circ p) \times_M(\calP_BM,
\tar\times \sou)\] given by
\[(\gamma_1,\gamma_2)\mapsto (\gamma_2^{-1}, (\gamma_1(u_1),\gamma_2(u_2)),
H(\gamma_1(u_1),\gamma_2(u_2))\gamma_1).\] The path
$H(\gamma_1(u_1),\gamma_2(u_2))$ is as in \myref{fiblocalcont}.
This map descends to an equivalence
\[\fibs_B D(\nmalbdl_B)/\sim \rightarrow (\calP_BM,\tar\times \sou)_+\odot
\triangle_! S^{\nmalbdl_B}\odot (\calP_BM,\tar\times \sou)_+.\]

The inclusion of $M$ into $\calP_BM$ as constant paths defines a map
\[(M,\triangle)_+\rightarrow (\calP_BM,\tar\times \sou)_+.\]  This map is 
an equivalence and so there is an equivalence
between $\fibs D(\nmalbdl_B)/\sim$ and $\triangle_!S^{\nmalbdl_B}\odot
(\calP_BM,\tar\times \sou)_+$.

\end{proof}

When we defined the stable homotopy Euler characteristic we pulled
the fibration back before taking the fiberwise suspension.  These
operations commute, so we have a weak equivalence between
$S_Mf^*\fibs (M\times_B M-\triangle)$ and $\triangle_!S^{\nmalbdl_B}\odot
\Lambda^f_{B}M_+$ where \[\Lambda^f_BM=\{(\gamma,u)\in \calP_BM
|f(\gamma(u))=\gamma(0)\}.\]  This is a space over $M$ by
$\gamma \mapsto \gamma(u).$

\begin{theorem}\mylabel{fibidentify2}
Let $M\rightarrow B$ be a smooth fiber bundle with compact
manifold fibers and $f\colon M\rightarrow M$ a fiberwise map.  Then
there is an isomorphism \[\{S_B^0, T_{M,B}\odot S_ME\}_B\cong
\{S^0_B,S^n\wedge_B\Lambda^f_{B}M_+\}_B\] and under this
isomorphism
\[R^{KW}_B(f)=R^{htpy}_B(f).\]

\end{theorem}

\begin{proof}

To define the stable homotopy Euler characteristic we used the map
\[\sect_+\amalg \sect_-\colon S^0_M\rightarrow f^*S_{M\times_B M}\fibs_B
(M\times_B M-\triangle).\]
The corresponding map $\secttwo\colon S^0_M\rightarrow
\triangle_!S^{\nmalbdl_B}\odot \Lambda^f_{B}M_+$ takes the section of $S^0_M$
to the section of $\triangle_!S^{\nmalbdl_B}\odot \Lambda^f_{B}M_+$.  On the
other copy of $M$, $\secttwo$ is defined by
\[\secttwo(m)= ((m,f(m),H(f(m),m))\] where $H$ is as in \myref{fiblocalcont}.

Let $\phi_B$ be the weak equivalence of \myref{fibidentify}.
Then the following diagram of isomorphisms commutes.
\[\xymatrix{[S^0_M,S_Mf^*\fibs _B(M\times_B M-\triangle)]_M\ar[d]_F
\ar[r]^-{\phi_{B*}}&
[S^0_M,\triangle_!S^{\nmalbdl_B}\odot \Lambda^f_{B}M_+]_M\ar[d]^F\\
\{S^0_M,S_Mf^*\fibs _B(M\times_B M-\triangle)\}_M\ar[r]^-{\phi_{B*}}\ar[d]_D&
\{S^0_M,\triangle_!S^{\nmalbdl_B}\odot \Lambda^f_{B}M_+\}_M\ar[d]^D\\
\{S^0_B,T_{M,B}\odot S_Mf^*\fibs (M\times
M-\triangle)\}_B\ar[r]_-{(\id\odot \phi_B)_*}&
\{S^0_B,T_{M,B}\odot \triangle_!S^{\nmalbdl_B}\odot \Lambda^f_{B}M_+\}_B }\]
The stabilization map $F$ is an isomorphism because of dimension
assumptions.  The map $D$ is the isomorphism from
\myref{CWfibdualmaps}.

In the top left corner we have the stable cohomotopy
Euler class and in the bottom right corner the corresponding  map is
\[\xymatrix{S^0_B\ar[r]^-\eta& T_{M,B}\odot S_M\ar[r]^-{\id\odot \secttwo}
& T_{M,B}\odot
\triangle_!S^{\nmalbdl_B}\odot \Lambda^f_{B}M_+\cong S^n{\wedge}_B
\Lambda^f_BM_+}.\] This map is the trace of a lift of $f\colon M\rightarrow M$ 
to the space of fiberwise Moore paths.
\end{proof}

\begin{proof}[Proof of \myref{reidemeister2}] By \myref{fibidentify2}
\[R^{htpy}_B(f)=R^{KW}_B(f).\]  By \myref{KWfibconverse9},
$R^{KW}_B(f)$ is zero if and only if $f$ is homotopic to a map with no
fixed points.
\end{proof}

\chapter{A review of bicategory theory}\label{reviewbicat}

In this chapter we will give 
several examples of bicategories with shadows. 
We used the bicategories in Sections \ref{catex2} and \ref{catex4}
earlier.
The other 
examples in this section are not necessary for what came earlier,
but they may be helpful as an alternative source of motivation.

The first example, in Section \ref{catex1}, can be interpreted as
a generalization of the bicategory of rings, bimodules, and
homomorphisms to a symmetric monoidal category that is not the
category of abelian groups.  

The example Section \ref{catex2} is also a
generalization of the bicategory of rings, bimodules, and
homomorphisms and of the example in Section \ref{catex1}.  
It is an enriched version
of the bicategory of categories, bimodules, and natural
transformations.  In this context a bimodule is an enriched functor 
\[\sA\otimes \sB^{op}\rightarrow \sV\] where $\sV$ is a symmetric monoidal
category and $\sA$ and $\sB$ are categories enriched in $\sV$.
This is
the `many object' version of the bicategory of rings, bimodules, and 
homomorphisms and was used to define the
unbased algebraic Reidemeister trace.

The example of Section \ref{catex3} 
is the internal version of the bicategory of categories, 
bimodules, and natural transformations. The last example comes back
to the definitions of Section \ref{catex1} and is a generalization 
of these definitions from symmetric monoidal categories to  bicategories.

\section{Bicategory of enriched monoids, bimodules, and maps}\label{catex1}

In this and the following two sections we will describe
three bicategories that can be defined from a symmetric monoidal
category.  These bicategories all have shadows and any
symmetric monoidal functor induces a lax functor of the associated
bicategories that is compatible with shadows.

For this section and the two that follow let
$\sV$\nidx{v@$\protect\sV$} be a symmetric monoidal category with
unit object $I$,\nidx{i@$I$} tensor product $\otimes$,
\nidx{$\otimes$} and symmetry isomorphism $\gamma$.
\nidx{gamma@$\gamma$} The category $\sV$ must also have all colimits.

A \emph{monoid}\idx{monoid} in $\sV$ is an  object $A$ in $\sV$
with  maps $\mu\colon A\otimes A\rightarrow A$\nidx{mu@$\mu$} and $\iota\colon 
I\rightarrow A$\nidx{iota@$\iota$} which are unital and associative. An
\emph{ $A$-$B$-bimodule}\idx{bimodule} is an object  $X$ in $\sV$
with a pair of maps $\kappa_A\colon A\otimes X\rightarrow
X$\nidx{kappa@$\kappa$} and $\kappa_B\colon X\otimes B\rightarrow X$ that are
unital and associative with respect to the maps $\mu$ and $\iota$
for $A$ and $B$. We also require that \[\kappa_B(\kappa_A\otimes
\id_B)=\kappa_A(\id_A \otimes \kappa_B).\]

Let $X$ and $Y$ be $A$-$B$-bimodules.  A map $f\colon X\rightarrow Y$ in $\sV$ is
a \emph{map of bimodules}\idx{map of bimodules} if the following diagram
\[\xymatrix{A\otimes X\ar[r]^-{\kappa_A}\ar[d]_{\id\otimes f}&X\ar[d]^f\\
A\otimes Y\ar[r]^-{\kappa_A}&Y}\] and the corresponding diagram for
the map $\kappa_B$ commute.

\begin{definition}\mylabel{monoidsodot} Let $X$ be an $A$-$B$-bimodule
and $Y$ a $B$-$C$-bimodule. Then $X\odot Y$,\nidx{$\odot$}
an $A$-$C$-bimodule, is the
coequalizer
\[\xymatrix{X\otimes B\otimes
Y\ar@<.5ex>[r]^-{\kappa_B\otimes \id}\ar@<-.5ex>[r]_-{\id\otimes
\kappa_B}&X\otimes Y\ar[r]^-{\pi}&X\odot Y.}\]  The left module
structure is induced by the map $\kappa_A\otimes \id$ and the right
module structure is induced by $\id \otimes \kappa_C$.\end{definition}

Define a bicategory $\sNV$\nidx{nv@$\protect\sNV$}
with 0-cells the monoids in $\sV$ and
$\sNV(A,B)$ the category of $B$-$A$-bimodules and bimodule maps. The
$\odot$ product is described in \myref{monoidsodot}, and
the unit functor associated to a monoid $A$ is that monoid
considered as an $A$-$A$-bimodule using the monoid multiplication
$\mu$.\footnote{An alternative bicategory composition and shadows
for these 0-cells, 1-cells, and 2-cells is described  in
\cite{Nicas}.  This is related to the homotopy colimits we used
earlier.}

Duality in this bicategory was considered in \cite[2.1]{Pareigis5}.

The bicategory $\sNV$ is symmetric with involution $t$.  For
a monoid $A$, $tA$ is $A$ as an object of $\sV$ and has multiplication
given by
\[\xymatrix{A\otimes A\ar[r]^{\gamma}&A\otimes A\ar[r]^-{\mu}&A.}\]
For a module $X$, $tX$ is $X$ as an object of $\sV$ and the  multiplication
is given by
\[\xymatrix{A\otimes X\ar[r]^{\gamma}&X\otimes A\ar[r]^-{\kappa_A}&X.}\]

\begin{definition}For $Z\in\sNV(A, A)$ the \emph{shadow}\idx{shadow} of $Z$,
$\sh{Z}$,\nidx{shad@$\protect\sh{-}$} is the coequalizer
\[\xymatrix{A\otimes Z\ar@<.5ex>[r]^-{\kappa_A}\ar@<-.5ex>[r]_-{\kappa_A\gamma}
&Z\ar[r]&\sh{Z}}\]
\end{definition}
The target of the shadow is the category $\sV$.  This is also
the category of $I$-$I$-bimodules where $I$ is
regarded as monoid using the unit isomorphism.

Let $X$ be an $A$-$B$-bimodule and $Y$  a $B$-$A$-bimodule. Noting
that \[A\otimes (X\odot Y)\cong (A\otimes X)\odot Y,\] the natural
transformations $\theta_{B,A}$ are induced by:
\[\xymatrix{A\otimes X\odot Y\ar@<.5ex>[r]\ar@<-.5ex>[r]&X\odot Y
\ar[r]\ar[d]&\sh{X\odot Y}\ar@{.>}[dl]^{\theta_{B,A}}\\
&\sh{Y\odot X}}\] where the map $X\odot Y\rightarrow
\sh{Y\odot X}$ is induced by the composite
\[X\otimes Y\rightarrow Y\otimes X\rightarrow Y\odot X
\rightarrow \sh{Y\odot X}.\]\nidx{theta@$\theta$}

Let $\Ab$ be the category of abelian groups.   The monoids in $\Ab$ are the
associative rings with unit and $\odot$ is the usual tensor
product over a ring.  The bicategory $\sN_{\Ab}$ is $\Mod$, the
bicategory of rings, bimodules, and homomorphisms. The bicategory
 $\sN_{\Mod_R}$ is the
bicategory with 0-cells algebras over $R$, 1-cells bimodules and
2-cells homomorphisms.

Let $\Ch$ be the category of chain complexes of abelian groups and
chain maps.  A ring $R$, thought of as a chain complex
concentrated in degree 0, is a monoid in $\Ch$. Then
$\sN_{\Ch}(\mathbb{Z},R)$ is the category of chain complexes of
left $R$-modules and chain maps. The functor $\odot$ is the usual
tensor product of chain complexes over $R$.

\begin{definition}
Let $F\colon \sV\rightarrow \sU$ be a lax  monoidal functor.
Define a lax functor of bicategories $\sNV\rightarrow \sNU$ as follows.
\begin{enumerate}
\item If $A$ is a monoid with composition $\mu\colon A\otimes
A\rightarrow A$ and unit $\iota\colon I\rightarrow A$ then $FA$ is a
monoid with composition \[\xymatrix{FA\otimes FA\ar[r]^\symf&
F(A\otimes A)\ar[r]^-{F(\mu)}&F}\]
and unit
\[\xymatrix{I'\ar[r]&F(I)\ar[r]^{F(\iota)}&F(A)}.\]
\item
If $X$ is an $A$-$B$-bimodule with action $\kappa_A\colon A\otimes
X\rightarrow X$ by $A$ and action $\kappa_B\colon X\otimes B\rightarrow X$ by $B$
then $FX$
is an $FA$-$FB$-bimodule with action $F(\kappa_A)\symf$ by $F(A)$ and
action $F(\kappa_B)\symf$ by $F(B)$.
\item The natural transformations $\phi_{X,Y}$
are induced by $\symf$. \item The natural transformations $\phi_A$
are the identity.
\end{enumerate}\end{definition}

\begin{lemma}The lax functor induced by a lax symmetric monoidal functor
 is compatible with shadows.\end{lemma}
\begin{proof}The natural map
$\sh{FX}\rightarrow F\sh{X}$ is defined
using the following diagram.
\[\xymatrix@C=50pt{FX\otimes FA\ar@<.5ex>[r]^-{F(\kappa_A)\symf}
\ar@<-.5ex>[r]_-{F(\kappa_A)\symf \gamma'}\ar[d]_{\symf}&FX
\ar[r]\ar@{=}[d]&\sh{FX}\ar@{.>}[d]\\
F(X\otimes
A)\ar@<.5ex>[r]^-{F(\kappa_A)}\ar@<-.5ex>[r]_-{F(\kappa_A)\gamma}
&FX\ar[r]&F\sh{X}}\]
\end{proof}

\section{Bicategory of enriched categories, bimodules, and maps}\label{catex2}

In some ways the examples in this section and the following
section are modeled on the 2-category of categories, functors,
and natural transformations, but this is not the most helpful
motivation.  It is much more useful to think about the bicategory
of rings, bimodules, and homomorphisms and view these bicategories
as ``many object'' generalizations.

A category $\sA$ is  \emph{enriched}\idx{enriched category}
\idx{category!enriched} in $\sV$ if for each $a,b\in \ob(\sA)$,
$\sA(a,b)\in \ob(\sV)$ and for $a,b,c\in \ob(\sA)$ composition
\[\sA(b,c)\otimes \sA(a,b)\rightarrow \sA(a,c)\] is a map in $\sV$.
We also require that for  each $a\in \ob(\sA)$, there is a map $I
\rightarrow \sA(a,a)$ in $\sV$ and these maps and the composition maps
satisfy unit and associativity conditions. For more details see
\cite[p. 23]{Kelly}.

For two enriched categories $\sA$ and $\sB$, an \emph{enriched
$\sA$-$\sB$-bimodule}
\idx{enriched bimodule}\idx{bimodule!enriched}
is an enriched functor $\sX\colon \sA\otimes \sB^{op}\rightarrow \sV$.  It
consists of an object $\sX(a,b)$ in
$\sV$ for each $a\in \ob(\sA)$, $b\in\ob(\sB)$ and maps
\[\kappa\colon \sA(a,a')\otimes \sB(b,b')\rightarrow \sV(\sX(a,b'),\sX(a',b))\] in $\sV$.
These maps must be compatible with composition and the unit
objects in $\sA$ and $\sB$. Functors of this form are
sometimes called distributors.\idx{distributors}

As discussed in the beginning of this section, these objects
should be thought of as many object generalizations of monoids and
bimodules.  This is reflected in the use of the term bimodule here and
is compatible with the previous use in the sense that   a category
with one object enriched in $\sV$ is a monoid in $\sV$ and a
bimodule (enriched functor) from a pair of categories each with one
object is a bimodule (1-cell in $\sNV$).

Let $\sX,\sY$ be $\sA$-$\sB$-bimodules.  An \emph{enriched natural
transformation}\idx{enriched natural transformation}\idx{natural
transformation!enriched} is a collection of morphisms in $\sV$
\[\{\delta_{a,b}\colon \sX(a,b)\rightarrow \sY(a,b)\in {\mathrm{Mor}} \,\sV\}_{a
\in \ob(\sA), b\in\ob(\sB)}\]
such that the diagram below commutes for all $a,a'\in \ob(\sA)$
and $b,b'\in\ob(\sB)$.
\[\xymatrix{{\sA(a,a')\otimes \sB(b,b')}\ar[r]^-\kappa
\ar[d]_\kappa&{\sV(\sX(a,b'),\sX(a',b))}
\ar[d]^-{\sV(\id,
\delta_{a',b})}\\
{\sV(\sY(a,b'),\sY(a',b))}\ar[r]_-{\sV(\delta_{a,b'},\id)}&{\sV(\sX(a,b'),
\sY(a',b))}}\]

\begin{definition}\mylabel{functortensor}
Let  $\sX$ be an $\sA$-$\sB$ bimodule and $\sY$  a $\sB$-$\sC$
bimodule.  Then $\sX\odot \sY$\nidx{$\odot$}
is the $\sA$-$\sC$ bimodule defined
by the coequalizer.

\[\xymatrix@C=70pt{{\displaystyle\coprod_{b,b'\in \ob(\sB)}\sX(a,b')\otimes
\sB(b,b')
\otimes \sY(b,c)}\ar@<.5ex>[r]^-{
\coprod \kappa_\sB
\otimes \id}
\ar@<-.5ex>[r]_-{\coprod \id\otimes \kappa_\sB}& {\displaystyle
\coprod_{b\in \ob(\sB)}
\sX(a,b)\otimes \sY(b,c)}.}\]

The $\sA$-$\sC$-bimodule structure is induced by the map
\[\xymatrix@R=5pt{{\sA(a,a')\otimes \sC(c,c')}\ar[r]^-\kappa
&{\sV(\sX(a,b)\otimes \sY(b,c'),
\sX(a',b)\otimes \sY(b,c))}\\
\,\ar[r]&{ \sV(\sX(a,b)\otimes
\sY(b,c'), \sX\odot \sY(a,c'))}}\]

\end{definition}

This is the usual tensor
product of functors\idx{tensor product of functors}
except the first coproduct is indexed over pairs of objects
rather than morphisms
since $\sA(a,a')$ is an object in $\sV$ rather than a set.

We define  a bicategory $\sEV$ with 0-cells categories enriched in
$\sV$,\nidx{ev@$\protect\sEV$}
1-cells enriched bimodules, and 2-cells enriched natural
transformations.  The $\odot$ product for $\sEV$ is given in
\myref{functortensor}. For any category $\sA$ enriched
over $\sV$  define an $\sA$-$\sA$-bimodule by
$U_{\sA}(a,a')=\sA(a',a)\in \sV$.  The associativity of
composition gives a single map
\[\sA(a,a')\otimes \sA(b',a)\otimes \sA(b,b')\rightarrow
\sA(b,a').\]  Using symmetry and the $\sV(-,-)-\otimes$
adjunction, we get a map \[\sA(a,a')\otimes \sA(b,b')\rightarrow
\sV(\sA(b',a), \sA(b,a')).\]  The target of this map is
$\sV(U_{\sA}(a,b')U_{\sA}(a',b))$ and this gives
the required action.  This is the unit functor for $\sEV$.

The bicategory $\sEV$ is symmetric with involution $t$ which takes a category
$\sA$ to its opposite. For $\sX\colon \sA\otimes \sB^{op}\rightarrow\sV$,
$t\sX\colon \sB^{op}\otimes (\sA^{op})^{op}\rightarrow \sV$ is defined by $t\sX(b,a)
=\sX(a,b)$. The map
\[\sB^{op}(b,b')\otimes \sA^{op}(a,a')\rightarrow \sV(t\sX(b,a'),t\sX(b',a))\]
is the symmetry map $\gamma$  in $\sV$ composed with the action of $\sA$
and $\sB$ on $\sV$.  

This bicategory  also has shadows.
\begin{definition}\cite[2.2.5]{coufal}
Let $\sZ\colon \sA\otimes \sA^{op}\rightarrow \sV$ be an
$\sA$-$\sA$-bimodule. Define the \emph{shadow}\idx{shadow} of
$\sZ$, $\sh{\sZ}$,\nidx{shad@$\protect\sh{-}$} to be the coequalizer
of the pair of maps 
\[\xymatrix@C=70pt{{\displaystyle\coprod_{a,a'\in \sA} \sA(a,a')\otimes
\sZ(a,a')}
\ar@<.5ex>[r]^-{\coprod
\kappa_\sA}
\ar@<-.5ex>[r]_-{\coprod(\kappa_\sA\circ \gamma)}
&{\displaystyle\coprod_{a\in \sA}\sZ(a,a)}
.}\]
\end{definition}

The target of the shadow is the category $\sV$.  This is the category of 
modules over the category with one object and with
morphisms object  $I\in \sV$.  A composition of maps similar to those used in
$\sNV$ induces the maps $\theta_{\sA,\sB}$ in $\sE_\sV$.

The bicategory $\Mod$ of rings, bimodules, and homomorphisms is the
bicategory $\sN_{\Ab}$.   The corresponding bicategory $\sE_{\Ab}$
is less familiar.  We can think of the 0-cells in $\sE_{\Ab}$ as
rings with many objects and the 1-cells as modules with many
objects.  For an enriched category the abelian group structure on
the hom sets gives the ``addition'' and the category composition
gives the ``multiplication''. Similarly, we think of an enriched
functor with target $\Ab$ as a module with many objects.  The
addition comes from the hom sets and the action of the category
corresponds to the action of the ring.

Let $\sV, \sU$ be  symmetric monoidal categories and
$F\colon \sV\rightarrow \sU$ be a lax symmetric monoidal functor. For a
category  $\sA$ enriched in $\sV$ define a category $F(\sA)$
enriched in $\sU$ with the same objects as $\sA$ and with
morphisms defined by \[F(\sA)(a,a')=F(\sA(a,a')).\] Composition in
$F(\sA)$ is defined by \[\xymatrix{F(\sA)(a',a'')\otimes
F(\sA)(a,a')\ar@{=}[d]
& F(\sA(a,a''))\\
F(\sA(a',a''))\otimes F
(\sA(a,a'))\ar[r]^-{\symf}&F(\sA(a',a'')\otimes
\sA(a,a'))\ar[u].}\]

Let $\sX$ be an $\sA$-$\sB$-bimodule.  Then $F(\sX)$ is the
$F(\sA)$-$F(\sB)$-bimodule defined by $F(\sX)(a,b)=F(\sX(a,b))$.
The action of $F(\sA)$ and $F(\sB)$ is given by
\[\xymatrix{F(\sA)(a,a')\otimes F(\sB)(b,b')\ar[d]^\symf& { \sV(F(\sX(a,b')),
F(\sX(a',b)))}\\
 F(\sA(a,a')\otimes \sB(b,b'))
\ar[r]^-{F(\kappa)}& F(\sV(\sX(a,b'),\sX(a',b)))\ar[u]} \]

\begin{definition} Given a symmetric monoidal functor $F\colon \sV\rightarrow
\sU$
define a lax functor of bicategories $\sEV\rightarrow \sEU$ as follows.
\begin{enumerate}\item The function $\ob \sEV\rightarrow \ob\sEU$ is
given by $\sA\mapsto F(\sA)$.
\item The functor $\sEV(\sA,\sB)\rightarrow \sEU(F(\sA),
F(\sB))$ is given by composition with $F$.
\item The natural transformation $\phi_{X,Y}$ is induced by $\symf$
via the composites \[F (\sX(a,b))\otimes F
(\sY(b,c))\stackrel{\symf}{\rightarrow} F(\sX(a,b)\otimes
\sY(b,c))
\rightarrow 
F(\sX\odot \sY)(a,c).\]
\item The natural transformation $\phi_A$ is the identity
 $\id\colon t{F(\sA)}
\rightarrow F(t{\sA})$.
\end{enumerate}
\end{definition}

\begin{lemma}A lax functor $\sEV\rightarrow \sEU$ induced
 by a symmetric monoidal functor $F\colon \sV\rightarrow \sU$ is
 compatible with shadows.\end{lemma}

\begin{proof}The natural transformations $\psi_\sA\colon \sh{F(\sX)}\rightarrow
F\sh{\sX}$ are induced by the composites $F \sX(a,a)\rightarrow F(\coprod
\sX(a',a'))\rightarrow F\sh{\sX}$.
\end{proof}

Lax functors of bicategories give one way of comparing dual pairs,
but as in \myref{dualcomposites1}, composites of dual pairs
give another way. In this bicategory there is one example of
a composite of dual pairs that is particularly relevant.

Let $\Phi\colon \sC \rightarrow \sA$ be an enriched functor. If $\sX$ is
a $\sB$-$\sA$-bimodule let $\sX_\Phi$\nidx{xphi@$\protect\sX_\Phi$}
be the $\sB$-$\sC$-bimodule
defined by $\sX_\Phi (b,c)=\sX(b,\Phi(c))$.  If $\sY$ is an
$\sA$-$\sB$-bimodule let $_\Phi\sY$\nidx{yphi@$_\Phi\protect\sY$}
be the $\sC$-$\sB$-bimodule defined
by $_\Phi\sY(c,b)=\sY(\Phi(c),b)$. In particular, $\Phi$ can be
used to define an $\sA$-$\sC$-bimodule ${U_{\sA}}_\Phi$ and a
$\sC$-$\sA$-bimodule $_\Phi U_{\sA}$.
These  bimodules form a dual pair with coevaluation
\[\eta\colon U_{\sC}\rightarrow \,_\Phi U_{\sA}\odot
{U_{\sA}}_\Phi\] induced by the functor $\Phi$
and evaluation \[\epsilon\colon {U_{\sA}}_\Phi\odot \,_\Phi U_{\sA}
\rightarrow U_{\sA}\]  which is composition in $\sA$.

\begin{lemma}\mylabel{enrichedadjunction} Let $\sV$ be a
symmetric monoidal category and $\sA$, $\sB$, and $\sC$ be
categories enriched in $\sV$. Suppose $\Phi\colon \sC\rightarrow \sA$ is
an enriched functor and $\sX$ is a $\sB$-$\sC$-bimodule with dual
$\sY$ and coevaluation and evaluation maps
\[\xymatrix{{U_{\sB}}\ar[r]^-\chi &{\sX\odot \sY}
&{\mathrm{and}}&{\sY\odot  \sX}\ar[r]^-\theta &{U_{\sC}}.}\]

\begin{enumerate}

\item If $\Phi$ has right adjoint $\Psi\colon \sA\rightarrow \sC$
then $\sX_\Phi$ is dualizable with dual ${U_{\sA}}_\Phi\odot \sY$.
\item If $\Psi$ is also a left adjoint for $\Phi$, then $\sX_\Phi$ is dualizable
with dual $_\Psi\sY$.

\item Suppose $\Psi$ is a both a left and right adjoint for $\Phi$.
Given a 2-cell \[f\colon \sY\odot \sQ\rightarrow \sP\odot \sY\] let $f'$ be the 2-cell
\[\xymatrix{{_\Psi\sY\odot \sQ}\ar[r]^-{f(\Psi-,-)}&{_\Psi\sP\odot \sY}
\ar[d]^\cong & {_\Psi\sP_\Psi\odot \,_\Psi\sY}
\\
&{_\Psi\sP\odot
U_{\sC}\odot \sY}\ar[r]^-{\id\odot \eta \odot \id} &{_\Psi\sP\odot \,_\Phi U_{
\sA}\odot {U_{\sA}}_\Phi\odot \sY}.\ar[u]^\cong
}\]
Then
\[\xymatrix{{\sh{\sQ}}\ar[d]^-{tr(f)}&&&{\sh{_\Psi\sP_\Psi}}\\
{\sh{\sP}}\ar[r]^-\cong
&{\sh{\sP\odot U_{\sC}}}\ar[r]^-{\sh{\id\odot \eta}}&
{\sh{\sP\odot \,_\Phi U_{\sA}\odot
{U_{\sA}}_\Phi}}\ar[r]^-\cong &{\sh{{U_{\sA}}_\Phi\odot \sP\odot
\,_\Phi U_{\sA}}}\ar[u]^-\cong }\] is the trace
of $f'$.
\end{enumerate}
\end{lemma}

In \myref{covtransdual} we proved a special case of this result.  The
following corollary of \myref{enrichedadjunction} generalizes \myref{covtransdual}.

\begin{corollary}\mylabel{ectonc} Let $\Pi$ be a connected groupoid enriched
in $\sV$ and $\sX\colon \Pi\rightarrow \sV$ a right $\Pi$-module. If
$\sX(x)$ is dualizable as a right  $\Pi(x,x)$-module in $\sNV$ for
some $x\in \Pi$ then $\sX$ is dualizable in $\sEV $.

If $\sP$ is a $\Pi$-$\Pi$-bimodule, $f\colon \sX\rightarrow \sX\odot\sP$
is a map of right modules, and $f_x\colon \sX(x)\rightarrow (\sX\odot
\sP)(x)$ the map $f$ restricted to $x\in \ob\Pi$, then
$\tr(f)=\tr(f_x)$.\end{corollary}

\section{Internal and topological bicategories}\label{catex3}

The bicategory $\sEV$ is particularly well adapted to examples
where the symmetric monoidal category is abelian groups, modules
over a commutative ring, or chain complexes of modules over a
commutative ring.  It works less well with the category of
topological spaces. In particular, we would rather have a space of
objects than just a set of objects.

One way to
resolve this is to replace enriched categories by internal
categories.  Unfortunately this change does not give the duality
theory we used in Chapters \ref{classfpsec} and \ref{fibfpsec}.
The problem is that the base points are missing.  In the next
section we will define a bicategory that gives the duality of
Chapters \ref{classfpsec} and \ref{fibfpsec}.  The bicategory
described in this section can be thought of as a transition
between the bicategory of enriched categories, bimodules, and
homomorphisms, and the bicategory of monoids in Section
\ref{catex4}.

We now require that $\sV$ be a Cartesian monoidal category with
unit object $\ast$\nidx{ast@$\ast$} and product
$\times$\nidx{$\times$}. This assumption is important since we
will need to use pullbacks. If $B\rightarrow {A}$ and $
C\rightarrow {A}$ are morphisms in $\sV$, the object $B \times_{A}
C$ in $\sV$ is the pullback
\[\xymatrix{B\times_{A} C\ar[r]\ar[d]&B\ar[d]\\ C\ar[r]&{A}.}\]

\begin{definition} An \emph{internal category}\idx{internal category}
\idx{category!internal} in $\sV$ is a pair of objects, $A$ and
$\sA$, in $\sV$ and maps \nidx{iota@$\iota$}
\[\xymatrix{{\sA}\ar@<.5ex>[r]^\sou\ar@<-.5ex>[r]_\tar
&A\ar@/_1pc/[l]_\iota}\] such that $\sou\circ
\iota=\tar\circ \iota=\id$.  We also require  a map
$\mu\colon \sA\times_A\sA \rightarrow \sA$\nidx{mu@$\mu$} over $A\times A$
such that
\[\xymatrix{{\sA=A\times_{A}\sA} \ar[r]^-{\iota\times
\id}&{\sA\times_{A}\sA}\ar[r]^-\mu&{\sA}}\] and
\[\xymatrix{{\sA=\sA\times_{A}{A}}
\ar[r]^-{\id\times \iota}&{\sA\times_{ A}\sA}\ar[r]^-\mu&{\sA}}\]
are the identity map of $\sA$ and
\[\xymatrix{{\sA\times_{ A}\sA\times_{ A}\sA}\ar[r]^-{\mu\times
\id}\ar[d]_{\id\times \mu}&{\sA\times_{ A}\sA}\ar[d]^\mu\\
{\sA\times_{ A}\sA}\ar[r]_-\mu&{\sA}}\] commutes.
\end{definition}

In this definition we use the convention that \[\sA\times_A\sA\coloneqq
\{(a,a')\in \sA\times \sA|\sou(a)=\tar(a')\}\]  This is the convention
used in Chapters \ref{classfpsec} and \ref{fibfpsec}.

An internal category in $\sV$ will be abbreviated $\sA$.  Think of
$A$ as the objects of the category and $\sA$ as the morphisms.
Then $\sou$ and $\tar$ are the source and target maps, $\iota$ is
the identity and $\mu$ is composition. If $\sV$ is the category of
topological spaces, and  $\sA$ is an internal category with a
discrete object space $A$, the internal category $\sA$ is a
category enriched in topological spaces.

\begin{definition} If $\sA$ and $\sB$ are two internal categories,
an \emph{internal $\sA$-$\sB$-bimodule}\idx{internal bimodule}
\idx{bimodule!internal} is an object $ \sX$ of $\sV$, with maps
$\pro_A\colon  \sX\rightarrow A$, $\pro_B\colon \sX\rightarrow B$  in $\sV$, a
map $\kappa_A\colon \sA\times_A \sX \rightarrow   \sX$\nidx{kappa@$\kappa$} in
$\sV$ over $ A$, and a map $\kappa_B\colon \sX\times_B\sB\rightarrow
\sX$ over $B$ which are unital and associative with respect to the
maps $\mu$ and $\iota$ of $\sA$ and $\sB$ and such that
\[\xymatrix{{\sA}\times_A \sX\times_B\sB\ar[r]^-{\kappa_A\times
\id}\ar[d]_{\id\times\kappa_B}&
{\sX}\times_B\sB\ar[d]^{\kappa_B}\\
{\sA}\times_A\sX\ar[r]_{\kappa_A}&{\sX}}\]
commutes.
\end{definition}

A 1-cell in $\sEV$, a bimodule, is a collection of objects in
$\sV$ and ways of relating them.  An internal bimodule is similar
except all of the objects have been collected into a single object
of $\sV$.

\begin{definition}Let $ \sX$ be an $\sA$-$\sB$-bimodule and $ \sY$
be a $\sB$-$\sC$-bimodule. Then $ \sX\odot  \sY$\nidx{$\odot$} is
the $\sA$-$\sC$-bimodule defined by the coequalizer
\[\xymatrix{ {\sX\times_{B} \sB\times_{B} \sY}\ar@<.5ex>[r]^-{\kappa\times
\id}\ar@<-.5ex>[r]_-{\id\times \kappa} & {\sX\times_{B}
\sY}\ar[r]&{\sX \odot \sY}.}\]  \end{definition}

This is the $\odot$ product in a bicategory we will call
$\sIV$.\nidx{iv@$\protect\sIV$} The 0-cells of $\sIV$ are internal
categories and the 1-cells are internal bimodules.  The 2-cells
are called internal natural transformations and they are similar
to enriched natural transformations.

\begin{definition} An \emph{internal natural
transformation}\idx{internal natural transformation}\idx{natural
transformation!internal} from an $\sA$-$\sB$-bimodule $\sX$ to an
$\sA$-$\sB$-bimodule $\sY$ is a morphism $f\colon  \sX\rightarrow  \sY$
in $\sV$ over $ A\times B$, such that
\[\xymatrix{{\sA\times_{A} \sX}\ar[r]^-\kappa\ar[d]_{\id\times f}&
{ \sX}\ar[d]^f\\{\sA\times_{ A} \sY}\ar[r]_-{\kappa'}& {\sY}}\]
and the corresponding diagram for the action by $\sB$ commute.\end{definition}

The unit in this bicategory associates to an internal category $
\sA$ the bimodule $(\tar\times \sou)\triangle\colon  \sA\rightarrow {
A}\times { A}$. The required diagrams commute since corresponding
diagrams are part of the definition of a category.  This gives a
simple example of an $\sA$-$\sA$-bimodule.

\begin{definition}Let $ \pro\times \pro'\colon \sZ{\rightarrow}
{ A}\times { A}$ be an $\sA$-$\sA$-bimodule. Define the
\emph{shadow}\idx{shadow} of $\sZ$,  $\sh{ \sZ}$,
\nidx{shad@$\protect\sh{-}$} by the coequalizer diagram
\[\xymatrix{{\sA\times_{ A\times { A}} \sZ}
\ar@<.5ex>[r]^-{\kappa}\ar@<-.5ex>[r]_-{\kappa\circ\gamma}&
{ \sZ_{{ A},{ A}}}\ar[r]&{\sh{ \sZ}}}\]
where $ \sZ_{{ A},{ A}}$ and $\sA
\times_{{ A}\times { A}} \sZ$
are the equalizers
\[\xymatrix{ {\sZ_{ A, A}}\ar[r]& {\sZ}
\ar@<.5ex>[r]^\pro\ar@<-.5ex>[r]_{\pro'}&{ A}}\] and
\[\xymatrix{{\sA\times_{{ A}\times { A}} \sZ}\ar[r]&{\sA\times  \sZ}
\ar@<.5ex>[r]^{\tar\times \pro} \ar@<-.5ex>[r]_{\sou\times \pro'}&{ A}\times
{ A}}\]
\end{definition}

The target of the shadow functor is the category $\sV$.  This
category is the category of $I$-$I$-bimodules. Let $\sX$ be an
$\sB$-$\sA$-bimodule and $\sY$ be an $\sA$-$\sB$-bimodule.  The
morphisms $\theta$ are  induced by 
\[( \sX\times_{ A} \sY)_{ B, B}
\rightarrow ( \sY\times_{ B} \sX)_{ A, A}\rightarrow ( \sY\odot
 \sX)_{ A, A}\rightarrow \sh{ \sY\odot \sX}.\]

\section{Bicategory of bicategorical monoids}\label{catex4}
In this section we generalize both the bicategory from Section
\ref{catex1} and the bicategory of monoids, bimodules, and maps in
$\Ex$ and $\Ex_B$ from Chapters \ref{classfpsec} and \ref{fibfpsec}
to any bicategory.

\begin{rmk} There is one important difference between this section and
Chapters \ref{classfpsec}, \ref{classfpsec2}, \ref{fibfpsec}, and
\ref{fibfpsec2}.  In Chapters \ref{classfpsec}, \ref{classfpsec2},
\ref{fibfpsec}, and \ref{fibfpsec2} we used homotopy colimits.  In
this section we will use colimits.  Despite this difference, many
of the results in this section have analogues in Chapters
\ref{classfpsec}, \ref{classfpsec2}, \ref{fibfpsec}, and
\ref{fibfpsec2}.
\end{rmk}

This bicategory includes many of the best features of the three
previous examples and additional elements that are necessary for
topological applications.  As we saw in Chapters \ref{classfpsec}
and \ref{fibfpsec} this bicategory eliminates the need to choose
base points but retains the structure necessary to define
topological duality.

In this section $\sW$\nidx{w@$\protect\sW$} is a bicategory with bicategory
composition $\boxtimes$.  The hom categories of $\sW$ must have
all coequalizers.

\begin{definition} Let $A$ be a 0-cell in $\sW$. A \emph{monoid}
\idx{monoid} in $\sW$ is a 1-cell $\sA\in\sW(A,A)$ with 2-cells
\[\xymatrix{U_A\ar[r]^-\iota &{\sA}&{\mathrm{ and}}
&{\sA\boxtimes \sA}\ar[r]_-\mu& {\sA}}\]\nidx{iota@$\iota$}
\nidx{mu@$\mu$} which are unital and associative.  That is,
\[\xymatrix{{\sA\cong U_A\boxtimes \sA} \ar[r]^-{\iota\boxtimes \id}
&{\sA\boxtimes \sA}\ar[r]^-\mu&{\sA}}\] and
\[\xymatrix{{\sA\cong \sA\boxtimes U_A}\ar[r]^-{\id
\boxtimes \iota}&{\sA\boxtimes \sA}\ar[r]^-\mu&{\sA}}\] are the
identity map of $\sA$ and
\[\xymatrix{{\sA\boxtimes \sA\boxtimes \sA}\ar[r]^-{\mu\boxtimes
\id}\ar[d]^{\id\boxtimes \mu}&{\sA\boxtimes \sA}\ar[d]^\mu\\
{\sA\boxtimes \sA}\ar[r]^-\mu&{\sA}}\] commutes.  We call $\iota$ the \emph{unit} and
$\mu$ the \emph{composition}.
\end{definition}

The simplest example of a monoid is the 1-cell $U_A$ for a 0-cell
$A$ in $\sW$.  The unit map is the identity and the
composition is the unit isomorphism.

\begin{definition} Let $\sA$ and $\sB$ be monoids in $\sW$. An
\emph{$\sA$-$\sB$-bimodule}\idx{bimodule}
in $\sW$ is a 1-cell $\sX\in \sW(B,A)$
and two 2-cells \nidx{kappa@$\kappa$} 
\[\kappa_A\colon \sA\boxtimes \sX\rightarrow \sX\] and
\[\kappa_B\colon \sX\boxtimes \sB\rightarrow \sX\]
that are unital and associative with respect to the monoid
structure of $\sA$ and $\sB$. We also require that the actions
$\kappa_A$ and $\kappa_B$ commute.
\end{definition}

Any 1-cell $X$ in $\sW(B,A)$ is a $U_{A}$-$U_{B}$-bimodule.  The
monoids $U_A$ and $U_{B}$ act by the unit isomorphisms. Any
monoid $\sA$ is an $\sA$-$\sA$-bimodule with the left and right
actions given by $\mu$.

By neglect of structure an $\sA$-$\sB$-bimodule $\sX$ is also a
left $\sA$-module or a right $\sB$-module.  Let $L\sX$ be
the monoid $\sX$ regarded as an $\sA$-$U_B$-bimodule with left action given
by $\kappa$ and right action the unit isomorphism.  Let $R\sX$ be the
monoid $\sX$ regarded as a $U_A$-$\sB$-bimodule with left action the unit
isomorphism and right action $\kappa$.

\begin{definition} Let $\sA$ and $\sB$ be monoids and $\sX$ and $\sY$ be
$\sA$-$\sB$-bimodules. A \emph{map of bimodules}\idx{map of
bimodules} is a 2-cell $f\colon \sX\rightarrow \sY$ such that
\[\xymatrix{{\sA\boxtimes \sX}\ar[r]^-\kappa\ar[d]^{\id\boxtimes
f}&{\sX} \ar[d]^f\\{\sA\boxtimes \sY}\ar[r]_-{\kappa'} &{\sY}}\]
and the corresponding diagram for $\sB$ commute in $\sW(B,A)$.
\end{definition}

\begin{definition} Let $\sA$, $\sB$, and $\sC$ be monoids, $\sX$ be an
$\sA$-$\sB$-bimodule, and $\sY$ be a $\sB$-$\sC$-bimodule.  Then
$\sX\odot \sY$\nidx{$\odot$} is the $\sA$-$\sC$-bimodule defined
by the coequalizer
\[\xymatrix{{\sX\boxtimes \sB\boxtimes \sY}\ar@<.5ex>[r]\ar@<-.5ex>[r]&
{\sX\boxtimes \sY}\ar[r]&{\sX\odot \sY}}\] in $\sW(C,A)$.  The
left $\sA$ action is induced by the left action of $\sA$ on $\sX$.
The right $\sC$ action is induced by the right action of $\sC$ on
$\sY$.\end{definition}

If $\sB$ is the monoid $U_{B}$ then $\sX\odot \sY=\sX\boxtimes \sY$.

This defines a bicategory $\sM_\sW$ \nidx{mw@$\protect\sM_{\protect\sW}$} with
0-cells monoids in $\sW$, 1-cells bimodules in $\sW$, and 2-cells
maps of bimodules.  The unit for a monoid $\sA$ is given by
regarding that monoid as a bimodule over itself.  The bicategory
composition is $\odot$.

In the bicategory $\sM_\sW$ we have some very simple examples of
dualizable objects given by the monoids of $\sW$.

\begin{prop}\mylabel{monoiddual} For a monoid $\sA$, $(R(U_\sA),L(U_\sA))$
is a dual pair.
\end{prop}

\begin{proof} 
Let $N(U_\sA)$ be $\sA$
regarded as a $U_A$-$U_A$-bimodule.  Then the unit map
\[\iota\colon  U_A\rightarrow N(U_\sA)\] is a map of $U_A$-$U_A$-bimodules.  The
associativity diagram for the composition $\mu$ implies that
\[\mu\colon L(U_\sA)\odot R(U_\sA)=L(U_\sA)\boxtimes R(U_\sA)\rightarrow U_\sA\]
is a map of $\sA$-$\sA$-bimodules. The unit conditions imply 
the composites \[\xymatrix{{R(U_\sA) \cong U_A\odot
R(U_\sA)}\ar[r]^-{\iota\odot\id} & {N(U_\sA) \odot
R(U_\sA)}\ar[r]^-\mu&{R(U_\sA)}}\]
\[\xymatrix{{L(U_\sA)\cong L(U_\sA)\odot U_A}\ar[r]^-{\id\odot \iota}&
{L(U_\sA)\odot N(U_\sA)}\ar[r]^-\mu&{L(U_\sA)}}\] are identity maps.

The coevaluation map \[\eta\colon U_A\rightarrow N(U_\sA)\cong
R(U_\sA)\odot L(U_\sA)\] is the unit.  The evaluation map
\[\epsilon\colon L(U_\sA)\odot R(U_\sA)\cong L(U_\sA)\boxtimes R(U_\sA)\rightarrow
U_\sA\]
is the composition.
The diagrams demonstrating that this is a dual pair commute by the
unit and associativity conditions of the monoid.
\[\xymatrix{{R(U_\sA)}\ar[d]_{\mathrm{id}}& {U_A\odot R(U_\sA)}\ar[r]^-{\iota
\odot \id}\ar[l]_-\cong&{N(U_\sA)\odot R(U_\sA)}\ar[d]^{\cong\odot \id}
\ar[dll]_-\mu\\
{R(U_\sA)}&{R(U_\sA)\odot U_\sA}\ar[l]^-\mu&
{R(U_\sA)\odot L(U_\sA)\odot R(U_\sA)}\ar[l]^-{\id\odot \mu}}\]

\[\xymatrix{{L(U_\sA)}\ar[d]_{\mathrm{id}}& {L(U_\sA)\odot U_A}\ar[r]^-{\id\odot
\iota}\ar[l]_-\cong&{L(U_\sA)\odot N(U_\sA)}\ar[d]^{\id\odot \cong}
\ar[dll]_-\mu\\
{L(U_\sA)}&{U_\sA\odot U_L(\sA)}\ar[l]^-\mu&
{L(U_\sA)\odot R(U_\sA)\odot L(U_\sA)}\ar[l]^-{\mu\odot \id}}\]
\end{proof}

\begin{lemma}\mylabel{bicatshlift} Suppose $\sW$ is a 
bicategory with shadows $[[-]]$ and  chosen
0-cell $I$.  If $\sW$ is a bicategory with
shadows and $\sW(I,I)$ has all coequalizers
then  $\sM_\sW$ is a bicategory with shadows.\end{lemma}

\begin{proof}
Then the chosen 0-cell of $\sM_\sW$ is the monoid
$U_I$. The shadows of $\sM_\sW$, $\sh{-}$, are defined by the
coequalizers \[\xymatrix{ [[\sA\boxtimes
\sZ]]\ar@<.5ex>[r]\ar@<-.5ex>[r]&[[\sZ]]\ar[r]&\sh{\sZ}}\] The
composite
\[ [[\sX\boxtimes \sY]]\rightarrow [[\sY\boxtimes \sX]]
\rightarrow [[\sY\odot\sX]]\rightarrow \sh{\sY\odot\sX}\] induces
a map \[[[\sX\odot\sY]]\rightarrow \sh{\sY\odot\sX}.\] This map
induces the isomorphisms $\theta_{\sB,\sA}$.
\end{proof}

Note that the shadows constructed in Section \ref{catex1} are of
this form.   A symmetric monoidal category $\sV$ is a bicategory
with a single 0-cell and the identity functor defines shadows for
this bicategory.  The shadows constructed in Chapter
\ref{classfpsec} for the bicategory $\sM_{\Ex}$ are also of this
form.  The chosen 0-cell in $\Ex$ is the 1-point space and the
functors $[[-]]$ are $r_!\triangle^*$.

\begin{lemma}\mylabel{bicatfunctor}
Let $\sW$ and $\sU$ be bicategories.
A lax functor $F\colon \sW\rightarrow \sU$ of bicategories defines a lax
functor $\sM_F\colon \sM_{\sW}\rightarrow \sM_{\sU}$.

If the shadows of $\sM_{\sW}$ are defined as in \myref{bicatshlift}
and $F$ is compatible with shadows then $\sM_F$ is compatible with
shadows.

\end{lemma}

\begin{proof}The functor  $\sM_F$ takes a monoid
$\sA$ to the monoid $F\sA$ with unit
\[U_{FA}\rightarrow F(U_A)\rightarrow F\sA\] and composition
\[F\sA\boxtimes F\sA\rightarrow F(\sA\boxtimes \sA)\rightarrow F\sA.\]
An $\sA$-$\sB$-bimodule $\sX$ is taken to the $F\sA$-$F\sB$-bimodule $F\sX$
with actions
\[F\sA\boxtimes F\sX\rightarrow F(\sA\boxtimes \sX)\rightarrow F\sX\]
\[F\sX\boxtimes F\sB\rightarrow F(\sX\boxtimes \sB)\rightarrow F\sX\]
The map $F\sX\boxtimes F\sY\rightarrow F(\sX\boxtimes \sY)$ induces a
map \[F\sX\odot F\sY\rightarrow F(\sX\odot\sY).\]

If shadows in $\sM_\sW$ are defined from the shadows of $\sW$ and
$F$ is compatible with shadows, $\sM_F$ is compatible with
shadows.  The maps $\sh{F\sX}\rightarrow F\sh{\sX}$ are induced by
the corresponding maps \[[[F\sX]]\rightarrow F[[\sX]].\]
\end{proof}

\backmatter

\Printindex{idx}{Index}
\Printindex{nidx}{Index of Notation}

\end{document}